\documentclass[12pt]{article}
\usepackage{amsthm}
\usepackage[bbgreekl]{mathbbol} %\mathbb{abcABC\nabla\bbalpha\bbbeta\Delta  }
\usepackage{bm}
\usepackage{bbm}
\usepackage{ulem}  % for "sout"
\usepackage{comment}
\usepackage{graphicx}
\usepackage{latexsym}
\usepackage{amssymb}
\usepackage{amsmath}
\usepackage{enumerate}
\usepackage{color}
\usepackage{layout}
\usepackage{dsfont}
\usepackage{eufrak}
\newtheorem{prop}{Proposition}
\newtheorem{lemma}{Lemma}

\newtheorem{corollary}{Corollary}
\newtheorem{theorem}{Theorem}
\newtheorem{remark}{Remark}

\usepackage{authblk}

\newcommand{\E}{\mathbb{E}}

\def\tti{{\tt i}}

\def\sfP{{\sf P}}
\def\sfR{{\sf R}}
\def\sfk{{\sf k}}

\def\sfq{{\sf q}}
\def\blambda{{\boldsymbol{\lambda}}}

\newcommand{\ba}{\begin{array}}
\newcommand{\ea}{\end{array}}
\newcommand{\be}{\begin{equation}}
\newcommand{\ee}{\end{equation}}
\newcommand{\bea}{\begin{eqnarray}}
\newcommand{\eea}{\end{eqnarray}}
\newcommand{\beaa}{\begin{eqnarray*}}
\newcommand{\eeaa}{\end{eqnarray*}}

\font\tenmath=msbm10 \font\sevenmath=msbm7 \font\fivemath=msbm5
\newfam\mathfam \textfont\mathfam=\tenmath
\scriptfont\mathfam=\sevenmath \scriptscriptfont\mathfam=\fivemath

\def \={{\buildrel {\rm (law)} \over =}}

\def\D{\Delta}

\def\cB{{\cal B}}

\def\qed{ \hfill \vrule width.25cm height.25cm depth0cm\smallskip}

\newcommand{\basa}{\begin{assumption}}
\newcommand{\easa}{\end{assumption}}

\newcommand{\bas}{\begin{assum}}
\newcommand{\eas}{\end{assum}}

\def\limsup{\mathop{\overline{\rm lim}}}

\def\wt{\widetilde}

\def\wt{\widetilde}

\def\bE{{\bf E}}

\def\bea{\begin{eqnarray}}
\def\eea{\end{eqnarray}}
\def\beas{\begin{eqnarray*}}
\def\eeas{\end{eqnarray*}}
\def\bd{\begin{description}}
\def\ed{\end{description}}
\def\im{\item}
\newif\ifcol
\colfalse % Comment out this line for color characters
\ifcol
\newcommand{\cred}{\color[rgb]{0.8,0,0}}

\newcommand{\colorg}{\color[rgb]{0,0.5,0}}

\else
\newcommand{\cred}{\color{black}}%{\color[rgb]{0.8,0,0}}
\newcommand{\colorg}{\color{black}}% {\color[rgb]{0,0.5,0}}

\fi
\excludecomment{en-text}
\includecomment{jp-text}
\includecomment{comment}
\def\bd{\begin{description}}
\def\ed{\end{description}}

\def\D2{\bbD_{2,\infty-}}

\def\ba{\bar{A}}

\def\D{{\bf D}}
\def\E{{\bf E}}

\def\cale{{\cal E}}

\def\calh{{\cal H}}

\def\calp{{\cal P}}

\def\cals{{\cal S}}

\def\ds{\displaystyle}
\def\yeq{\>=\>}
\def\yleq{\>\leq\>}

\def\sfk{{\sf k}}

\def\simleq{\ \raisebox{-.7ex}{$\stackrel{{\textstyle <}}{\sim}$}\ }

\def\ep{\epsilon}
\def\half{\frac{1}{2}}

\def\halflineskip{\vspace*{3mm}}
\def\nn{\nonumber}
\def\be{\begin{equation}}
\def\ee{\end{equation}}
\def\bea{\begin{eqnarray}}
\def\eea{\end{eqnarray}}
\def\beas{\begin{eqnarray*}}
\def\eeas{\end{eqnarray*}}
\def\bi{\begin{itemize}}
\def\ei{\end{itemize}}
\def\im{\item}
\def\bd{\begin{description}}
\def\ed{\end{description}}
\def\r{\right}
\newcommand{\bbD}{{\mathbb D}}

\newcommand{\bbI}{{\mathbb I}}
\newcommand{\bbJ}{{\mathbb J}}

\newcommand{\bbN}{{\mathbb N}}

\newcommand{\bbR}{{\mathbb R}}
\newcommand{\bbS}{{\mathbb S}}

\newcommand{\bbZ}{{\mathbb Z}}

\newcommand{\ignore}[1]{}
\begin{document}

\renewcommand{\thefootnote}{\fnsymbol{footnote}}

\renewcommand{\thefootnote}{\fnsymbol{footnote}}

\title{High order asymptotic expansion for Wiener functionals
%{High order asymptotic expansion for random vectors in Wiener chaos}
\footnote{
This work was in part supported by 
Japan Science and Technology Agency CREST JPMJCR14D7; 
Japan Society for the Promotion of Science Grants-in-Aid for Scientific Research 
No. 17H01702 (Scientific Research);  
%No. 24340015 (Scientific Research), 
%Nos. 24650148 and 
%No. 26540011 (Challenging Exploratory Research); \koko
%the Global COE program ``The Research and Training Center for New Development in Mathematics'' of the Graduate School of Mathematical Sciences, University of Tokyo; 
%NS Solutions Corporation; 
and by a Cooperative Research Program of the Institute of Statistical Mathematics. %%
%The main parts in this paper were presented at 
%International conference ``Statistique Asymptotique des Processus Stochastiques VII'', Universit\'e du Maine, Le Mans, March 16-19, 2009, 
%MSJ Spring Meeting 2010, March 24-27, 2010, Keio University, Mathematical Society of Japan, 
%and 
%International conference ``DYNSTOCH Meeting 2010'', Angers, June 16-19, 2010. 
%The author thanks to the organizers of the meetings for opportunities of the talks. 
}
}
%%%% Authors %%%%%%%%%%%%%%%%%%%%%%%
%%%%%%%%%%%%%%%%%%%%%%%%%%%%%%%
\author[1,2]{Ciprian A. Tudor}
\author[3,4,5]{Nakahiro Yoshida}
\affil[1]{Universit\'e de Lille 1
\footnote{Universit\'e de Lille 1: 59655 Villeneuve d'Ascq, France}
        }
%\affil[2]{Samos, Universit\'e de Panth\'eon-Sorbonne Paris 1
%\footnote{Samos, Universit\'e de Panth\'eon-Sorbonne Paris 1: 90, rue de Tolbiac, 75013 Paris, %France. e-mail: Ciprian.Tudor@univ-paris1.fr}
 %       }
\affil[3]{Graduate School of Mathematical Sciences, University of Tokyo
\footnote{Graduate School of Mathematical Sciences, University of Tokyo: 3-8-1 Komaba, Meguro-ku, Tokyo 153-8914, Japan. e-mail: nakahiro@ms.u-tokyo.ac.jp}
        }
\affil[4]{CREST, Japan Science and Technology Agency}
\affil[5]{The Institute of Statistical Mathematics
%%\footnote{}
        }
%%%% Date %%%%%%%%%%%%%%%%%%%%%%%%
%%%%%%%%%%%%%%%%%%%%%%%%%%%%%%%
%\date{September 14, 2010, \\
%%Revised February 25, 2012
%}
%%%%%%%%%%%%%%%%%%%%%%%%%%%%%%%
\maketitle

\begin{abstract}
{\colorg By combining the Malliavin calculus with Fourier techniques, we develop a high-order asymptotic expansion theory for a sequence of vector-valued random variables. Our asymptotic expansion formulas give the development of the characteristic functional and of the  local density of  the random vectors up to an arbitrary order. We analyzed in details an example related to  the wave equation  with space-time white noise which also provides  interesting facts on the correlation structure of the solution to this equation.  }
\end{abstract}

\vskip0.3cm

{\bf 2010 AMS Classification Numbers: } 62M09, 60F05, 62H12

\vskip0.3cm


%{\cred 
%\begin{screen}
%To do. 
%\bi
%\im Check Lemma \ref{not} (a). 
%\im Check the construction of $\calh$ in Section \ref{201904220153}. 
%\im Check Multidimensional case of Section 5.7-5.8. 
%\im Complete Introduction. 
%\im Avoid duplication of Section \ref{201904220142}. 
%\im The last two equations of Section \ref{201904220306} have been left unchecked. 
%\im Section \ref{201904220344} should be checked. 
%\ei
%\end{screen}
%}

\section{Introduction}
{
The asymptotic expansion of probability distributions for random variables and vectors represents a fundamental topic in probability theory and mathematical statistics. This theory has been widely applied to several fields, including the { efficiency} of estimators, hypothesis testing, 
{ information criterion for} model selection, 
{ prediction theory, bootstrap methods and resampling plans, and information geometry}. There exists now huge literature on asymptotic expansion. We refer, among many others, to {\cite{BR} and \cite{bhattacharya1990asymptotic}} for the case of sequences of i.i.d. random variables 
{ and applications}, to \cite{GH} for sequences of weakly dependent  variables, to \cite{My}, \cite{Yos2}, \cite{Yos3}, 
\cite{KusuokaYoshida2000}, \cite{Yoshida2001b}, \cite{SakamotoYoshida2004}, 
\cite{Yos1}, \cite{PY} and 
\cite{podolskij2017edgeworth} for asymptotic expansion of martingales and of classical diffusions, and to \cite{NY} or \cite{TY} for general sequences of random variables. The reader may consult the monographs \cite{BR}, \cite{Hall} and \cite{La} for complete expositions on these topics. 

Generally speaking, the asymptotic expansion theory aims at finding the expansion of the density functions for a sequence of random variables that converges in law to a target distributions (usually, the Gaussian distribution, but other target distributions, such as mixed normal, are possible). The theory usually  provides the leading term and  second order term  in the asymptotic development of the density function.  For estimators and test statistics, while the leading term is used  for confidence limits and testing, the {\colorg higher order terms provide} a more accurate inference. Some higher order-type asymptotic expansions for particular sequences of diffusion-type can be found in 
{
\cite{Yoshida1992}, \cite{Yoshida1993} with applications to statistics, and in
\cite{Yoshida1992a}, \cite{KunitomoTakahashi2001}, 
\cite{UchidaYoshida2004b}, 
\cite{Ta}, \cite{TaTa}, \cite{TaTaTo} 
}with applications to finance.

Our purpose is to provide a general method to obtain the asymptotic expansions of density functions up to an arbitrary order, i.e. to find the further terms that appear in the asymptotic behavior of the density. Our approach combines the so-called Fourier approach and the recent Stein-Malliavin theory (see \cite{NPbook}) and applies to general sequences of random vectors. Our main finding is that the asymptotic expansion up to any order of the family of densities  of a vector-valued random sequence $(F_{N})_{N\geq 1}$  is completely characterized by the expectation of the so-called   {\it Gamma factors} associated to the sequence $F_{N}$, or equivalently, by the joint cumulants of the components of this sequence. These Gamma factors (defined in Section \ref{sec21}) are defined in terms of the Malliavin operators of $F_{N}$. Consequently, the knowledge of the Taylor expansions of the cumulants, together with some regularity in the Malliavin sense of $F_{N}$,  gives the higher order asymptotic expansion of the  density. We also mention that, in contrast with the classical assumptions in the martingale case (see e.g. \cite{Yos3}, \cite{Yos1}) or in the Malliavin calculus case (see \cite{TY}), the joint convergence in distribution of $F_{N}$ together with its "bracket" (which is the usual martingale bracket when we deal with sequences of martingales and it is defined in terms of the Malliavin derivative in the non-martingale case) is not assumed in our work.  

As mentioned above, our strategy is based on the Stein-Malliavin calculus combined with the so-called Fourier {\colorg approach}. We start by analyzing the behavior of the (truncated) characteristic function of the sequence $(F_{N}) _{N\geq 1}$ via an interpolation method and the Malliavin-type integration by parts. We notice the appearance of the Gamma factors in the principal part of asymptotic expansion of the characteristic function. The Fourier inversion, together with some regularity of the distribution expressed in term of the Malliavin calculus, allows to develop asymptotically the sequence of 
{(local)} 
densities of $ F_{N}$.   Some regular ordering of the cumulants is assumed and this is checked in examples. Usually, the second order term in the asymptotic expansion 
{comes} %provienes 
from leading term in the expansion of third cumulant, while the third order term is due to the second and the fourth cumulant. A general formula is obtained.

As an example, we analyze the behavior of the spatial quadratic variation for the solution to the wave equation driven by a space-time white noise. We treat both the one dimensional case (i.e we  fix the time $t$ and we study the quadratic variation in space for the solution), as well as a two-dimensional case (i.e. we consider {\colorg a } the two-dimensional random sequence whose components are the spatial quadratic variations of the solution at two different times). In both cases,  based on a sharp analysis of the correlation structure of the solution, we are able to find the asymptotic expansion up to at least the third order term.  {\colorg Let us emphasize that, besides being a toy example to apply our asymptotic expansion theory,  this last part of our work   shows some interesting facts related to the solution to the wave equation driven by  a space-time white noise. We obtain the precise correlation of the increments of the solution, at fixed time and when the time is moving and it appears that the dependence structure of these increments depends in a non-trivial way on the spatial and temporal lags.  }

We organized our paper as follows. {\colorg Section 2 presents, after the definition and some basic properties of the Gamma-factors, the asymptotic expansion up to an arbitrary order of the (truncated) characteristic function of a sequence of vector-valued random variables $(F_{N}) _{N\geq 1}$.} This expansion depends on the Gamma-factors (or equivalently, the cumulants) of the vector $F_{N}$. Based on this expansion, we obtain in Section 3, by inverting in Fourier sense the principal part of the characteristic function, the  approximate density for our sequence. This will approximate the local (truncated) density of $F_{N}$. The asymptotic expansion is further explicited in Section 4 where we show that if the cumulants of $F_{N}$ admit a specific Taylor expansion, then a more precise expansion of the local density can be derived. In Section 5 we treat in details a concrete example related to the solution to the wave equation driven by an additive space-time white noise. The last {\colorg section} is the Appendix which contains the basic tools of the Malliavin calculus. 
}

\section{Expansion of the characteristic functional}%{Assumptions and notations}
{In this section,  we analyze  the asymptotic behavior of a general {\colorg sequence} of random vectors, by using an interpolation method and the Malliavin integration by parts.  We will distinguish a principal part of the characteristic function, written in terms of the Gamma factors and a neglijible part. These two parts are then estimated separately.  }

\subsection{The Gamma factors}\label{sec21}
{\colorg  Let $(W(h), h\in H)$ be an isonormal process on a standard  probability space $(\Omega, \mathcal{F}, P)$. }For the definition of the Malliavin operators {\colorg with respect to $W$, } see Section \ref{app}.
{
The pseudo-inverse $L^{-1}$ of $L$ is defined by 
$L^{-1}F=\sum_{q=1}^\infty q^{-1}J_qF$ for $F=\sum_{q=0}^\infty J_qF\in L^2(\Omega)$, 
where $J_q$ is the orthogonal projection to the $q$-th chaos.
For $p\in\bbN$ and $F\in\bbD_{1,p}$,} 
we define the  {\it Gamma- factors } $\Gamma ^{(p)}(F) $ is a reccursive way, see e.g. \cite{NPbook}.
\begin{eqnarray*}
&& \Gamma^{(1)}(F)= F\\
&& \Gamma^{(2)} (F)= \langle DF, D(-L) ^ {-1} F \rangle _{H},\\
&&....\\
&&\Gamma ^{(p)} (F)= \langle DF, D(-L)^ {-1} \Gamma^{(p-1)} (F) \rangle _{H}.
\end{eqnarray*}
These variables are well defined by Lemma \ref{20181102-1} below, {\colorg if $F$ is regular enough in the sense of the Malliavin calculus. }

We have the following formula that links the Gamma-factors and the cumulants: for every $m\geq 1$
\begin{equation}
\label{kg}
k_{m} (F)= (m-1) ! \mathbf{E}\left[ \Gamma ^{(m)} (F)\right].
\end{equation}
Recall that the $m$th cumulant of a  random variable $F\in L^ {m}(\Omega)$ is given by 
\begin{equation*}
k_{m}(F)= (-i) ^ {m} \frac{\partial}{\partial t ^ {m}}\ln\mathbf{E}\left[ e ^ {itF}\right]\big| _{t=0}.
\end{equation*}

We also introduce {\it the multidimensional Gamma factors} of a random vector $F=(F^{(1)},..,F^{(d)}){ \in\bbD_{1,p}(\bbR^d)}$  are defined in the following way. For $i=1,..,d$,
$$ \Gamma ^{(1)}_i (F)= F  ^{(i)} $$
and for $i_{1}, i_{2} =1,.., d$,
$$\Gamma ^{(2)} _{i_{!}, i_{2} } (F) =\langle DF ^{(i_{2})}, D(-L) ^{-1} F^{(i_{1})}\rangle_H$$
while for $i_{1},.., i_{p}=1,..,d$
\begin{equation}
\label{gf}
\Gamma ^{(p)} _{i_{1},.., i_{p}}(F) = \langle DF^{(i_{p})}, D(-L) ^{-1} \Gamma ^{(p-1) }_{i_{1},.., i_{p-1}}(F)\rangle_H .
\end{equation}

The muldimensional Gamma factors $\Gamma ^{(p)} _{i_{1},.., i_{p}}(F) $ are also related to the joint cumulants of the random vector $F$.  Recall that if  $m =(m_{1},.., m_{d}) \in \mathbb{N} ^{d} $, then the $m$th cumulant of the random vector  $F=(F^{(1)},..,F^{(d)}) $ is 

$$k_{m}(F) = k_{(m_{1},.., m_{d}) } (F^{(1)},.., F^{(d)}) =(-{\tt i} ) ^{\vert m\vert } \frac{\partial ^{\vert m\vert }}{\partial t ^{m}} \log \mathbf{E}\left[ e ^{{\tt i} \langle t,F\rangle}\right] | _{t=0}$$
where $\vert m\vert = m_{1}+...+m_{d}$.  See \cite{NPbook} for the precise link between the multidimensional Gamma factors and the cumulants.

{
%\noindent{\bf Question.} 
%Does $F\in\bbD_{\ell,\infty-}$ imply $\Gamma^{(p+1)}_{i_1,...,i_{p+1}}(F)\in\bbD_{\ell,1+}$? Yes, 
Let $\bbZ_+=\{0,1,2,...\}$. 
\begin{lemma}\label{20181102-1}
\bd
\im[(a)] 
Let $\ell\in\bbZ_+$ %$\ell\geq0$ 
and $r>1$. Then 
$D(-L)^{-1}F\in\bbD_{\ell+1,r}({\colorg H})$ if $F\in\bbD_{\ell,r}$, and there exists a constant $C_{\ell,r}$ such that 
\beas 
\big\|D(-L)^{-1}F\big\|_{\ell+1,r}
&\leq& 
C_{\ell,r}\big\|F\big\|_{\ell,r}
\eeas
for all $F\in\bbD_{\ell,r}$.%, where $\wt{F}=F-E[F]$. 
\im[(b)] Let $\ell\in\bbZ_+$, $r>1$ and $p\in\{2,3,...\}$. Then 
$\Gamma^{(p)}_{i_1,...,i_p}(F)\in\bbD_{\ell,r}$ 
if $F=(F^{(1)},...,F^{(d)})\in\bbD_{\ell+1,pr}(\bbR^d)$, and there exists a constant $C_{\ell,r,p}$ such that 
\bea\label{201812291317} 
\big\|\Gamma^{(p)}_{i_1,...,i_p}(F)\big\|_{\ell,r}
&\leq& 
C_{\ell,r,p}\big\|F\big\|_{\ell+1,pr}^p
\eea
for all $F\in\bbD_{\ell+1,pr}$ and $i_1,...,i_p\in\{1,...d\}$. 
In particular, 
$\Gamma^{(p)}_{i_1,...,i_p}(F)\in\bbD_{\ell,\infty-}$ if $F\in\bbD_{\ell+1,\infty-}=\cap_{r>1}\bbD_{\ell+1,r}$. 
\ed
\end{lemma}
\proof 
\begin{en-text}
Let $r\in[1,\infty)$ and let $q=r/(r-1)$. 
For $F, G\in\cals$ (the set of smooth functionals), 
\beas
E\big[\langle DG,D(-L)^{-1}F\rangle_\calh\big]
&=&
E\big[G\>\delta D(-L)^{-1}F\big]\yeq E[GF]
\eeas
and in particular
\beas
\big|E\big[\langle DG,D(-L)^{-1}F\rangle_\calh\big]\big|
%&=&
%\big|E\big[G\>\delta D(-L)^{-1}F\big]\big|
&\leq&
\|F\|_r\|G\|_{r'}
\eeas
Since $\{DG;\>G\in\cals\}$ is dense in $L^{r'}(\calh)$, we have 
Suppose that $F\in\bbD_{\ell,r}$ for some $\ell\geq0$. 
because for $\Phi\in\calp$, 
\end{en-text}
Let $r>1$. 
Since 
$
D(-L)^{-1}F= (I-L)^{-1}DF
$
for $F\in\calp$, the set of polynomial functionals, 
we have 
\beas 
\big\|D(-L)^{-1}F\big\|_{\ell+1,r}
&\simleq& 
\big\|(I-L)^{(\ell+1)/2}D(-L)^{-1}F\big\|_r
\yeq
\big\|(I-L)^{(\ell-1)/2}DF\big\|_r
\\&\simleq&
\big\|DF\big\|_{\ell-1,r}
\>\simleq\>
\big\|F\big\|_{\ell,r}
\eeas
for all $F\in\calp$ uniformly. 
Therefore, $D(-L)^{-1}$ is extended as a continuous linear operator 
from $\bbD_{\ell,r}$ to $\bbD_{\ell+1,r}(\calh)$. 
Thus we obtained (a). 

Suppose that $F=(F^{(1)},...,F^{(d)})\in\bbD_{\ell+1,pr}(\bbR^d)$. 
\begin{en-text}
Since $\Gamma^{(2)}_{i_1,i_2}(F)=\langle DF ^{(i_{2})}, D(-L) ^{-1} F^{(i_{1})}\rangle_H$, 
we have
\beas 
\big\|\Gamma^{(2)}_{i_1,i_2}(F)\big\|_{\ell,r}
&\simleq&
\big\|F^{(i_2)}\big\|_{\ell+1,2r}\big\|D(-L)^{-1}F^{(i_1)}\big\|_{\ell,2r}
\\&\simleq&
\big\|F^{(i_2)}\big\|_{\ell+1,2r}\big\|F^{(i_1)}\big\|_{\ell-1,2r}
\\&\simleq&
\big\|F\big\|_{\ell+1,2r}^2.
\eeas
\end{en-text}
The property (b) follows from (a) by induction. 
Indeed, 
since $\Gamma^{(k+1)}_{i_1,...,i_{k+1}}(F)
=\langle DF^{(i_{k+1})}, D(-L)^{-1} \Gamma^{(k)}_{(i_1,...,i_k)}\rangle_H$, 
if (\ref{201812291317}) holds for $k$ $(\leq p-1)$ in place of $p$, then 
\beas 
\big\|\Gamma^{(k+1)}_{i_1,...,i_{k+1}}(F)\big\|_{\ell,r}
&\simleq&
\big\|F^{(i_{k+1})}\big\|_{\ell+1,(k+1)r}\big\|D(-L)^{-1}\Gamma^{(k)}_{(i_1,...,i_k)}(F)
\big\|_{\ell,k^{-1}(k+1)r}
\\&\simleq&
\big\|F^{(i_{{\cred k+1}})}\big\|_{\ell+1,(k+1)r}\big\|\Gamma^{(k)}_{(i_1,...,i_k)}(F)\big\|_{\ell-1,k^{-1}(k+1)r}
\\&\simleq&
\big\|F^{(i_{{\cred k+1}})}\big\|_{\ell+1,(k+1)r}\big\|F\big\|_{\ell,(k+1)r}^k
\\&\simleq&
\big\|F\big\|_{\ell+1,(k+1)r}^{k+1}.
\eeas
{\cred Thus, (\ref{201812291317}) holds for $p=k+1$. 
The inequality (\ref{201812291317}) is trivial when $p=1$, which completes the proof. 
\qed
}

\subsection{Interpolation}%{Perturbation}
{\colorg Let $d\geq 1$. Consider a sequence of centered  random variables $(F_{N}) _{N\geq 1}$ in $\mathbb{R} ^{d}$ of the form
$$F_{N}= \left( F_{N} ^{(1)}, \ldots, F_{N} ^{(d)}\right).$$ Let $C=(C_{i,j})_{i,j=1}^d$ be a deterministic $d\times d$ positive definite symmetric matrix.  For every $N\geq 1$, we introduce {\cred a} %the 
truncation functional $\Psi_{N}$ which is a smooth random variable. A more precise form of this functional will be chosen later in Section 2.} We consider the following condition for the truncation functional $\Psi_N$: if $p$ is a positive integer
\bd
\im[[$\Psi$\!\!]] 
For each $N\in\bbN$, 
{$\Psi_N:\Omega\to[0,1]$ and 
$\Psi_N\in\bbD_{1,p+1}$.}
% and there exists a constant $c_N$ such that $|F_N|\leq c_N$ whenever { $|\Psi_N|+|D\Psi_N|_H>0$.}
%$\sum_{i=0}^1|D^i\Psi_N|_{{ H^{\otimes i}}}>0$. 
\ed

Let {\colorg  us define the interpolation functional }
\bea\label{en}
e(\theta, \boldsymbol{\lambda}, F_{N})
&=& 
\exp\bigg(
{\tt i} \theta \langle \boldsymbol{\lambda} , F_{N} \rangle 
- \half(1-\theta ^{2}) 
\boldsymbol{\lambda}^{T} C\boldsymbol{\lambda} \bigg)
\eea
for every $\theta \in [0,1]$ and  $\boldsymbol{\lambda}\in \mathbb{R} ^ {d} $. 
Let $ \boldsymbol{\lambda}= (\lambda _{1},..., \lambda _{d}) \in \mathbb{R } ^{d}$ and $\theta \in [0,1]$ and let us consider the truncated interpolation
\bea\label{interpol}
\varphi _N^{\Psi} (\theta, \boldsymbol{\lambda} ) 
&=& 
\mathbf{E} \big[ \Psi _{N} e(\theta, \boldsymbol{\lambda}, F_{N})\big].
\eea
Notice that $\varphi^\Psi_{N} (1, \boldsymbol{\lambda}) = \mathbf{E}\big[\Psi_{N} e ^{{ \tti}\langle \boldsymbol{\lambda}, F_{N} \rangle }\big]$ represents  the "truncated" characteristic function of $F_{N}$, 
while $\varphi^\Psi_{N} (0, \boldsymbol{\lambda}) =\mathbf{E}[\Psi_{N}]e ^{-\boldsymbol{\lambda} ^{T}C\boldsymbol{\lambda}/2}$ is  the "truncated"characteristic function of the limit in law of $F_{N}$. 

{\colorg
The first step is to get the expansion  of the derivative with respect to the variable $\theta$  of the characteristic functional.}  
{

%Let $\Lambda(c)=\{x+{\tt i}y;\>x\in\bbR,\>y\in(-c,c)\}^d$ for $c>0$. 
}

\begin{lemma}\label{20181031-1}
{ Suppose that $F_N\in{\bbD_{1,p+1}}(\bbR^d)$  {\colorg for any $N\geq 1$ } and that $[\Psi]$ holds. 
Then the functional $\varphi^\Psi_N(\theta,\blambda)$ is well defined for 
$(\theta,\blambda)\in[0,1]\times{ \bbR^d}$, %\Lambda(c)$, 
and it holds that 
}
\begin{eqnarray}\label{20190414-1}
 \frac{\partial}{\partial \theta} \varphi _N ^{\Psi} (\theta, \boldsymbol{\lambda} ) 
 &=& 
 {\tt i} ({\tt i} \theta)^{p}  \sum_{i_{1}, ..., i_{p+1} =1} ^{d} \lambda _{i_{1}}....\lambda _{i_{p+1}}
 \mathbf{E}\bigg[\Psi_{N} e(\theta, \boldsymbol{\lambda}, F_{N}) \Gamma^{(p+1)}_{i_{1}, ...,  i_{p+1}} (F_{N}) \bigg]
\nn\\&&
+ {\tt i} ({\tt i} \theta)^{p-1}  \sum_{i_{1}, ..., i_{p} =1} ^{d} \lambda _{i_{1}}....\lambda _{i_{p}}
\mathbf{E}\big[\Psi_{N} e(\theta, \boldsymbol{\lambda}, F_{N})\big]
\mathbf{E}\big[\Gamma^{(p)}_{i_{1}, ...,  i_{p}} (F_{N})\big]
\nn\\&&
+  {\tt i} ({\tt i} \theta)^{p-2}  \sum_{i_{1}, ..., i_{p-1} =1} ^{d} \lambda _{i_{1}}....\lambda _{i_{p-1}}
\mathbf{E}\big[\Psi_{N} e(\theta, \boldsymbol{\lambda}, F_{N})\big] 
\mathbf{E}\big[\Gamma^{(p-1)}_{i_{1}, ...,  i_{p-1}} (F_{N})\big]
\nn\\&&
+\cdots
\nn\\&&
+ {\tt i} ({\tt i} \theta)^{2} \sum_{i_{1}, i_{2}, i_{3} =1} ^{d} \lambda _{i_{1}}\lambda _{i_{2}}\lambda _{i_{3}}
\mathbf{E}\big[\Psi_{N} e(\theta, \boldsymbol{\lambda}, F_{N})\big]
\mathbf{E}\big[\Gamma^{(3)}_{i_{1}, i_{2}, i_{3}} (F_{N})\big]
\nn\\&&
+{\tt i} ({\tt i} \theta)\sum_{i_{1}, i_{2}=1} ^{d}  \lambda _{i_{1}}{ \lambda _{ i_{2}}}
\mathbf{E}\big[\Psi_{N} e(\theta, \boldsymbol{\lambda}, F_{N})\big] 
\big( \mathbf{E}\big[\Gamma^{(2)}_{i_{1}, i_{2}} (F_{N})\big] - C_{i_{1},i_{2}}\big) 
\\&&
+ \sum_{j=1}^pR_{j,N}(\theta,\blambda)%+ R_{2,N} +\cdots+ R_{{ p,N}}
\nn\end{eqnarray}
{ for $(\theta,\blambda)\in[0,1]\times{ \bbR^d}$, {\colorg $\blambda=(\lambda_{1},\ldots, \lambda _{d})$},
where}
\bea\label{rnj}
R_{{ j,N}}(\theta,\blambda)
&=&
{\tt i } ({\tt i} \theta ) ^{j-1} \sum_{i_{1}, ..., i_{{ j}} =1} ^{d} \lambda _{i_{1}}....\lambda _{i_{{ j}}} \mathbf{E}  \left[  e(\theta, \boldsymbol{\lambda}, F_{N})
\big\langle D\Psi_{N}, D(-L)^ {-1} \Gamma ^{(j)}_{i_{1},.., i_{j}} (F_{N})\big\rangle_H \right]
\nn\\&&
\eea
for $j=1,2,.., p$. 
{
In (\ref{20190414-1}), $\Gamma^{(2)}_{i_{1}, i_{2}} (F_{N})$ can be replaced by 
the symmetrized version $\Gamma^{(2sym)}_{i_{1}, i_{2}} (F_{N})=2^{-1}\big(\Gamma^{(2)}_{i_{1}, i_{2}} (F_{N})
+\Gamma^{(2)}_{i_{2}, i_{1}} (F_{N})\big)$.}
\end{lemma}
\proof
{  
The Malliavin derivative of the random variable $e(\theta, \boldsymbol{\lambda}, F_{N})$ 
given by (\ref{en}) can be calculated as follows
\begin{equation}
\label{den}De (\theta, \boldsymbol{\lambda}, F_{N})
= {\tt i} \theta \>e (\theta, \boldsymbol{\lambda}, F_{N})\sum_{k=1} ^{d} \lambda _{k} DF_{N} ^{(k)}. 
\end{equation}
}
\noindent
We differentiate $\varphi _N^{\Psi} (\theta, \boldsymbol{\lambda} )$ 
given by (\ref{interpol}) 
with respect to $\theta$, and we use the identity (\ref{aaa}) as $F_{N}^ {(j)} = \delta D(-L) ^ {-1} F_{N}^ {(j)} $ for every $j=1,..,d$ and the duality relationship and (\ref{den})  to obtain
\bea\label{1s-1}
 \frac{\partial}{\partial \theta} \varphi _N^{\Psi} (\theta, \boldsymbol{\lambda} ) 
&=&
\sum_{j=1} ^{d} {\tt i} \lambda _{j} \mathbf{E}\big[\Psi _{N}\> e(\theta,\blambda,F_N)
F_{N}^{(j)}\big]
+\sum_{j=1}^{d} \theta \lambda _{j}\lambda _{k} 
\mathbf{E} \big[ \Psi _{N}\> e(\theta,\blambda,F_N)C_{j,k}\big]
\nn\\ &=&
 -\theta \sum_{j,k=1} ^{d} \lambda _{j} \lambda _{k} 
 \mathbf{E} \bigg[ \Psi _{N}\> e(\theta,\blambda,F_N) \left(  \langle DF_{N} ^{(k)}, D(-L) ^{-1} F _{N} ^{(j)}\rangle_{H} - C_{j,k}\right) \bigg] 
\nonumber \\&&
+{\tt i} \sum_{j=1} ^{d}\lambda _{j}  
\mathbf{E}\bigg[ e(\theta,\blambda,F_N)\langle D\Psi_{N} , D(-L) ^{-1} F _{N} ^{(j)}\rangle _{H} \bigg]
\nonumber\\
%&=& {\tt i } ({\tt i} \theta) \sum_{j,k=1} ^{d} \lambda _{j} \lambda _{k} \mathbf{E} \left( \Psi _{N} e(\theta, \boldsymbol{\lambda}, F_{N}) \left(  \langle DF_{N} ^{(k)}, D(-L) ^{-1} F _{N} ^{(j)}\rangle_{H} - C_{j,k}\right) \right) \nonumber \\ && +R_{1,N}\nonumber\\
&=& 
{\tt i } ( {\tt i} \theta) \sum_{i_{1}, i_{2}=1} ^{d} \lambda _{i_{1}}{\lambda_{i_{2}}}
 \mathbf{E}\big[\Psi_{N} e(\theta, \lambda, F_{N})\big] 
\big( \Gamma^{(2)}_{i_{1}, i_{2}} (F_{N}) - C_{i_{1},i_{2}}\big) 
\nn\\&&
+ R_{1,N}(\theta,\blambda)
\eea
with 
\beas 
R_{1,N}(\theta,\blambda)
&=&
{\tt i} \sum_{j=1} ^{d}\lambda _{j}  
\mathbf{E} \bigg[ e(\theta, \boldsymbol{\lambda}, F_{N})\langle D\Psi_{N} , D(-L) ^{-1} F _{N} ^{(j)}\rangle _{H} \bigg].
\eeas
The formula (\ref{1s-1}) has been also obtain in \cite{TY} and it allows to obtain the second order terms in the asymptotic expansion of the sequence $(F_{N})_{N\geq 1}$. In order to get the higher order term in this asymptotic expansion { when $p\geq2$}, we refine the above formula (\ref{1s-1}). 
By (\ref{aaa}), we write
\begin{eqnarray}\label{201812292250}
%\langle DF_{N} ^{(i_{1})}, D(-L) ^{-1} F _{N} ^{(i_{2})}\rangle_{H} - C_{i_{1},i_{2}}&=& 
\Gamma ^{(2)} _{i_{1}, i_{2}}(F_N)- C_{i_{1},i_{2}}
\nn&=&
\Gamma ^{(2)} _{i_{1}, i_{2}}(F_N) -E \big[\Gamma ^{(2)} _{i_{1}, i_{2}}(F_N)\big]+E\big[\Gamma ^{(2)} _{i_{1}, i_{2}}(F_N)\big]-C_{i_{1}, i_{2}}
\\
&=& \delta D (-L) ^{-1} \Gamma_{i_{1}, i_{2} } (F_{N}) +E\big[\Gamma ^{(2)} _{i_{1}, i_{2}}(F_N)\big]-C_{i_{1}, i_{2}}
\end{eqnarray}
for every $i_{1}, i_{2} \in \{1,.., d\}$, %,  and with the matrix $C$ from (\ref{g2})
and we apply the duality and (\ref{den}) once again to obtain 
\bea\label{1s-2}&&
\E \big[\Psi_{N} e(\theta,\boldsymbol{ \lambda}, F_{N}) 
\delta D (-L) ^{-1} \Gamma_{i_{1}, i_{2} }^{{ (2)}}  (F_{N}) \big]
\nn\\&=&
{\tt i} \theta \E \bigg[\Psi _{N} e(\theta,\boldsymbol{ \lambda}, F_{N}) \sum_{i_{3}=1} ^{d} \lambda _{i_{3}} \langle DF_{N} ^{(i_{3})} , D (-L) ^{-1} \Gamma_{i_{1}, i_{2} }^{{ (2)}}  (F_{N}) \rangle_H\bigg] \nonumber 
\\&&
+\E \bigg[  e(\theta,\boldsymbol{ \lambda}, F_{N}) 
\langle D\Psi_{N},  D (-L) ^{-1} \Gamma_{i_{1}, i_{2} }^{{ (2)}}  (F_{N}) \rangle_H\bigg]. 
\eea
From (\ref{1s-1}), (\ref{201812292250}) and (\ref{1s-2}), we obtain 
\begin{eqnarray*}
 \frac{\partial}{\partial \theta} \varphi _N^{\Psi} (\theta, \boldsymbol{\lambda} ) 
 &=&
  {\tt i} ({\tt i} \theta)^{2} \sum_{i_{1}, i_{2}, i_{3} =1} ^{d} \lambda _{i_{1}}\lambda _{i_{2}}\lambda _{i_{3}}
  \E \big[\Psi_{N} e(\theta, \boldsymbol{\lambda}, F_{N}) \Gamma^{(3)}_{i_{1}, i_{2}, i_{3}} (F_{N}) \big]\\
&&
+ {\tt i} ({\tt i} \theta)\sum_{i_{1}, i_{2}=1} ^{d}  \lambda _{i_{1}}\lambda _{ i_{2}} 
\E \big[\Psi_{N} e(\theta, \boldsymbol{\lambda}, F_{N})\big]
\big( \E\big[\Gamma^{(2)}_{i_{1}, i_{2}} (F_{N})\big] - C_{i_{1},i_{2}}\big) \\
&&+ R_{1,N}(\theta,\blambda)+ R_{2,N} (\theta,\blambda)
\end{eqnarray*}
for every $\theta \in [0,1]$ and  $\boldsymbol{\lambda}\in{ \bbR^d}$, where 
\beas
R_{2,N}(\theta,\blambda)
&=&
 {\tt i} ({\tt i} \theta)\sum_{i_{1}, i_{2}=1} ^{d}  \lambda _{i_{1}}\lambda _{ i_{2}} 
\E \bigg[ e(\theta, \boldsymbol{\lambda}, F_{N}) \langle D\Psi_{N}, D(-L) ^{-1} \Gamma ^{(2)} _{i_{1}, i_{2}} (F_{N})\rangle_H\bigg].
\eeas
By iterating this procedure, we obtain the desired expansion of  
$\frac{\partial}{\partial \theta} \varphi _N^{\Psi} (\theta, \boldsymbol{\lambda} )$. 
\qed\halflineskip

{ %------------------------------------------\\ \noindent\underline{Note}
\subsection{Approximation to the characteristic functional and estimate of the error}
{\colorg The result in Lemma \ref{20181031-1} shows that the behavior of the characteristic functional $\varphi _N^{\Psi} (\theta, \boldsymbol{\lambda} ) $ is closely related to the behavior of the Gamma factors (or equivalently, the cumulants) of $F_{N}$. Let us define, for $(\theta,\blambda)\in[0,1]\times{ \bbR^d}$,  the following quantity, which will be called in the sequel {\it the principal part} }
\bea\label{20181117-1}
\sfP_N(\theta,\blambda)
\nn&=& 
{\tt i} ({\tt i} \theta)^{p}  \sum_{i_{1}, ..., i_{p+1} =1} ^{d} \lambda _{i_{1}}....\lambda _{i_{p+1}}\E\big[\Gamma^{(p+1)}_{i_{1}, ...,  i_{p+1}} (F_{N})\big] \nonumber \\
&&+ 
{\tt i} ({\tt i} \theta)^{p-1}  \sum_{i_{1}, ..., i_{p} =1} ^{d} \lambda _{i_{1}}....\lambda _{i_{p}}
\E\big[ \Gamma^{(p)}_{i_{1}, ...,  i_{p}} (F_{N})\big]
\nn\\&&+  {\tt i} ({\tt i} \theta)^{p-2}  \sum_{i_{1}, ..., i_{p-1} =1} ^{d} \lambda _{i_{1}}....\lambda _{i_{p-1}}\E\big[ \Gamma^{(p-1)}_{i_{1}, ...,  i_{p-1}} (F_{N}) \big]
\nonumber\\&&
+\cdots
\nonumber\\&&
+ {\tt i} ({\tt i} \theta)^{2} \sum_{i_{1}, i_{2}, i_{3} =1} ^{d} \lambda _{i_{1}}\lambda _{i_{2}}\lambda _{i_{3}} \E\big[\Gamma^{(3)}_{i_{1}, i_{2}, i_{3}} (F_{N})\big] 
\nonumber \\&&
+{\tt i} ({\tt i} \theta)\sum_{i_{1}, i_{2}=1} ^{d}  \lambda _{i_{1}}\lambda _{ i_{2}}  
\big( \E\big[\Gamma^{(2)}_{i_{1}, i_{2}} (F_{N})\big] - C_{i_{1},i_{2}}\big)
%&&+ \left(R_{1,N}+ R_{2,N} +...+ R_{p, N}\right)
\eea
and let
\beas 
\sfR_N(\theta,\blambda)
&=&
{\tt i} ({\tt i} \theta)^{p}  \sum_{i_{1}, ..., i_{p+1} =1} ^{d} \lambda _{i_{1}}....\lambda _{i_{p+1}}
\E \bigg[\Psi_{N} e(\theta, \boldsymbol{\lambda}, F_{N})\bigg( \Gamma^{(p+1)}_{i_{1}, ...,  i_{p+1}} (F_{N}) 
-\E\big[\Gamma^{(p+1)}_{i_{1}, ...,  i_{p+1}} (F_{N})\big]\bigg)\bigg]
\\&&
+ \sum_{j=1}^pR_{j,N}(\theta,\blambda). %+ R_{2,N} +...+ R_{p, N}\right).
\eeas

\begin{en-text}
For $\sfk\in\bbN$, let 
\beas 
\check{\sfP}_N(\theta,\blambda)
&=&
\sfP_N(\theta,\blambda)+
\sum_{j=1}^\sfk 
\int_0^\theta\int_0^{\theta_1}\cdots\int_0^{\theta_{\sfk-1}}
d\theta_j\cdots d\theta_2d\theta_1
\sfP_N(\theta_j,\blambda)\cdots \sfP_N(\theta_1,\blambda)\sfP_N(\theta,\blambda),
\eeas
\beas
\check{\sfR}_N(\theta,\blambda)
&=&
\bigg\{ \underline{\sfR_N(\theta,\blambda)}+
\sum_{j=1}^\sfk 
\int_0^\theta\int_0^{\theta_1}\cdots\int_0^{\theta_{j-1}}
d\theta_j\cdots d\theta_2d\theta_1
\sfR_N(\theta_j,\blambda)\sfP_N(\theta_{j-1},\blambda)\cdots \sfP_N(\theta_1,\blambda)\sfP_N(\theta,\blambda)\bigg\}
\eeas
and
\bea\label{20181030-1} 
\bar{\sfR}_N(\theta,\blambda)
&=&
\int_0^\theta\int_0^{\theta_1}\cdots\int_0^{\theta_{\sfk}}
d\theta_{\sfk+1}\cdots d\theta_2d\theta_1
\bigg\{
\sfR_N(\theta_{\sfk+1},\blambda)
+\underline{\varphi _{F_{N}} ^{\Psi} (\theta_{\sfk+1}, \boldsymbol{\lambda})}
\bigg\}  
\nn\\&&
\times
\sfP_N(\theta_{\sfk},\blambda)\cdots \sfP_N(\theta_1,\blambda)\sfP_N(\theta,\blambda)
\eea

By definition, 
$\sfP_N(\theta,\blambda)$ and $\sfR_N(\theta,\blambda)$ depend on $p$, and 
$\check{\sfP}_N(\theta,\blambda)$, $\check{\sfR}_N(\theta,\blambda)$ and $\bar{\sfP}_N(\theta,\blambda)$ 
depend on $\sfk$ as well as $p$.

\begin{lemma}\label{20181114-11}%%%%%%%%%%%%%%%%%%%%%%%
Suppose that $F_N\in\bbD_{p+1,p+1}(\bbR^d)$ and that $[\Psi]$ holds. 
\bea\label{20181113-1}
\frac{\partial}{\partial \theta} \varphi _N^{\Psi} (\theta, \boldsymbol{\lambda} )
&=&
\varphi _N^{\Psi} (0, \boldsymbol{\lambda})
\check{\sfP}_N(\theta,\blambda)
+\check{\sfR}_N(\theta,\blambda)
+
\bar{\sfR}_N(\theta,\blambda).
\eea
\end{lemma}
\proof By Lemma \ref{20181031-1}, we have  
\bea\label{20181113-2} 
\frac{\partial}{\partial \theta} \varphi _{F_{N}} ^{\Psi} (\theta, \boldsymbol{\lambda} )
&=&
\sfR_N(\theta,\blambda)
+\varphi _{F_{N}} ^{\Psi} (\theta, \boldsymbol{\lambda})\sfP_N(\theta,\blambda).
\eea
Therefore, 
\beas
\frac{\partial}{\partial \theta} \varphi _{F_{N}} ^{\Psi} (\theta, \boldsymbol{\lambda} )
&=&
\sfR_N(\theta,\blambda)
+\varphi _{F_{N}} ^{\Psi} (0, \boldsymbol{\lambda})\sfP_N(\theta,\blambda)
+\int_0^\theta
\frac{\partial}{\partial\theta_1}\varphi _{F_{N}} ^{\Psi} (\theta_1, \boldsymbol{\lambda})d\theta_1
\sfP_N(\theta,\blambda).
\eeas
Recursively using the last expression, we obtain
\beas 
\frac{\partial}{\partial \theta} \varphi _{F_{N}} ^{\Psi} (\theta, \boldsymbol{\lambda} )
&=&
\sfR_N(\theta,\blambda)
+\varphi _{F_{N}} ^{\Psi} (0, \boldsymbol{\lambda})\sfP_N(\theta,\blambda)
\\&&
+\int_0^\theta
\sfR_N(\theta_1,\blambda)
d\theta_1
\sfP_N(\theta,\blambda)
+\varphi _{F_{N}} ^{\Psi} (0, \boldsymbol{\lambda})\int_0^\theta
\sfP_N(\theta_1,\blambda)
d\theta_1
\sfP_N(\theta,\blambda)
\\&&
+\int_0^\theta
\int_0^{\theta_1}
\frac{\partial}{\partial\theta_2}\varphi _{F_{N}} ^{\Psi} (\theta_2, \boldsymbol{\lambda})d\theta_2
\sfP_N(\theta_1,\blambda)d\theta_1
\sfP_N(\theta,\blambda)
\\&=&
\sfR_N(\theta,\blambda)
+\varphi _{F_{N}} ^{\Psi} (0, \boldsymbol{\lambda})\sfP_N(\theta,\blambda)
\\&&
+\int_0^\theta
\sfR_N(\theta_1,\blambda)
d\theta_1
\sfP_N(\theta,\blambda)
+\varphi _{F_{N}} ^{\Psi} (0, \boldsymbol{\lambda})
\int_0^\theta\sfP_N(\theta_1,\blambda)d\theta_1\sfP_N(\theta,\blambda)
\\&&
+\int_0^\theta
\int_0^{\theta_1}
\sfR_N (\theta_2, \boldsymbol{\lambda})d\theta_2
\sfP_N(\theta_1,\blambda)d\theta_1
\sfP_N(\theta,\blambda)
\\&&
+\varphi _{F_{N}} ^{\Psi} (0, \boldsymbol{\lambda})
\int_0^\theta\int_0^{\theta_1}\sfP_N(\theta_2,\blambda)d\theta_2
\sfP_N(\theta_1,\blambda)d\theta_1\sfP_N(\theta,\blambda)
\\&&
+\int_0^\theta
\int_0^{\theta_1}
\int_0^{\theta_2}
\frac{\partial}{\partial\theta_3}\varphi _{F_{N}} ^{\Psi} (\theta_3, \boldsymbol{\lambda})
d\theta_3\sfP_N(\theta_2,\blambda)
d\theta_2
\sfP_N(\theta_1,\blambda)d\theta_1
\sfP_N(\theta,\blambda).
\eeas
Switching back the last term by (\ref{20181113-2}), then we have 
Formula (\ref{20181113-1}) for $\sfk=2$. 
We may repeat this procedure to finally obtain (\ref{20181113-1}). 
\qed
\halflineskip
\end{en-text}

 The truncated characteristic function can be written in terms of the principal part $\sfP_N(\theta,\blambda)$ and of the {\cred residual} %rest 
 term $\sfR_N(\theta,\blambda)$.
\begin{lemma}\label{lll3}
Suppose that $F_N\in{\bbD_{1,p+1}}(\bbR^d)$ and that $[\Psi]$ holds. 
Then, for any positive number $c$, 
the functional $\varphi^\Psi_N(\theta,\blambda)$ admits the expression
\bea\label{20181119-5}
\varphi^\Psi_N(\theta,\blambda)
&=& 
\varphi^\Psi_N(0,\blambda)\exp\bigg(\int_0^\theta\sfP_N(\theta_1,\blambda)d\theta_1\bigg)
\nn\\&&
+\int_0^\theta \exp\bigg(\int_{\theta_1}^\theta\sfP_N(\theta_2,\blambda)d\theta_2\bigg)
\sfR_N(\theta_1,\blambda)d\theta_1
\eea
for 
$(\theta,\blambda)\in[0,1]\times{ \bbR^d}$. 
\end{lemma}
\proof By Lemma \ref{20181031-1}, we have  
\bea\label{20181113-2} 
\frac{\partial}{\partial \theta} \varphi _N^{\Psi} (\theta, \boldsymbol{\lambda} )
&=&
\sfR_N(\theta,\blambda)
+\varphi _N^{\Psi} (\theta, \boldsymbol{\lambda})\sfP_N(\theta,\blambda).
\eea
Solve (\ref{20181113-2}) to obtain (\ref{20181119-5}). 
\qed
\halflineskip

{\colorg 
In order to obtain the high-order asymptotic expansion, we need to assume several conditions. The first assumptions stated below concern the Malliavin regularity of $F_{N}$ and of its associated Gamma factors.}
We fix a positive number $\sfq$, {that will be the order of 
the asymptotic expansion.}
Given a positive integer $\ell$, we consider the following conditions. 
\bd 
\im[[A1\!\!]] %$_\ell$ 
{\bf(i)} 
$F_N\in\bbD_{\ell+1,\infty}$ %and $\Psi_N\in\bbD_{\ell,\infty}$, 
and 
\bea\label{20181113-3} 
\sup_{N\in\bbN}\big\|F_N\big\|_{\ell+1,r}%+\sup_{N\in\bbN}\big\|\Psi_N\big\|_{\ell,r}
&<&\infty
\eea
for every $r>1$. 
\bd
\im[(ii)] For some positive constant ${\sf a}$ 
{and some $d\times d$ non-singular matrix $C_1$ satisfying $C=2^{-1}\big(C_1+C_1^T\big)$, 
it holds that }
\beas 
\big\|\Gamma^{(2)}(F_N)-{ C_1}\big\|_{{ \ell,}r} &=& O(N^{-{\sf a}})
\eeas
as $N\to\infty$ for every $r>1$. 
\ed
\ed
\halflineskip

Let $\ell_1\in\{1,...,\ell\}$. 
Consider also the condition 
\bd 
\im[[A2\!\!]] 
{\bf (i)} For some $r>1$, 
\bea\label{20181113-12} 
\bigg\|
\Gamma^{(p+1)}_{i_{1}, ...,  i_{p+1}} (F_{N}) 
-\E\big[ \Gamma^{(p+1)}_{i_{1}, ...,  i_{p+1}} (F_{N})\big]
\bigg\|_{\ell_1,r}
&=&
O(N^{-\sfq})
\eea
as $n\to\infty$ 
{ for $i_1,...,i_{p+1}\in\{1,...,d\}$}. 
\bd 
\im[(ii)] For some $r>1$, 
\bea\label{20181114-1}
\bigg\|
\Gamma^{(p+1)}_{i_{1}, ...,  i_{p+1}} (F_{N}) 
-\E\big[ \Gamma^{(p+1)}_{i_{1}, ...,  i_{p+1}} (F_{N})\big]
\bigg\|_r
&=&
o(N^{-\sfq})
\eea
as $n\to\infty$ { for $i_1,...,i_{p+1}\in\{1,...,d\}$}. 
\ed
\ed
\halflineskip
%E.g., $\ell_1$ can be $0$ at this moment. When $\ell_1=0$, Condition $[A2]$ requires (ii) only. 

Obviously, a sufficient condition for $[A2]$ is 
\bd 
\im[[A2$^\sharp$\!\!]] 
For some $r>1$, 
\beas 
\bigg\|
\Gamma^{(p+1)}_{i_{1}, ...,  i_{p+1}} (F_{N}) 
-\E\big[ \Gamma^{(p+1)}_{i_{1}, ...,  i_{p+1}} (F_{N})\big]
\bigg\|_{\ell_1,r}
&=&
o(N^{-\sfq})
\eeas
as $n\to\infty$ { for $i_1,...,i_{p+1}\in\{1,...,d\}$}. 
\ed
\halflineskip

In what follows, we will work with $\Psi_N$ defined by 
$\Psi_N=\psi(\Xi_N)$ for 
a function $\psi\in C^\infty(\bbR;[0,1])$ such that 
$\psi(x)=1$ for $x\in[-1/2,1/2]$ and $\psi(x)=0$ for $x\in(-1,1)^c$, and a functional $\Xi_N$ given by 
\begin{equation}\label{8a-1}
\Xi_N
=
%N^{-2\xi'}|F_N|^2+
K_*N^{2{\sf a}'}\big|\Gamma^{(2)}(F_N)-{ C_1}\big|^2,
\end{equation}
where %$\xi'$ and  $K_*$ are positive numbers, 
{ $K_*$ is a positive number}
%${\sf a}'$ is a constant in $(0,{\sf a})$, 
and $\Gamma^{(2)}(F_N)=(\Gamma^{(2)}_{i,j}(F_N))_{i,j=1}^d$. 
By choosing a sufficiently large $K_*$, we may assume, {\colorg due to (\ref{8a-1}),} that there exist positive constants $c_1$ and $c_2$ 
such that 
\bea\label{201812300109} 
c_1\yleq{\big|}\det\Gamma^{(2)}(F_N){ \big|}\yleq c_2
\eea
for all $N\in\bbN$ and a.s. $\omega\in\Omega$ whenever $|\Xi_N|<1$. 
%By definition, Condition $[\Psi]$ is satisfied for $c_N=N^{\xi'}$ if $F_N\in\bbD_{2,4(p+1)}(\bbR^d)$. 

%We are assuming that $\Psi_N=0$ whenever $|F_N|>N^{\xi'}$ for some positive constant $\xi'$ independent of $N$. 

Let $\Lambda_N=\{\blambda\in\bbR^d;\>|\blambda|\leq N^\xi\}$, where $\xi$ is a positive constant. 
We will put a condition ($[A3]$) concerning $\xi$ later. 

\begin{lemma}\label{20181031-2}
%Let $\ell_1\in\{0,1,...,\ell\}$. %\in[0,\ell]\cap\bbZ_+$. 
Suppose that $[A1]$ is fulfilled. 
\bd
\im[(a)] 
Suppose that $[A2]$ $(i)$ is satisfied. 
Then, for any $\alpha\in\bbZ_+^d$, there exists a constant $C_\alpha$ such that 
\bea\label{20181113-4} 
\sup_{N\in\bbN}\sup_{\theta\in[0,1]}
1_{\{\blambda\in\Lambda_N\}}
N^\sfq
\bigg|\frac{\partial^\alpha}{\partial\blambda^\alpha}
\sfR_N(\theta,\blambda)\bigg|
&\leq&
C_\alpha(1+|\blambda|)^{-\ell_1+{ p+1}}
\quad(\blambda\in\bbR^d)
\eea
\im[(b)]
Suppose that $[A2]$ $(ii)$ is satisfied. 
Then 
\beas 
\sup_{\theta\in[0,1]}
\bigg|\frac{\partial^\alpha}{\partial\blambda^\alpha}
\sfR_N(\theta,\blambda)\bigg|
&=&
o(N^{-\sfq})
\eeas
as $n\to\infty$ for every $\blambda\in\bbR^d$ and $\alpha\in\bbZ_+^d$. 
\ed
\end{lemma}
\begin{remark}\rm 
(i) 
Under the assumption that $\sup_{N\in\bbN}\big\|F_N\big\|_{\ell+1,(p+1)r}<\infty$, we have 
\beas \sup_{N\in\bbN}\|\Gamma^{(p+1)}_{i_{1}, ...,  i_{p+1}} (F_{N})\|_{\ell,r}<\infty\eeas
thanks to Lemma \ref{20181102-1}. 
In this situation, 
Condition (\ref{20181113-12}) is requesting the order $N^{-\sfq}$ to the norm. 
(ii) 
It is possible to show a similar result under the condition that 
\beas 
\big\|
%\bigg(
\Gamma^{(p+1)}_{i_{1}, ...,  i_{p+1}} (F_{N})% -E \Gamma^{(p+1)}_{i_{1}, ...,  i_{p+1}} (F_{N})\bigg)
\big\|_{\ell_1,r}
&=&
o(N^{-\sfq})
\eeas
%for some $r>1$.  
%if we add the $(p+1)$th order of term to $\sfP_N(\theta,\blambda)$. 
instead of $[A2]$ 
{ though we will not pursue it here.} 
{ (iii) Logically, $\ell$ in Lemma \ref{20181031-2} can be $\ell_1$. }
\end{remark}
{\it Proof of Lemma \ref{20181031-2}.} 
\begin{en-text}
\noindent
{ {\bf Strategy for the proof (Later erased).} Use IBP and also the decay of 
$\Gamma^{(p+1)}_{i_{1}, ...,  i_{p+1}} (F_{N})% -E \Gamma^{(p+1)}_{i_{1}, ...,  i_{p+1}} (F_{N})
$. 
This is a subtle part. 
It is possible to get minus power by $\Lambda_N$, which reduces differentiability of 
$\Gamma^{(p+1)}_{i_{1}, ...,  i_{p+1}} (F_{N}) %-E \Gamma^{(p+1)}_{i_{1}, ...,  i_{p+1}} (F_{N})
$. 
Or if $p$ is large, then we only need estimate of $\|\Gamma^{(p+1)}_{i_{1}, ...,  i_{p+1}} (F_{N})\|_1$ 
if it can suppress $|\blambda|^{(p+1)+(d+1)}$. 
}
\end{en-text}
%%
%%%%%%%%%%%%%%%%%%%%%%%%%%
% 20190420 Remove estimates by Cauchy %%
{ 
Let $\widetilde{\Gamma}^{(p+1)}_{i_{1}, ...,  i_{p+1}} (F_{N})
=\Gamma^{(p+1)}_{i_{1}, ...,  i_{p+1}} (F_{N})
-\E\big[\Gamma^{(p+1)}_{i_{1}, ...,  i_{p+1}} (F_{N})\big]$. 
For $\alpha\in\bbZ_+^d$, we have 
\bea\label{201904201101} 
\frac{\partial^\alpha}{\partial\blambda^\alpha}\sfR_N(\theta,\blambda)
&=&
{\tt i} ({\tt i} \theta)^{p}  \sum_{i_{1}, ..., i_{p+1} =1} ^{d} 
\left(\begin{array}{c}\alpha\\ \gamma\end{array}\right)
E_{\gamma,N}(\theta,\blambda)
\frac{\partial^{\alpha-\gamma}}{\partial\blambda^{\alpha-\gamma}}\big(\lambda _{i_{1}}....\lambda _{i_{p+1}}\big)
\nn\\&&
+ \sum_{j=1}^p\frac{\partial^\alpha}{\partial\blambda^\alpha}R_{j,N}(\theta,\blambda)%+ R_{2,N} +...+ R_{p, N}\right)
\eea
where 
\beas 
E_{\gamma,N}(\theta,\blambda)
&=&
\E \bigg[\Psi_{N} \frac{\partial^\gamma}{\partial\blambda^\gamma}e(\theta, \boldsymbol{\lambda}, F_{N})
\widetilde{\Gamma}^{(p+1)}_{i_{1}, ...,  i_{p+1}} (F_{N})\bigg]. 
\eeas
By definition, 
\bea\label{201904201109} 
\frac{\partial^\gamma}{\partial\blambda^\gamma}e(\theta, \boldsymbol{\lambda}, F_{N})
&=&
p_\gamma\big(\theta F_N,(1-\theta^2)C\blambda,(1-\theta^2)C\big)\>e(\theta, \boldsymbol{\lambda}, F_{N})
\eea
for a polynomial $p_\gamma$. 

By Lemma \ref{20181102-1} { and $[A1]$ (i)}, 
$\Gamma^{(p+1)}_{i_{1}, ...,  i_{p+1}} (F_{N})\in\bbD_{\ell,r_1}$ for any $r_1>1$. 
We apply the IBP formula {\colorg (see Lemma 2 in \cite{TY})} to 
{\colorg the interpolation functional} $e(\theta,\blambda,F_N)$ 
that is taking advantage of (\ref{201812300109}) and 
based on the chain rule 
\beas 
\big\langle Df(F_N),D{ (-L)}^{-1}F_N^{(i_1)}\big\rangle_H 
&=& 
\sum_{i_2=1}^d\frac{\partial f}{\partial x_{i_2}}(F_N)\Gamma^{(2)}_{i_1,i_2}(F_N)
\eeas
for $f\in C^1(\bbR^d)$ of at most polynomial growth,  
in order to obtain 
\beas &&
({\tt i}\theta\blambda)^\beta
\E \bigg[\Psi_{N} e^{{\tt i}\theta\langle \blambda,F_N\rangle}
p_\gamma\big(\theta F_N,(1-\theta^2)C\blambda,(1-\theta^2)C\big)
\widetilde{\Gamma}^{(p+1)}_{i_{1}, ...,  i_{p+1}} (F_{N}) \bigg]
\nn\\&=&
\E \bigg[e^{{\tt i}\theta\langle \blambda,F_N\rangle}
\Phi^{F_N}_\beta\bigg(\Psi_N
p_\gamma\big(\theta F_N,(1-\theta^2)C\blambda,(1-\theta^2)C\big)
\widetilde{\Gamma}^{(p+1)}_{i_{1}, ...,  i_{p+1}} (F_{N})\bigg) \bigg]
\eeas
for $\blambda\in\Lambda_N$ and 
$\beta\in\bbZ_+^d$ with { the sum of the components} $|\beta|=\ell_1$. 
Here $\Phi^{F_N}_\beta$ 
is a linear functional and we see that %in $L^{\infty-}=\cap_{r>1}L^r$, and 
\bea\label{20181101-1}
\sup_{{N\in\bbN\atop\theta\in[0,1]}\atop\blambda\in\bbR^d}
\bigg\|
\Phi^{F_N}_\beta\bigg(\Psi_N
p_\gamma\big(\theta F_N,(1-\theta^2)C\blambda,(1-\theta^2)C\big)
N^\sfq\widetilde{\Gamma}^{(p+1)}_{i_{1}, ...,  i_{p+1}} (F_{N})\bigg) 
e^{-2^{-1}(1-\theta^2)\blambda^TC\blambda}
\bigg\|_1
&<&
\infty.
\nn\\&&
\eea
Here we used (\ref{20181113-3}), (\ref{20181113-12}) for which $r$ is strictly larger than $1$, 
and the estimate 
\beas 
\big|(1-\theta^2)C\blambda\big|
&\leq&
\big|C^{1/2}\big|\big(1+(1-\theta^2)\blambda^TC\blambda\big).
\eeas
For each $\lambda=(\lambda_i)\in\check{\Lambda}_N$, we choose a component $i$ 
that attains $\max\{|\lambda^j|;\>j=1,...d\}$ and use the above estimate for $\beta=(i,i...,i)$. 
Then we obtain 
\bea\label{201812300239} 
\sup_{n\in\bbN,\>\theta\in[1/2,1],\>\blambda\in\bbR^d}
\bigg\{N^\sfq
|\blambda|^{\ell_1}
\big|E_{\gamma,N}(\theta,\blambda)\big|\bigg\}
&<&
\infty.
\eea
Moreover, there exists a positive constant $c_0$ such that 
\bea\label{201812300252}
\sup_{N\in\bbN,\>\theta\in[0,1/2)}\big|e(\theta,\blambda,F_N)\big|
&\leq&
c_0^{-1}\exp\big(-c_0|{\bf \blambda}|^2\big)
\eea
for all $\blambda\in\bbR^d$. 
From (\ref{201812300239}) and (\ref{201812300252}), we obtain 
\bea\label{201812300253}&&
\sup_{N\in\bbN}\sup_{\theta\in[0,1]}
%1_{\{\blambda\in\check{\Lambda}_N\}}N^\sfq
\bigg|
E_{\gamma,N}(\theta,\blambda)
\frac{\partial^{\alpha-\gamma}}{\partial\blambda^{\alpha-\gamma}}\big(\lambda _{i_{1}}....\lambda _{i_{p+1}}\big)
\bigg|
\nn\\&\leq&
K_0(1+|\blambda|)^{-\ell_1+{ p+1}}
\eea
for all $\lambda\in\bbR^d$, where $K_0$ is a constant depending on $\alpha-\gamma$ 
but independent of $\lambda\in\bbR^d$. 
}
%%%%%%%%%%%%%%%%%%%%%%%%%%
%%%%%%%%%%%%%%%%%%%%%%%%%%s
%
\begin{en-text}
In particular, 
\bea\label{20181101-2} 
\sup_{N\in\bbN}
\big\||F_N|
\Phi^{F_N}_\beta\big(\Psi_NN^\sfq\Gamma^{(p+1)}_{i_{1}, ...,  i_{p+1}} (F_{N})\big)\big\|_1
&<&
\infty. 
\eea
\end{en-text}

Since $\|D\Psi_N\|_2=O(N^{-m})$ for every $m>0$, 
for $\Lambda_N=\{|\blambda|\leq N^\xi\}$ and for every $k>0$ and $s>0$, 
we see
\bea\label{20181101-3} 
\limsup_{N\to\infty}\sup_{\blambda\in\Lambda_N}
\bigg(N^s|\blambda|^k\bigg|
{ \frac{\partial^\alpha}{\partial\blambda^\alpha}}R_{j,N}(\theta,\blambda)\bigg|\bigg) &\simleq& 
%\limsup_{N\to\infty}\sup_{\blambda\in\Lambda_N} |\blambda|^jN^{-m}\>\leq\> |\blambda|^{j-m/\epsilon},
\limsup_{N\to\infty}N^{s+(k+j)\xi-m}\>=\>0
\eea
%which converges to $0$ as fast as we like. 
if we choose a sufficiently large $m$. 
Now, from (\ref{201812300253}) and (\ref{20181101-3}), { we obtain 
Inequality (\ref{20181113-4}).}

The property (b) is rather easy to prove since it only requires $\blambda$-wise estimate 
of an explicit expression of $\frac{\partial^\alpha}{\partial\blambda^\alpha}
\sfR_N(\theta,\blambda)$ { by using $[A2]$ (ii).} 
\qed
\ \\

Estimate of 
$
\partial^\alpha/\partial\blambda^\alpha\varphi _{F_{N}} ^{\Psi} (\theta, \boldsymbol{\lambda})
$ 
will be necessary in the proof of Proposition \ref{20181118-8} down below. 
The only possibility is to repeatedly apply the IBP formula. 
%This term appears outside and inside (\ref{20181030-1}). 
The estimate outside of $\Lambda_N$ cannot gain improvement by the decay of cumulants, 
and it is the worst one, in other words, it uses the highest order of repetition of the IBP formula 
among other terms. 

\begin{lemma}\label{20181114-3}
Suppose that $[A1]$ { (i)} %the condition (\ref{20181113-3}) 
is fulfilled. 
Then, for every $\alpha\in\bbZ_+^d$, 
there exists a constant $C_\alpha$ such that 
\bea\label{20181113-10} 
\bigg|\frac{\partial^\alpha}{\partial\blambda^\alpha}\varphi _N^{\Psi} (\theta, \boldsymbol{\lambda})
\bigg|
&\leq&
C_\alpha(1+|\blambda|)^{-\ell}
\eea
for all $\theta\in[0,1]$, $\blambda\in\bbR^d$ and $N\in\bbN$,{\colorg  with $\ell$ from condition $[A1]$. }
\end{lemma}
\proof 
{ 
This lemma can be proved by estimatrion quite similar to that for 
$E_{\gamma,N}(\theta,\blambda)$ in the proof of Lemma \ref{20181031-2}. 
We can repeat the IBP formula {\colorg (i.e. Lemma 2 in \cite{TY})} $\ell$-times in this case for $\theta\in[1/2,1]$. 
Estimation for $\theta\in[0,1/2)$ is also similar. 
}
\qed\halflineskip

\begin{en-text}
Let 
\beas 
A^\Psi_N(\theta,\blambda)
&=&
\varphi _{F_{N}} ^{\Psi} (0, \boldsymbol{\lambda})
\check{\sfP}_N(\theta,\blambda).
\eeas
\end{en-text}
%

%%%%%%%%%%%%%
{\colorg Recall the expression (\ref{20181119-5}) of the characteristic function $\varphi ^{\Psi}_{N}(\theta, \blambda) $  in Lemma \ref{lll3}.  Based on the estimated in Lemma \ref{20181031-2}, we can show that the last summand in the right-hand side can be neglegted. Consequently, the dominant part of  $\varphi ^{\Psi}_{N}(\theta, \blambda) $ 
will {\cred come} %vome 
from the first term in right-hand side of (\ref{20181119-5}), which involves the exponential of the principal part. We will introduce a new random variable, in which we keep the first terms in the Taylor expansion  of the exponential function. }For $\sfk\in\bbN$, let 
\beas 
\sfP^*_N(\blambda)
&=&
\sum_{j=0}^\sfk\frac{1}{j!}\bigg(\int_0^1 \sfP_N(\theta,\blambda)d\theta\bigg)^j
\eeas
{ and}
\beas 
\sfR^*_N(\blambda)
&=&
\varphi _N^{\Psi} (0, \boldsymbol{\lambda})
\sum_{j=\sfk+1}^\infty\frac{1}{j!}\bigg(\int_0^1 \sfP_N(\theta,\blambda)d\theta\bigg)^j
+\int_0^1 \exp\bigg(\int_{\theta_1}^1\sfP_N(\theta_2,\blambda)d\theta_2\bigg)
\sfR_N(\theta_1,\blambda)d\theta_1.
\eeas
%
%\beas 
%\sfA^\Psi_N(\blambda) &=& \varphi _{F_{N}}^{\Psi} (0, \boldsymbol{\lambda})\>\sfP^*_N(\blambda)
%\yeq \E\big(\Psi_N\big)e^{-\blambda^TC\blambda/2}\>\sfP^*_N(\blambda). 
%\eeas
%
Then, if one has the expansion (\ref{20181119-5}), then 
\bea\label{201901011614} 
\varphi_N^\Psi(1,\blambda) 
&=& 
\varphi_N^\Psi(0,\blambda)\sfP^*_N(\blambda)+\sfR^*_N(\blambda). 
\eea
\halflineskip

{\colorg At this point, let us make a brief summary to comment on the role of the parameters that appear in our work. Recall that $d$ is the dimension of the random vector, $p+1$ is the maximum order of the cumulant that appear in the decomposition of the characteristic functional $\varphi^{\Psi}_{N}(\theta, \blambda)$, $k$  comes from the Taylor expansion of the exponential function while $\ell, \ell_{1}, q$ appear in assumptions $[A1]-[A2]$.  $N^{-q} $ will be the  the size of the error term. Next, we need to assume some relation between all these parameters. Condition (\ref{20181118-11}) will be needed to evaluate 
the {\cred residual} terms.  }
%%%%%%%%%%%%%

\bd
\im[[A3\!\!]] %$_{p,\sfk}$ 
%There exist positive numbers $\sfq_0$ and $\xi$ such that  
The numbers $\sfq_0\in(0,\infty)$, $\xi\in(0,\infty)$, $\ell\in\bbN$ and $\ell_1\in\bbN$ 
satisfy 
\beas 
\sfq_0(\sfk+1)>\sfq,%\quad\frac{\sfq}{\ell-d}\><\>\xi\><\>\frac{\ell_1-d}{p+1}
{\quad\xi(\ell-d)>\sfq,\quad \ell_1>p+1+d}
\eeas
and 
\bea\label{20181118-11} 
N^{\sfq_0+2\xi}\big| \E\big[\Gamma^{(2{ sym})}(F_N)\big]-C\big|
+\sum_{j=3}^{p+1}N^{\sfq_0+j\xi}\big|\E\big[\Gamma^{(j)}(F_N)\big]\big|
&=& 
O(1)
\eea
as $N\to\infty$. {\colorg The expectation of a matrix {\cred is} understood componentwise and $\big| \cdot \big|$ denotes the Euclidean norm. }
\ed
\halflineskip

\begin{remark}\rm 
{\bf (a)} 
%%Condition $[A3]$ is implicitly requiring the following relations:
%\bea\label{20181118-1} 
%d+1\leq\ell_1\leq\ell,\quad
%\frac{\sfq}{\ell-d}&<&\frac{\ell_1-d}{p+1}. 
%\eea
Condition (\ref{20181118-11}) imposes a restriction on $\xi$ from above. 
Besides, $p$ will be asked to be large according to the value of $\sfq$ in order to satisfy $[A2]$. 
In this sense, $p$ is a function of $\sfq$: $p=p(\sfq)$. 
\bd
\im[(b)] 
If we take $\ell_1=\ell$, then Condition {$[A3]$ requires 
\beas 
%\sfq (p+1) < (\ell-d)^2.
\ell &>& d+\max\bigg\{\frac{\sfq}{\xi},p+1\bigg\}
\eeas
to $\ell$. 
If $d=1$, $p=1$, then $\ell\geq4$ at least. 
}
\im[(c)] 
Increasing $\sfk$ causes only increase of complexity of the formula 
and does not require $F_N$ to pay more, 
once we found $p$ for $[A2]$ and $\sfq_0$ for (\ref{20181118-11}). 
\im[(d)] 
The index $\sfq_0$ depends on the largest order of terms among the terms appearing 
in (\ref{20181118-11}). 
Often the term for $j=3$ dominates other terms. 
\im[(e)] 
{
In most common situations, the order $p$ of the controllable cumulants 
is the essential parameter. It determines the possible order $\sfq$ of the asymptotic expansion. 
The parameters $\ell$, $\sfk$, $\xi$ and $\sfq_0$ are somewhat technical but 
automatically determined by $p$ and the dimension $d$. See Section \ref{201908230639}, 
that treats a regularly ordered expansion. 
}
\ed
\end{remark}
\halflineskip

\begin{lemma}\label{20181117-6}
\bd
\im[(a)] 
Suppose that Conditions $[A1]$, $[A2]$ $(i)$ %(\ref{20181113-12}) 
and $[A3]$ are satisfied. 
Then 
there exists a constant $K$ such that  
\beas
N^\sfq1_{\{\blambda\in\Lambda_N\}}
\bigg|\frac{\partial^\alpha}{\partial\blambda^\alpha}
%\bigg\{\varphi^\Psi_N(1,\blambda)-\sfA^\Psi_N(\blambda)\bigg\}
\sfR^*_N(\blambda)
\bigg|
&\leq&
K(1+|\blambda|)^{-\ell_1+{ p+1}}
\eeas
for all $(\lambda,N)\in\bbR^d\times\bbN$. 
\im[(b)] Suppose that Conditions $[A1]$, $[A2]$ $(ii)$ %(\ref{20181114-1}) 
and $[A3]$ are satisfied. 
Then 
\beas 
%\sup_{\theta\in[0,1]}\bigg|
\frac{\partial^\alpha}{\partial\blambda^\alpha}
%\bigg\{\varphi^\Psi_N(1,\blambda)-\sfA^\Psi_N(\blambda)\bigg\}
\sfR^*_N(\blambda)
%\bigg|
&=&
o(N^{-\sfq})
\eeas
as $n\to\infty$ for every $\blambda\in\bbR^d$. 
\ed
\end{lemma}
\proof 
It follows from $[A3]$ that 
\bea\label{20181117-8} 
\sup_{\theta\in[0,1],\>\blambda\in\bbR^d}
1_{\{\blambda\in\Lambda_N\}}\bigg|
\frac{\partial^\alpha}{\partial\blambda^\alpha}\sfP_N(\theta,\blambda)\bigg|
&=&
O(N^{-\sfq_0})
\eea
as $N\to\infty$ for every $\alpha\in\bbZ_+^d$. 
Therefore $P_N$ undertakes $\blambda$'s by itself. 

The property (a) follows from Lemma \ref{20181031-2} (a), %Lemma \ref{20181114-3} 
and the representations of $\sfR^*_N(\blambda)$ %and $\bar{\sfR}_N(\theta,\blambda)$ 
by $\sfR_N(\theta,\blambda)$ and $\sfP_N(\theta,\blambda)$ % and $\varphi^\Psi_{F_N}(\theta,\blambda)$. 
{ and $\sfq_0(\sfk+1)>\sfq$.}
Similarly, the property (b) can be proved by using Lemma \ref{20181031-2} (b). 
\qed

\section{Asymptotic expansion of the expectation and its error bound}

The approximate density of the sequence $(F_{N}) _{N \geq 1} $ is defined as the Fourier inverse of the dominant part of the truncated characteristic function of $F_{N}$ as follows. 
\bea\label{20181120-1} 
\label{pn}
f_{N, p,\sfk} (x)= \frac{1}{(2\pi) ^{d} } \int_{\mathbb{R}^{ d}} e^{-{\tt i} \langle \boldsymbol{\lambda},  x\rangle} \varphi _{N,p,\sfk} (\boldsymbol{\lambda} ) d\boldsymbol{\lambda}, \hskip0.5cm x\in \mathbb{R}^{d}
\eea
where $\varphi _{N,p,\sfk}$ is given by 
\bea\label{20181120-2} 
\varphi _{N,p,\sfk}(\blambda)
&=&
e^{-\half\blambda^T C\blambda}\>\sfP^*_N(\blambda).
\eea
We recall that $\sfP^*_N$ depends on $p$ and $\sfk$. 
Obviously by definition, $\varphi _{N,p,\sfk}(\blambda)$ may include 
terms of higher-order than $\sfq$. 
It is necessary to find an expansion of each term in the expression (\ref{20181117-1}) of 
$\sfP_N(\theta,\blambda)$ if one wants to extract the principal part from the expansion. 
The { principal part} depends on the structure of the model in question, and really 
we will specify it in the later sections.

Let 
\bea\label{pfn}
\varphi ^{\Psi}_N(\boldsymbol{\lambda})
&=& 
\varphi^\Psi_N(1,\blambda)
\>\equiv\>
\mathbf{E} \big[\Psi_{N} e ^{{\tt i}\langle \boldsymbol{\lambda} , F_{N}\rangle}\big]. 
\eea
%Then $\varphi ^{\Psi}_N(\boldsymbol{\lambda})=\varphi _{N}  ^{\Psi} (1,\boldsymbol{\lambda})$. 
%
The local  (or truncated) density $p_N^{\Psi} $ of the random variable $F_{N}$ is defined as the inverse Fourier transform of the truncated characteristic function: 
\begin{equation*}
p_N^{\Psi} (x)= \frac{1}{(2\pi) ^{d}} \int_{\mathbb{R} ^{d}} e^{-{\tt i}\langle \boldsymbol{\lambda}, x\rangle} \varphi ^{\Psi}_N(\boldsymbol{\lambda}) d\boldsymbol{\lambda}
\end{equation*}
for $x\in \mathbb{R} ^{d} $. 
The local density is well-defined since obviously the truncated characteristic function $\varphi ^{\Psi} _{F_{N}} $ is integrable over $\mathbb{R}^{d} $ 
under $[A1]$ by Lemma \ref{20181114-3} { if $\ell>d$}. 

{\colorg  We estimate below the difference between  the approximate density  and the local density.  }
\begin{prop}\label{20181118-8}
Suppose that Conditions $[A1]$, $[A2]$ and $[A3]$ are satisfied. 
Then
\bea\label{20181117-5} 
\sup_{x\in\bbR^d}\bigg(|x|^m\big|p_N^{\Psi} (x)-f_{N, p,\sfk} (x)\big|\bigg)
&=&
o(N^{-\sfq})
\eea
as $N\to\infty$ for every $m>0$. 
\end{prop}
\proof 
\begin{en-text}
From Lemma \ref{20181114-11}, we have 
\beas
\varphi _{F_{N}} ^{\Psi} (1, \boldsymbol{\lambda} )-\varphi _{F_{N}} ^{\Psi} (0, \boldsymbol{\lambda} )
&=&
\int_0^1 
\frac{\partial}{\partial \theta} \varphi _{F_{N}} ^{\Psi} (\theta, \boldsymbol{\lambda} )d\theta
\\&=&
\varphi _{F_{N}} ^{\Psi} (0, \boldsymbol{\lambda})
\int_0^1 \check{\sfP}_N(\theta,\blambda)d\theta
+\int_0^1 \check{\sfR}_N(\theta,\blambda)d\theta
+
\int_0^1 \bar{\sfR}_N(\theta,\blambda)d\theta.
\eeas
\end{en-text}
Let 
\bea\label{20181117-7} 
\widetilde{\sfR}_N(\blambda)
&=&
\big\{\varphi^\Psi_N(0,\blambda)-e^{-\half\blambda^TC\blambda}\big\}
%\bigg(1+\int_0^1\check{\sfP}_N(\theta,\blambda)d\theta\bigg)
\sfP^*_ N(\blambda)
\nn\\&=&
\big\{\bE(\Psi_N)-1\big\}e^{-\half\blambda^TC\blambda}
%\bigg(1+\int_0^1\check{\sfP}_N(\theta,\blambda)d\theta\bigg).
\sfP^*_ N(\blambda). 
\eea
Since
\beas 
\varphi_{N,p,\sfk}(\blambda)
&=&
\varphi^\Psi_N(0,\blambda)%\bigg(1+\int_0^1\check{\sfP}_N(\theta,\blambda)d\theta\bigg)
\sfP^*_ N(\blambda)
%\\&&
%-\big\{\varphi^\Psi_{F_N}(0,\blambda)-e^{-\half\blambda^TC\blambda}\big\}
-\widetilde{\sfR}_N(\blambda)
%\bigg(1+\int_0^1\check{\sfP}_N(\theta,\blambda)d\theta\bigg)\sfP^*_ N(\blambda)
%\\&=&
%\varphi^\Psi_{F_N}(0,\blambda)+\int_0^1A^\Psi_N(\theta,\blambda)d\theta-\widetilde{\sfR}_N(\blambda)
,
\eeas
we have 
\bea\label{20181119-1}
\varphi _N^{\Psi} (\boldsymbol{\lambda} )-\varphi_{N,p,\sfk}(\blambda)
&=&
%\varphi _{F_N}^{\Psi} (\boldsymbol{\lambda})-\sfA^\Psi_N(\blambda)
\sfR^*_N(\blambda)
+
\widetilde{\sfR}_N(\blambda)
\eea
from (\ref{201901011614}).

By the integrability ensured by { Lemmas \ref{20181114-3} and \ref{20181117-6}}
and by { (\ref{20181117-7})}, %(\ref{20181119-1}), 
we obtain 
\beas 
\bigg|x^\alpha \bigg(p_N^{\Psi} (x)-f_{N, p,\sfk} (x)\bigg)\bigg|
&=&
\frac{1}{(2\pi) ^{d}} \bigg|
\int_{\mathbb{R} ^{d}} e ^{-{\tt i}\langle \boldsymbol{\lambda}, x\rangle} 
\bigg(\frac{\partial^\alpha}{\partial\blambda^\alpha}
\varphi ^{\Psi} _N(\boldsymbol{\lambda}) -
\frac{\partial^\alpha}{\partial\blambda^\alpha}
\varphi_{N,p,\sfk}(\blambda)
\bigg)d\blambda\bigg|
\\&\leq&
\bigg|
\int_{\Lambda_N} e ^{-{\tt i}\langle \boldsymbol{\lambda}, x\rangle} 
\bigg(\frac{\partial^\alpha}{\partial\blambda^\alpha}
\varphi ^{\Psi} _N(\boldsymbol{\lambda}) -
\frac{\partial^\alpha}{\partial\blambda^\alpha}
\varphi_{N,p,\sfk}(\blambda)
\bigg)d\blambda\bigg|
\\&&
+
\int_{\Lambda_N^c} 
\bigg|\frac{\partial^\alpha}{\partial\blambda^\alpha}
\varphi ^{\Psi} _N(\boldsymbol{\lambda}) \bigg|d\blambda
+ 
\int_{\Lambda_N^c} 
\bigg|
\frac{\partial^\alpha}{\partial\blambda^\alpha}
\varphi_{N,p,\sfk}(\blambda)
\bigg|d\blambda
\\&\leq&
\int_{\Lambda_N} 
%\sup_{\theta\in[0,1]}\bigg|\frac{\partial^\alpha}{\partial\blambda^\alpha}\bigg\{\varphi ^{\Psi} _{F_{N}} (\theta,\boldsymbol{\lambda}) -A^\Psi_N(\theta,\blambda)\bigg\}\bigg|d\blambda
\bigg|\frac{\partial^\alpha}{\partial\blambda^\alpha}\sfR_N^*(\blambda)\bigg|d\blambda
+\int_{\Lambda_N}\bigg|\frac{\partial^\alpha}{\partial\blambda^\alpha}\widetilde{\sfR}_N(\blambda)\bigg|
d\blambda
\\&&
+
\int_{\Lambda_N^c} 
\bigg|\frac{\partial^\alpha}{\partial\blambda^\alpha}
\varphi ^{\Psi} _N(\boldsymbol{\lambda}) \bigg|d\blambda
+ 
\int_{\Lambda_N^c} 
\bigg|
\frac{\partial^\alpha}{\partial\blambda^\alpha}
\varphi_{N,p,\sfk}(\blambda)
\bigg|d\blambda
\eeas
for every $\alpha\in\bbZ_+^d$. 
Thus, we obtain (\ref{20181117-5}) by the following estimates. 
\beas 
\lim_{N\to\infty}N^\sfq\int_{\Lambda_N} %\sup_{\theta\in[0,1]}
\bigg|
\frac{\partial^\alpha}{\partial\blambda^\alpha}
%\bigg\{\varphi ^{\Psi} _{F_{N}} (\theta,\boldsymbol{\lambda}) -A^\Psi_N(\theta,\blambda)\bigg\}
\sfR^*_N(\blambda)
\bigg|d\blambda
&=&
0
\eeas
by the dominated convergence theorem with Lemma \ref{20181117-6} 
since %$\sfq_0(\sfk+1)>\sfq$ and 
$\ell_1>{ p+1}+d$ by $[A3]$. 
Since (\ref{20181117-8}) holds and 
$\bE(\Psi_N)-1=O(N^{-L})$ as $N\to\infty$ for any $L>0$, 
the representation (\ref{20181117-7}) gives 
\beas 
\lim_{N\to\infty}N^L\int_{\Lambda_N}\bigg|\frac{\partial^\alpha}{\partial\blambda^\alpha}\widetilde{\sfR}_N(\blambda)\bigg|
d\blambda
&=&
0
\eeas
for every $L>0$. 
By Lemma \ref{20181114-3}, 
\beas
\int_{\Lambda_N^c} 
\bigg|\frac{\partial^\alpha}{\partial\blambda^\alpha}
\varphi ^{\Psi} _N (\boldsymbol{\lambda}) \bigg|d\blambda
&\leq&
C_\alpha\int_{\Lambda_N^c}(1+|\blambda|)^{-\ell}d\blambda
\\&\simleq&
\int_{r>N^\xi}r^{-\ell+d-1}dr
\\&=&
O(N^{-(\ell-d)\xi})
\yeq
o(N^{-\sfq})
\eeas
as $N\to\infty$ since $(\ell-d)\xi>\sfq$ by $[A3]$. 
Moreover 
\beas 
\int_{\Lambda_N^c} 
\bigg|
\frac{\partial^\alpha}{\partial\blambda^\alpha}
\varphi_{N,p,\sfk}(\blambda)
\bigg|d\blambda
&=&
O(N^{-L})
\eeas
as $N\to\infty$ by the Gaussian factor with the assistance of (\ref{20181118-11}). 
\qed
\halflineskip\halflineskip
\begin{en-text}
{
\noindent
\underline{Consideration} (erased later) 
%
%%\begin{en-text}
We needed $\ell_1>(p+1)\xi+d$ and $(\ell-d)\xi>\sfq$. 
Then 
\beas 
\frac{\sfq}{\ell-d}\><\>\xi\><\>\frac{\ell_1-d}{p+1}
\eeas
and hence we need 
\begin{screen}
\bea\label{20181117-10} 
\frac{\sfq}{\ell-d}&<&\frac{\ell_1-d}{p+1}. 
\eea
\end{screen}
%%\end{en-text}
%
In standard cases, probably we need $p\sim 2\sfq$. 
In this situation, a possible choice is $\ell_1\sim\sfq$ and $\ell\sim2\sfq$. 
If the highest order is $N^{-K/2}$, then this means $K=2\sfq$, therefore 
$\ell_1\sim K/2$, $\ell\sim K$ and $p\sim K$. 
A slightly sharper choice is $\ell=\ell_1\sim\sqrt{2}\sfq$. 
By this consideration, we see there are many combinations. 
So we will leave a choice of these parameters; just assume (\ref{20181117-10}). 
}
\end{en-text}

For $a,b>0$, we denote by $\cale(a,b)$ the set of measurable functions $g$ { on $\bbR^d$} 
satisfying $|g(x)|\leq a(1+|x|)^b$. 
\halflineskip
\begin{theorem}\label{20181120-21}
Suppose that Conditions $[A1]$, $[A2]$ and $[A3]$ are satisfied. 
Then,  for 
\beas 
\sup_{g\in\cale(a,b)}\bigg|\E\big[g(F_N)\big]-\int_{\bbR^d}g(x)f_{N,p,\sfk}(x)dx\bigg|
&=&
o(N^{-\sfq})
\eeas
as $N\to\infty$ for every $a,b>0$. 
\end{theorem}
\proof 
Since the Fourier transform $\varphi^\Psi_N$ of the measure 
$\nu(dx):=\E[\Psi_N|F_N=x]P^{F_N}(dx)$ is integrable, 
$\nu$ has a continuous density $p^\Psi_N$. 
Applying Proposition \ref{20181118-8}, we have  
\beas &&
\sup_{g\in\cale(a,b)}\bigg|\E\big[\Psi_Ng(F_N)\big] -\int_{\bbR^d}g(x)f_{N,p,\sfk}(x)dx\bigg|
\\&=&
\sup_{g\in\cale(a,b)}\bigg|\int_{\bbR^d}g(x)p^\Psi_N(x)dx-\int_{\bbR^d}g(x)f_{N,p,\sfk}(x)dx\bigg|
\\&\leq&
\sup_{g\in\cale(a,b)}\int_{\bbR^d}\big|g(x)\big|\big|p^\Psi_N(x)-f_{N,p,\sfk}(x)\big|dx
\\&\leq&
\sup_{x\in\bbR^d}\bigg((1+|x|)^{d+1+b}\big|p^\Psi_N(x)-f_{N,p,\sfk}(x)\big|\bigg)
\int_{\bbR^d}a(1+|y|)^{-d-1}dy
\\&=&
o(N^{-\sfq})
\eeas
as $N\to\infty$ for every $a,b>0$. 
This completes the proof because 
\beas 
\sup_{g\in\cale(a,b)}\bigg|\E\big[g(F_N)\big]
-\E\big[\Psi_Ng(F_N)\big]\bigg|
&\leq&
\big\|1-\Psi_N\big\|_2 \sup_{N'\in\bbN}\big\|a(1+|F_{N'}|)^b\big\|_2
\yeq O(N^{-L})
\eeas
as $N\to\infty$ for every $L>0$. 
\qed
%\halflineskip
%%
\begin{en-text}
{
\noindent--------------------------------\\
Memo (erased later). 
\begin{itemize}
\im
\sout{$\sfk+1\geq \gamma+\frac{p-1}{2}$ $=p/2?$}
We need about $2\sfq$ (or $\sfq/\gamma$) times derivative to obtain $\sfq$-th order of expansion 
(see [D2] below).

\im About $\partial_\blambda^\alpha\frac{\partial}{\partial \theta} \varphi _{F_{N}} ^{\Psi} (\theta, \blambda)
$. 
We may consider the first term of $\sfR_N(\theta_{\sfk+1},\blambda)$ and 
$ \varphi^\Psi_{F_N}(\theta_{\sfk+1},\blambda)$; 
$ \varphi^\Psi_{F_N}(0,\blambda)$ times a power of $\blambda$ is estimable, and 
$\sfP_N(\theta,\blambda)$ is just polynomial of $\blambda$. 
The cause that increases the power of $\lambda$ is only derivatives of 
$e^{-(1-\theta^2)\frac{\blambda^TC\blambda}{2}}$. 
However we have the estimate 
\beas 
\{|\blambda|(1-\theta^2)\}^k e^{-(1-\theta^2)\frac{\blambda^TC\blambda}{2}}
&\leq& 
%e^{-(1-\theta^2)\frac{\blambda^TC\blambda}{2}}\quad(\text{if lambda is less than one)}
%\\&&+
\frac{C_k}{|\blambda|^k}e^{-(1-\theta^2)\frac{\blambda^TC\blambda}{4}}
\eeas
Once $\alpha$ is fixed, we can fix $k$, and the same way of estimate for $\alpha=0$ 
still works for non-zero $\alpha$
(without increasing differentiability $\ell$). 

\im Concerning possible assumptions about 
an expansion of each $E\Gamma^{(j)}_{i_1,...,i_j}(F_N)$. 
Let us consider asymptotic expansion of $S_N+S_N'$, a sum of independent scaled variables 
each of which tends to a normal distribution but with different convergence rates. 
A mixture of powers of the different rates should appear 
in the asymptotic expansion of $S_N+S_N'$. 
Thus, a possible assumption (in the second stage, after writing a general asymptotic expansion formula without it) is 
\begin{equation}
E \Gamma ^{(j)} _{i_{1},.., i_{j} } (F_{N}) 
=
\frac{c_{j,1} ^{(i_{1},.., i_{j})} }{N^{\gamma^{(i_{1},.., i_{j})}_{j,1}}}
+\frac{c_{j,2} ^{(i_{1},.., i_{j})} }{N^{\gamma^{(i_{1},.., i_{j})}_{j,2}}}
+....
+\frac{c_{j,p} ^{(i_{1},.., i_{j})} }{N^{\gamma^{(i_{1},.., i_{j})}_{j,p}}}
+ o \left( \frac{1}{N^{\gamma^{(i_{1},.., i_{j})}_{j,p}}}\right)
\end{equation}
with some collection of real numbers $c_{j, i}=(c_{j, i}^{(i_1,...,i_j)})_{i_1,...,i_j}$, $i=1,..,p$. 
Remark that $p$, the length of the expansion of the expectation, 
can depend on $ \Gamma ^{(j)} _{i_{1},.., i_{j} } (F_{N})$. %$j$ and $(i_1,...,i_j)$. 
\end{itemize}

Let 
\beas 
A^\Psi_N(\theta,\blambda)
&=&
\varphi _{F_{N}} ^{\Psi} (0, \boldsymbol{\lambda})
\bigg\{
\sfP_N(\theta,\blambda)
\\&&
+
\sum_{j=1}^\sfk 
\int_0^\theta\int_0^{\theta_1}\cdots\int_0^{\theta_{\sfk-1}}
d\theta_j\cdots d\theta_2d\theta_1
\sfP_N(\theta_j,\blambda)\cdots \sfP_N(\theta_1,\blambda)\sfP_N(\theta,\blambda)\bigg\}.
\eeas

\noindent End of Memo. \\
------------------------------------------
}
\end{en-text}

\section{Reduced formulas}

{ If the cumulants $\mathbf{E} \left[\Gamma ^{(p)} (F_{N})\right]$, $p\geq 2$ admit a specific Taylor decomposition, we obtain a more explicite asymptotic expansion for the sequence $(F_{N}) _{N\geq 1}$. }

\subsection{Principal part of $f_{N,p,\sfk}$}
The asymptotic expansion formula is given by (\ref{20181120-1}) and (\ref{20181120-2}). 
However, it involves terms that is higher than $N^{-\sfq}$ in general. 
If the coefficients in $\sfP_N(\theta,\blambda)$ admit a specific expansion, then 
we can extract the principal part of $f_{N,p,\sfk}$. 

Let $\bbI=\{1,...,d\}$. 
For simplicity of notation, we will denote by $I_j$ a generic element $(i_1,...,i_j)$ of $\bbI^j$. 
The summation $\sum_{I_j}$ stands for $\sum_{(i_1,...,i_j)\in\bbI^j}$. 
For $j\in\{2,...,p+1\}$, suppose that a nonnegative integer $k(I_j)$ is given for each $I_j\in\bbI^j$. 
\halflineskip

{\colorg 
The below assumption gives the concrete Taylor expansion of the cumulants of $F_{N}$ in terms of power of $N$. }
\bd

\im[[B\!\!]] 
For each $j\in\{2,...,p+1\}$ and $I_j\in\bbI^j$, if $k(I_j)\geq1$, then for $k\in\{1,...,k(I_j)\}$, 
there exist sequences of real numbers $(c(I_j,k))_{k=1,...,k(I_j)}$ and 
$(\gamma(I_j,k))_{k=1,...,k(I_j)}$ such that the following conditions hold. 
\bd
\im[(i)] $0<\gamma(I_j,1)<\cdots<\gamma(I_j,k(I_j))\leq \sfq$ %and $c(I_j,1)\not=0$ 
(when $k(I_j)\geq1$). 
%unless $k(I_m)=0$. 
\im[(ii)] For $I_2\in\bbI^2$, 
\beas 
\E\big[\Gamma^{(2{ sym})}_{I_2}(F_N)\big]-C_{I_2}
&=& 
\sum_{k=1}^{k(I_2)}c(I_2,k)N^{-\gamma(I_2,k)}+o(N^{-\sfq}),
\eeas
where the sum $\sum_{k=1}^{k(I_2)}$ reads $0$ when $k(I_2)=0$.
\im[(iii)] For $j\in\{3,...,p+1\}$ and $I_j\in\bbI^j$, 
\beas 
\E\big[\Gamma^{(j)}_{I_j}(F_N)\big]
&=& 
\sum_{k=1}^{k(I_j)}c(I_j,k)N^{-\gamma(I_j,k)}+o(N^{-\sfq}),
\eeas
where the sum $\sum_{k=1}^{k(I_j)}$ reads $0$ when $k(I_j)=0$.
\im[(iv)] 
%There exist positive numbers $\sfq_0$ and $\xi$ such that 
The numbers $\sfq_0\in(0,\infty)$, $\xi\in(0,\infty)$, $\ell\in\bbN$ and $\ell_1\in\bbN$ 
satisfy 
\bea\label{201901030402}
\sfq_0(\sfk+1)>\sfq,%\quad\frac{\sfq}{\ell-d}\><\>\xi\><\>\frac{\ell_1-d}{p+1}
{\quad\xi(\ell-d)>\sfq,\quad \ell\geq\ell_1>p+1+d}
\eea
and 
\bea\label{201901030403} 
\sfq_0 &\leq& \min\bigg\{\gamma(I_j,1)-j\xi;\>I_j\in\bbI^j,\>j=2,...,p+1\bigg\}
\eea
with $\gamma(I_j,1)=\infty$ when $k(I_j)=0$. 
\ed
\ed
\halflineskip
%We set $\gamma(I_m,1)=\infty$ when $k(I_m)=0$. 
%When $c(I_m,1)=0$, it means there is no principal part and then we set $\gamma(I_m,1)=\infty$. 

{ Assumption $ [B] (ii)-(iii) $  indicates that the random vector $(F_{N})_{N\geq 1} $ converges in distribution to a centered Gaussian vector with covariance matrix $C$. }

\begin{remark}\rm
In many cases, $\gamma(I_j,k)$ is a multiple of a constant such as $1/2$. 
However, it is not always true. 
For example, the asymptotic expansion formula for $F_N=S^{(1)}_N+S^{(2)}_N$ 
has two scales $N^{-1/2}$ and ${\cred \lceil N^{\pi}\rceil^{-1/2}}$ %$N^{-\pi/2}$ 
and their mixtures 
when 
$S^{(1)}_N=N^{-1/2}\sum_{j=1}^N\big((\xi_j^{(1)})^2-1\big)$ and 
$S^{(2)}_N={\cred{\lceil N^{\pi}\rceil}^{-1/2}}\sum_{j=1}^{\lceil N^{\pi}\rceil}\big((\xi_j^{(2)})^2-1\big)$, 
where $\big\{\xi_j^{(1)},\xi_j^{(2)};j\in\bbN\big\}$ are independent standard Gaussian random variables.  
\end{remark}

We write $\blambda_{I_m}=\lambda_{i_1}\cdots\lambda_{i_m}$ for 
$\blambda=(\lambda_1,...,\lambda_d)$ and $I_m=(i_1,...,i_m)$. 
Under $[B]$, $\int_0^1\sfP_N(\theta,\blambda)d\theta$ is given by 
\beas 
\int_0^1\sfP_N(\theta,\blambda)d\theta
&=&
\sum_{m=2}^{p+1}\sum_{I_m\in\bbI^m}\sum_{k=1}^{k(I_m)}\blambda_{I_m}
\bigg\{\frac{{\tt i}^mc(I_m,k)}{m}
N^{-\gamma(I_m,k)}+o(N^{-\sfq})\bigg\}
\eeas
Therefore, $\sfP^*_N(\blambda)$ is expressed as 
\beas 
\sfP^*_N(\blambda)
&=&
1+
\sum_{j=1}^\sfk
\sum_{m_1=2}^{p+1}\cdots\sum_{m_j=2}^{p+1}
\sum_{I_{m_1}^{(1)}\in\bbI^{m_1}}\cdots\sum_{I_{m_j}^{(j)}\in\bbI^{m_j}}
\sum_{k_1=1}^{k(I_{m_1}^{(1)})}\cdots\sum_{k_j=1}^{k(I_{m_j}^{(j)})}
\blambda_{I_{m_1}^{(1)}}\cdots\blambda_{I_{m_j}^{(j)}}
\\&&
\times\bigg\{\frac{{\tt i}^{m_1+\cdots+m_j}}{j!m_1\cdots m_j}
c(I_{m_1}^{(1)},k_1)\cdots c(I_{m_j}^{(j)},k_j)
1_{\{\gamma(I_{m_1}^{(1)},k_1)+\cdots+\gamma(I_{m_j}^{(j)},k_j)\leq\sfq\}}
\\&&
\qquad\times N^{-\{\gamma(I_{m_1}^{(1)},k_1)+\cdots+\gamma(I_{m_j}^{(j)},k_j)\}}
+\ep_N{\big(j;m_1,...,m_j;I^{(1)}_{m_1},...,I^{(j)}_{m_j};k_1,...,k_j\big)}
%o(N^{-\sfq})%-(m_1+\cdots+m_j)\xi})
\bigg\},
\eeas
where $\ep_N{\big(j;m_1,...,m_j;I^{(1)}_{m_1},...,I^{(j)}_{m_j};k_1,...,k_j\big)}=o(N^{-\sfq})$ independent of $\blambda$. 
%The last term of $o(N^{-\sfq})$ on the right-hand side does not depend on $\blambda$. 

{\colorg Given the Taylor expansion of the cumulants from Condition  $[B]$ and the expression of the principal part $P_{N}(\theta, \blambda ) $ in (\ref{20181117-1}), we will replace $P^*_N(\blambda)$ by a new functional written in terms of powers of $N$.}  Define $\widetilde{\sfP}_N(\blambda)$ by 
\bea\label{201901020544}
\widetilde{\sfP}_N(\blambda)
&=& 
1+\sum_{j=1}^\sfk
\sum_{m_1=2}^{p+1}\cdots\sum_{m_j=2}^{p+1}
\sum_{I_{m_1}^{(1)}\in\bbI^{m_1}}\cdots\sum_{I_{m_j}^{(j)}\in\bbI^{m_j}}
\sum_{k_1=1}^{k(I_{m_1}^{(1)})}\cdots\sum_{k_j=1}^{k(I_{m_j}^{(j)})}
\blambda_{I_{m_1}^{(1)}}\cdots\blambda_{I_{m_j}^{(j)}}
\nn\\&&
\times\frac{{\tt i}^{m_1+\cdots+m_j}}{j!m_1\cdots m_j}
c(I_{m_1}^{(1)},k_1)\cdots c(I_{m_j}^{(j)},k_j)
1_{\{\gamma(I_{m_1}^{(1)},k_1)+\cdots+\gamma(I_{m_j}^{(j)},k_j)\leq\sfq\}}
\nn\\&&
\times
N^{-\{\gamma(I_{m_1}^{(1)},k_1)+\cdots+\gamma(I_{m_j}^{(j)},k_j)\}}
.
\eea
Let
\beas%\label{20181120-11} 
%\label{pn}
\widetilde{f}_{N, p,\sfk} (x) &=& \frac{1}{(2\pi) ^{d} } 
\int_{\mathbb{R}^{ d}} e^{-{\tt i} \langle \boldsymbol{\lambda},  x\rangle} \widetilde{\varphi} _{N,p,\sfk} (\boldsymbol{\lambda} ) d\boldsymbol{\lambda} \qquad(x\in \mathbb{R}^{d})
\eeas
where $\widetilde{\varphi} _{N,p,\sfk}$ is given by 
\beas%\label{20181120-12} 
\widetilde{\varphi} _{N,p,\sfk}(\blambda)
&=&
e^{-\half\blambda^T C\blambda}\>\widetilde{\sfP}_N(\blambda).
\eeas
Then
\bea\label{201901020543} 
\widetilde{f}_{N, p,\sfk} (x) 
&=&
\widetilde{\sfP}_N({\tt i}\partial_x)\phi(x;0,C), 
\eea
where 
\beas
\phi(x;0,C) &=& 
(2\pi) ^{-d/2}(\det C)^{-1/2}e^{-\half x^T C^{-1}x}.
\eeas

Define the $\alpha$-th Hermite polynomial $H_\alpha(x;C)$ by 
\beas 
H_\alpha(x;C) &=& e^{x^TC^{-1}x/2}\big(-\partial_x\big)^\alpha e^{-x^TC^{-1}x/2}
\qquad(x\in\bbR^d)
\eeas
for $\alpha\in\bbZ_+^d$. 
Define the multi-index $\alpha(I_{m_1}^{(1)},...,I_{m_j}^{(j)})\in\bbZ_+^d$ by 
\beas 
\blambda_{I_{m_1}^{(1)}}\cdots\blambda_{I_{m_j}^{(j)}}
&=&
\blambda^{\alpha(I_{m_1}^{(1)},...,I_{m_j}^{(j)})}
\qquad(\blambda\in\bbR^d). 
\eeas
That is, the $i$-th component of $\alpha(I_{m_1}^{(1)},...,I_{m_j}^{(j)})$ is 
the number of $i$'s appearing in the sequence $I_{m_1}^{(1)},...,I_{m_j}^{(j)}$. 
Then the density function $\wt{f}_{N,p,\sfk}$ is expressed as 
\bea\label{201901040612}
\wt{f}_{N,p,\sfk}(x) 
&=& 
\phi(x;0,C)
\nn\\&&
+\sum_{j=1}^\sfk
\sum_{m_1=2}^{p+1}\cdots\sum_{m_j=2}^{p+1}
\sum_{I_{m_1}^{(1)}\in\bbI^{m_1}}\cdots\sum_{I_{m_j}^{(j)}\in\bbI^{m_j}}
\sum_{k_1=1}^{k(I_{m_1}^{(1)})}\cdots\sum_{k_j=1}^{k(I_{m_j}^{(j)})}
\bigg\{\frac{1}{j!m_1\cdots m_j}
\nn\\&&
\quad\times
c(I_{m_1}^{(1)},k_1)\cdots c(I_{m_j}^{(j)},k_j)\>
H_{\alpha(I_{m_1}^{(1)},...,I_{m_j}^{(j)})}(x;C)\>
\phi(x;0,C)
\nn\\&&
\quad\times1_{\{\gamma(I_{m_1}^{(1)},k_1)+\cdots+\gamma(I_{m_j}^{(j)},k_j)\leq\sfq\}}
N^{-\{\gamma(I_{m_1}^{(1)},k_1)+\cdots+\gamma(I_{m_j}^{(j)},k_j)\}}\bigg\}
.
\eea

%The following theorem rephrases Theorem \ref{20181120-21}. 
The following theorem validates $\wt{f}_{N,p,\sfk}$ as  
a reduced asymptotic expansion formula. 
\begin{theorem}\label{201901021106}
Suppose that Conditions $[A1]$, $[A2]$ and $[B]$ are fulfilled. 
%Moreover, Conditions  are satisfied 
%for some $\sfq_0$ %=\min\{\gamma(I_m,1);\>I_m\in\bbI^m,\>m=1,...p+1\}$. 
%satisfying 
%\beas 
%\sfq_0 &\leq& \min\bigg\{\gamma(I_j,1)-j\xi;\>I_j\in\bbI^j,\>j=2,...,p+1\bigg\}.
%\eeas
Then 
\bea\label{201901040033}
\sup_{g\in\cale(a,b)}\bigg|\E\big[g(F_N)\big]-\int_{\bbR^d}g(x)\widetilde{f}_{N,p,\sfk}(x)dx\bigg|
&=&
o(N^{-\sfq})
\eea
as $N\to\infty$ for every $a,b>0$. 
\end{theorem}
\proof 
We have 
\beas &&
\sup_{x\in\bbR^d}\bigg(|x^\alpha|\big|f_{N, p,\sfk} (x)-\wt{f}_{N, p,\sfk} (x)\big|\bigg)
\\&=&
{ o(N^{-\sfq})}%|\ep_N|
\sup_{x\in\bbR^d}\bigg|
\frac{1}{(2\pi)^d}\int_{\bbR^d}
\big(({\tt i}\partial_\blambda)^\alpha e^{-{\tt i}\langle\blambda,x\rangle}\big)
\bigg\{
e^{-\half\blambda^TC\blambda}
\sum_{j=1}^\sfk
\sum_{m_1=2}^{p+1}\cdots
\\&&\qquad\qquad\qquad\cdots
\sum_{m_j=2}^{p+1}
\sum_{I_{m_1}^{(1)}\in\bbI^{m_1}}\cdots\sum_{I_{m_j}^{(j)}\in\bbI^{m_j}}
\sum_{k_1=1}^{k(I_{m_1}^{(1)})}\cdots
\sum_{k_j=1}^{k(I_{m_j}^{(j)})}
\blambda_{I_{m_1}^{(1)}}\cdots\blambda_{I_{m_j}^{(j)}}\bigg\}d\blambda\bigg|
\\&=&
o(N^{-\sfq})
\eeas
as $N\to\infty$ for every $\alpha\in\bbZ_+^d$. 
Here the last equality follows from integrtation-by-parts and that $\ep_N=o(N^{-\sfq})$. 
Therefore, 
\beas 
\sup_{g\in\cale(a,b)}\bigg|
\int_{\bbR^d}g(x)f_{N,p,\sfk}(x)dx
-\int_{\bbR^d}g(x)\wt{f}_{N,p,\sfk}(x)dx
\bigg|
&=&
o(N^{-\sfq})
\eeas
for any $a,b>0$. 
Now Theorem \ref{201901021106} follows from Theorem \ref{20181120-21}. 
\qed
\halflineskip

\subsection{Regular ordering}\label{201908230639}%{Case $\gamma=1/2$ }
In this section, we will consider the situation where 
the exponents $\gamma(I_j,k)$ are multiples of some positive number $\gamma$. 
Suppose that $p\geq2$. 
{ We consider the following situation.}
\bd
\im[[C\!\!]] 
{\bf (i)} For each $I_2\in\bbI^2$, 
\beas 
\E\big[\Gamma^{(2{ sym})}_{I_2}(F_N)\big]-C_{I_2}
&=& 
\sum_{k=1}^{p-1}c(I_2,k)N^{-k\gamma}+o(N^{-(p-1)\gamma})
\eeas
%where the sum $\sum_{k=1}^{k(I_2)}$ reads $0$ when $k(I_2)=0$.
as $N\to\infty$ 
for some constants $c(I_2,k)$ ($k=1,...,p-1$). 
\bd
\im[(ii)] For each $j\in\{3,...,p+1\}$ and $I_j\in\bbI^j$, 
\beas 
\E\big[\Gamma^{(j)}_{I_j}(F_N)\big]
&=& 
\sum_{k=1}^{p-j+2}c(I_j,k)N^{-(j-3+k)\gamma}+o(N^{-(p-1)\gamma})
\eeas
as $N\to\infty$ 
for some constants $c(I_j,k)$ ($k=1,...,p-j+2$). 
\ed
\ed
\halflineskip\halflineskip

Given an integer $p\geq2$ and the dimension $d$ of $F_N$, 
we suppose that the positive integers $\ell$ and $\ell_1$ satisfy
\bea\label{201901040029} 
\quad\ell>3(p-1)+d,\quad \ell\geq\ell_1>p+1+d. 
\eea
For example, the inequalities (\ref{201901040029}) hold if $\ell=\ell_1>3(p-1)+d$, when $p\geq2$. 
\begin{en-text}
in any one of the following cases:
\bd\im[(a)] $p=1$, $\ell=\ell_1>p+1+d$.
\im[(b)] $p\geq2$, $\ell=\ell_1.3(p-1)+d$. 
\ed
\end{en-text}

Suppose that Condition $[C]$ and (\ref{201901040029}) are fulfilled and that 
a positive number $\xi$ and an integer $\sfk$ satisfy 
\beas 
\xi\in\bigg(\frac{p-1}{\ell-d}\gamma,\>\frac{1}{3}\gamma\bigg)
\eeas
and %choose an integer $\sfk$ such that 
\bea\label{201901040244} 
\sfk &>& \frac{(p-2)\gamma+3\xi}{\gamma-3\xi}
\yeq \frac{(p-1)\gamma}{\gamma-3\xi}-1.
\eea
Such numbers $\xi$ and $\sfk$ exist under (\ref{201901040029}). 
Let $\sfq=(p-1)\gamma$ and let 
$\sfq_0 \yeq \gamma-3\xi$ $(>0)$. 
% 0=(j-2)-j/3=2(j-3)/3
%
Let 
\beas 
\gamma(I_2,k)&=& k\gamma\qquad(k=1,...,p-1)
\eeas
and 
\beas 
\gamma(I_j,k)&=&(j-3+k)\gamma\qquad(k=1,...,p-j+2)
\eeas 
for $j=3,...,p+1$. 
Then the inequalities in (\ref{201901030402}) and (\ref{201901030403}) are met, 
and hence Condition [B] holds for $k(I_2)=p-1$ and $k(I_j)=p-j+2$ for $j=3,...,p+1$. 
\begin{en-text}
\beas 
\sfq_0(\sfk+1)>\sfq,%\quad\frac{\sfq}{\ell-d}\><\>\xi\><\>\frac{\ell_1-d}{p+1}
{\quad\xi(\ell-d)>\sfq,\quad \ell_1>p+1+d}
\eeas
and 
\beas 
\sfq_0 &\leq& \min\bigg\{\gamma(I_j,1)-j\xi;\>I_j\in\bbI^j,\>j=2,...,p+1\bigg\}
\eeas
are satisfied. 
\end{en-text}

We can write $\gamma(I_j,k)=((j-3)_++k)\gamma$ for $j\in\{2,...,p+1\}$, $x_+=\max\{x,0\}$. 
Then the symbol $\widetilde{\sfP}_N(\blambda)$ takes the form of 
\bea\label{201901031100}
\widetilde{\sfP}_N(\blambda)
&=& 
1+\sum_{j=1}^{p-1}%\sfk
\sum_{m_1=2}^{p+1}\cdots\sum_{m_j=2}^{p+1}
\sum_{I_{m_1}^{(1)}\in\bbI^{m_1}}\cdots\sum_{I_{m_j}^{(j)}\in\bbI^{m_j}}
\sum_{k_1=1}^{{ p-1-(m_1-3)_+}}%{p-m_1+2}
\cdots\sum_{k_j=1}^{{ p-1-(m_j-3)_+}}%{p-m_j+2}
\blambda_{I_{m_1}^{(1)}}\cdots\blambda_{I_{m_j}^{(j)}}
\nn\\&&
\times\frac{{\tt i}^{m_1+\cdots+m_j}}{j!m_1\cdots m_j}
c(I_{m_1}^{(1)},k_1)\cdots c(I_{m_j}^{(j)},k_j)
1_{\{\sum_{i=1}^j(m_{i}-3)_++\sum_{i=1}^jk_{i}\leq p-1\}}
\nn\\&&
\times
N^{-\{\sum_{i=1}^j(m_{i}-3)_++\sum_{i=1}^jk_{i}\}\gamma}.
\eea
We remark that the first summation on the right-hand side of (\ref{201901031100}) 
has become $\sum_{j=1}^{p-1}$  
though it was originally $\sum_{j=1}^\sfk$, by the following reason. 
The condition (\ref{201901040244}) entails $\sfk\geq p-1$, however, 
for $j\geq p$, the summands vanish due to the indicator function. Therefore 
only the terms for $j$ up to $p-1$ can contribute. 
According to (\ref{201901031100}), the density $\wt{f}_{N,p,\sfk}$ has the expression 
\bea\label{201901040614}
\wt{f}_{N,p,\sfk}(x) 
&=& 
\phi(x;0,C)
\nn\\&&
+\sum_{j=1}^{p-1}
\sum_{m_1=2}^{p+1}\cdots\sum_{m_j=2}^{p+1}
\sum_{I_{m_1}^{(1)}\in\bbI^{m_1}}\cdots\sum_{I_{m_j}^{(j)}\in\bbI^{m_j}}
\sum_{k_1=1}^{{ p-1-(m_1-3)_+}}%{p-m_1+2}
\cdots\sum_{k_j=1}^{{ p-1-(m_j-3)_+}}%{p-m_j+2}
\bigg\{\frac{1}{j!m_1\cdots m_j}
\nn\\&&
\quad\times
c(I_{m_1}^{(1)},k_1)\cdots c(I_{m_j}^{(j)},k_j)\>
H_{\alpha(I_{m_1}^{(1)},...,I_{m_j}^{(j)})}(x;C)\>
\phi(x;0,C)
\nn\\&&
\quad\times1_{\{\sum_{i=1}^j(m_{i}-3)_++\sum_{i=1}^jk_{i}\leq p-1\}}
N^{-\{\sum_{i=1}^j(m_{i}-3)_++\sum_{i=1}^jk_{i}\}\gamma}\bigg\}
.
\eea

Applying Theorem \ref{201901021106}, we obtain an asymptotic expansion formula 
when the gamma factors have a regularly ordered expansion. 
\begin{theorem}\label{201901040035}
Assume $[A1]$ and $[A2]$ for some pair $(\ell,\ell_1)$ of integers 
satisfying (\ref{201901040029}) for given integers $p\geq2$ and { the dimension} $d$ { of $F_N$}. 
Moreover assume $[C]$. 
Then 
\bea\label{201901040631}
\sup_{g\in\cale(a,b)}\bigg|\E\big[g(F_N)\big]-\int_{\bbR^d}g(x)\widetilde{f}_{N,p,\sfk}(x)dx\bigg|
&=&
o(N^{-(p-1)\gamma})
\eea
as $N\to\infty$ for every $a,b>0$ 
for $\widetilde{f}_{N, p,\sfk}$ of (\ref{201901040614}). 
%expressed by (\ref{201901020543}) with $\widetilde{\sfP}_N(\blambda)$ of (\ref{201901031100}).  
\end{theorem}
\halflineskip

In the rest of this section, we will state several special cases 
of the asymptotic expansion in terms of the symbol $\widetilde{\sfP}_N(\blambda)$ 
of (\ref{201901031100}). 
When $p=2$, Formula (\ref{201901031100}) gives 
\bea\label{201901040301}
\widetilde{\sfP}_N(\blambda)
&=& 
%1+\sum_{I_{2}^{(1)}\in\bbI^{2}}\half{\tt i}^2\blambda_{I_{2}^{(1)}}c(I^{(1)}_2,1)
%N^{-\gamma}
%+\sum_{I_{3}^{(1)}\in\bbI^{3}}\frac{1}{3}{\tt i}^3\blambda_{I_{3}^{(1)}}c(I^{(1)}_3,1)
%N^{-\gamma}
%\nn\\&=& 
1+
\left\{
\sum_{I_{2}^{(1)}\in\bbI^{2}}\half{\tt i}^2\blambda_{I_{2}^{(1)}}c(I^{(1)}_2,1)
+\sum_{I_{3}^{(1)}\in\bbI^{3}}\frac{1}{3}{\tt i}^3\blambda_{I_{3}^{(1)}}c(I^{(1)}_3,1)
\right\}N^{-\gamma}. 
\eea
When $p=3$, Formula (\ref{201901031100}) gives 
\begin{en-text}
\beas
\widetilde{\sfP}_N(\blambda)
&=& 
1+\sum_{I_{2}^{(1)}\in\bbI^{2}}\sum_{k_1=1}^{2}\half{\tt i}^2\blambda_{I_{2}^{(1)}}c(I^{(1)}_2,k_1)
N^{-k_1\gamma}
\\&&%%
+\sum_{I_{3}^{(1)}\in\bbI^{3}}\sum_{k_1=1}^{2}\frac{1}{3}{\tt i}^3\blambda_{I_{3}^{(1)}}c(I^{(1)}_3,k_1)
N^{-k_1\gamma}
+%(j=1,m_1=4,k_1=1)
\sum_{I_{4}^{(1)}\in\bbI^{4}}\frac{1}{4}{\tt i}^4\blambda_{I_{4}^{(1)}}c(I^{(1)}_4,1)N^{-2\gamma}
\\ &&
+%(j=2,m_1=2,m_2=2, k_1=k_2=1)
\sum_{I_{2}^{(1)}\in\bbI^{2}}\sum_{I_{2}^{(2)}\in\bbI^{2}}\frac{1}{2\cdot2\cdot2}{\tt i}^4
\blambda_{I_{2}^{(1)}}\blambda_{I_{2}^{(2)}}c(I^{(1)}_2,1)c(I^{(2)}_2,1)N^{-2\gamma}
\\&&
+%2(j=2,m_1=2,m_2=3, k_1=k_2=1)
2\sum_{I_{2}^{(1)}\in\bbI^{2}}\sum_{I_{3}^{(2)}\in\bbI^{3}}
\frac{1}{2\cdot2\cdot3}{\tt i}^5\blambda_{I_{2}^{(1)}}\blambda_{I_{3}^{(2)}}c(I^{(1)}_2,1)c(I^{(2)}_3,1)N^{-2\gamma}
\\&&
+%(j=2,m_1=3,m_2=3, k_1=k_2=1)
\sum_{I_{3}^{(1)}\in\bbI^{2}}\sum_{I_{3}^{(2)}\in\bbI^{3}}
\frac{1}{2\cdot3\cdot3}{\tt i}^6\blambda_{I_{3}^{(1)}}\blambda_{I_{3}^{(2)}}c(I^{(1)}_3,1)c(I^{(2)}_3,1)N^{-2\gamma}.
\eeas
Therefore, 
\end{en-text}
\bea\label{201901040309}
\widetilde{\sfP}_N(\blambda)
&=& 
1+
\left\{
\sum_{I_{2}^{(1)}\in\bbI^{2}}\half{\tt i}^2\blambda_{I_{2}^{(1)}}c(I^{(1)}_2,1)
+\sum_{I_{3}^{(1)}\in\bbI^{3}}\frac{1}{3}{\tt i}^3\blambda_{I_{3}^{(1)}}c(I^{(1)}_3,1)
\right\}N^{-\gamma}
\nn\\&&
+\left\{
\sum_{I_{2}^{(1)}\in\bbI^{2}}\half{\tt i}^2\blambda_{I_{2}^{(1)}}c(I^{(1)}_2,2)
+\sum_{I_{3}^{(1)}\in\bbI^{3}}\frac{1}{3}{\tt i}^3\blambda_{I_{3}^{(1)}}c(I^{(1)}_3,2)
\right.
\nn\\ &&
\qquad+%(j=1,m_1=4,k_1=1)
\sum_{I_{4}^{(1)}\in\bbI^{4}}\frac{1}{4}{\tt i}^4\blambda_{I_{4}^{(1)}}c(I^{(1)}_4,1)
\nn\\&&
\qquad+%(j=2,m_1=2,m_2=2, k_1=k_2=1)
\sum_{I_{2}^{(1)}\in\bbI^{2}}\sum_{I_{2}^{(2)}\in\bbI^{2}}\frac{1}{8}{\tt i}^4
\blambda_{I_{2}^{(1)}}\blambda_{I_{2}^{(2)}}c(I^{(1)}_2,1)c(I^{(2)}_2,1)
\nn\\&&
\qquad+%2(j=2,m_1=2,m_2=3, k_1=k_2=1)
\sum_{I_{2}^{(1)}\in\bbI^{2}}\sum_{I_{3}^{(2)}\in\bbI^{3}}
\frac{1}{6}{\tt i}^5\blambda_{I_{2}^{(1)}}\blambda_{I_{3}^{(2)}}c(I^{(1)}_2,1)c(I^{(2)}_3,1)
\nn\\&&
\qquad+%(j=2,m_1=3,m_2=3, k_1=k_2=1)
\left.
\sum_{I_{3}^{(1)}\in\bbI^{2}}\sum_{I_{3}^{(2)}\in\bbI^{3}}
\frac{1}{18}{\tt i}^6\blambda_{I_{3}^{(1)}}\blambda_{I_{3}^{(2)}}c(I^{(1)}_3,1)c(I^{(2)}_3,1)\right\}N^{-2\gamma}.
\eea

In particular, when $p=2$, Formula (\ref{201901040301}) is reduced to 
\beas
\widetilde{\sfP}_N(\blambda)
&=& 
1+
\sum_{I_{3}^{(1)}\in\bbI^{3}}\frac{1}{3}{\tt i}^3\blambda_{I_{3}^{(1)}}c(I^{(1)}_3,1)
N^{-\gamma}
\eeas
if $c(I_2,1)=0$ for all $I_2\in\bbI^2$. 
When $p=3$, Formula (\ref{201901040309}) is reduced to 
\beas
\widetilde{\sfP}_N(\blambda)
&=& 
1+
\sum_{I_{3}^{(1)}\in\bbI^{3}}\frac{1}{3}{\tt i}^3\blambda_{I_{3}^{(1)}}c(I^{(1)}_3,1)
N^{-\gamma}
\\&&
+\left\{
\sum_{I_{2}^{(1)}\in\bbI^{2}}\half{\tt i}^2\blambda_{I_{2}^{(1)}}c(I^{(1)}_2,2)
\right.
+%(j=1,m_1=4,k_1=1)
\sum_{I_{4}^{(1)}\in\bbI^{4}}\frac{1}{4}{\tt i}^4\blambda_{I_{4}^{(1)}}c(I^{(1)}_4,1)
\\&&
\qquad+%(j=2,m_1=3,m_2=3, k_1=k_2=1)
\left.
\sum_{I_{3}^{(1)}\in\bbI^{2}}\sum_{I_{3}^{(2)}\in\bbI^{3}}
\frac{1}{18}{\tt i}^6\blambda_{I_{3}^{(1)}}\blambda_{I_{3}^{(2)}}c(I^{(1)}_3,1)c(I^{(2)}_3,1)\right\}N^{-2\gamma}
\eeas
if $c(I_2,1)=0$ for all $I_2\in\bbI^2$ and if $c(I_3,2)=0$ for all $I_3\in\bbI^3$. 
If additionally $c(I_2,3)=0$ for all $I_2\in\bbI^2$ and if $c(I_4,1)=0$ for all $I_4\in\bbI^4$, 
then for $p=4$, we obtain
\beas
\widetilde{\sfP}_N(\blambda)
&=& 
1+
\sum_{I_{3}^{(1)}\in\bbI^{3}}\frac{1}{3}{\tt i}^3\blambda_{I_{3}^{(1)}}c(I^{(1)}_3,1)
N^{-\gamma}
\\&&
+\left\{
\sum_{I_{2}^{(1)}\in\bbI^{2}}\half{\tt i}^2\blambda_{I_{2}^{(1)}}c(I^{(1)}_2,2)
\right.
+%(j=1,m_1=4,k_1=1)
\sum_{I_{4}^{(1)}\in\bbI^{4}}\frac{1}{4}{\tt i}^4\blambda_{I_{4}^{(1)}}c(I^{(1)}_4,1)
\\&&
\qquad+%(j=2,m_1=3,m_2=3, k_1=k_2=1)
\left.
\sum_{I_{3}^{(1)}\in\bbI^{3}}\sum_{I_{3}^{(2)}\in\bbI^{3}}
\frac{1}{18}{\tt i}^6\blambda_{I_{3}^{(1)}}\blambda_{I_{3}^{(2)}}c(I^{(1)}_3,1)c(I^{(2)}_3,1)\right\}N^{-2\gamma}
\\&&
+\left\{%(j=1,m_1=3,k_1=3)
\sum_{I_{3}^{(1)}\in\bbI^{3}}\frac{1}{3}{\tt i}^3\blambda_{I_{3}^{(1)}}c(I^{(1)}_3,3)
\right.
+%(j=1,m_1=5,k_1=1)
\sum_{I_{5}^{(1)}\in\bbI^{5}}\frac{1}{5}{\tt i}^5\blambda_{I_{5}^{(1)}}c(I^{(1)}_5,1)
\\&&
\qquad+%2(j=2,m_1=2,m_2=3,k_1=2,k_2=1)
\sum_{I_{2}^{(1)}\in\bbI^{2}}\sum_{I_{3}^{(2)}\in\bbI^{3}}
\frac{1}{6}{\tt i}^5\blambda_{I_{2}^{(1)}}\blambda_{I_{3}^{(2)}}c(I^{(1)}_2,2)c(I^{(2)}_3,1)
\\&&
\qquad+%2(j=2,m_1=3,m_2=4,k_1=1,k_2=2)
\sum_{I_{3}^{(1)}\in\bbI^{3}}\sum_{I_{4}^{(2)}\in\bbI^{4}}
\frac{1}{12}{\tt i}^7\blambda_{I_{3}^{(1)}}\blambda_{I_{4}^{(2)}}c(I^{(1)}_3,1)c(I^{(2)}_4,2)
\\&&
\left.\qquad+%(j=3,m_1=3,m_2=3,m_3=3,k_1=1,k_2=1,k_3=1)
\sum_{I_{3}^{(1)}\in\bbI^{3}}\sum_{I_{3}^{(2)}\in\bbI^{3}}\sum_{I_{3}^{(3)}\in\bbI^{3}}
\frac{1}{162}{\tt i}^9\blambda_{I_{3}^{(1)}}\blambda_{I_{3}^{(2)}}\blambda_{I_{3}^{(3)}}
c(I^{(1)}_3,1)c(I^{(2)}_3,1)c(I^{(3)}_3,1)\right\}N^{-3\gamma}.
\eeas
}
{\colorg In the last situation above (which is the case of the example studied below), 
the {\cred first} order term comes from the leading term in the Taylor expansion of the third cumulant, the second  order term comes {\cred in} %is 
the sum of the leading terms in the expansion of the second and fourth cumulant, the high order terms being a mixture of terms in the Taylor expansion of all the cumulants. }
\begin{en-text}
\section{Functionals in finite Wiener chaoses}
\subsection{Assumptions}\label{assumptions}
Consider a sequence of centered  random variables $(F_{N}) _{N\geq 1}$ in $\mathbb{R} ^{d}$ of the form
\begin{equation}
\label{rv}
F_{N}= \left( F_{N} ^{(1)}, \ldots, F_{N} ^{(d)}\right)=\left( I_{q_{1}} (f_{N} ^ {(1)}), \ldots  I_{q_{d}} (f_{N} ^ {(d)})\right) 
\end{equation}
where $q_{1},.., q_{d} \geq 1$ are integer numbers, $I_{q}$ denote the multiple integral with respect to a Gaussian process  $(W(h), h\in H)$ and $f_{N} ^ {(i)}$, $i=1,..,d$  are symmetric functions in $ H ^ { \otimes  n}$. 
\end{en-text}

\section{Application to the wave equation}

In order to illustrate our theoretical results, we will consider an example of a random sequence related to the solution to the wave equation driven by a space-time white noise. More precisely, we will analyze the asymptotic behavior of the quadratic variation in space of this solution. We will first analyzed the asymptotic expansion of the spatial quadratic variations at a fixed time and then we will study the two-dimensional random vector whose components are the spatial quadratic variations at different times. We show  that the assumptions considered in the previous sections are satisfied in this case. 

Let us start by recalling some basic facts concerning the stochastic wave equation and its solution.

\subsection{The wave equation with space-time white noise}\label{section-wave_equation}

Our object of study is the solution to the following stochastic { partial} 
differential equation in dimension 1
\begin{equation}
\left\{
\begin{array}{rcl}\label{wave}
\frac{\partial^2 u}{\partial t^2}(t,x)&=&\Delta
u(t,x)+\dot W(t,x),\quad t>0,\;x \in \mathbb{R}  \\
\noalign{\vskip 2mm}
u(0, x)&=& 0, \quad x \in \mathbb{R}\\
\noalign{\vskip 2mm} \frac{\partial u}{\partial t}(0,x) &=& 0,\quad
x \in \mathbb{R}.
\end{array} \right.
\end{equation}
We denoted by  $\Delta$ the Laplacian on $\mathbb{R}$ and by 
$W=\{W_t(A);\;t \geq 0,\;A \in \cB_{b}(\mathbb{R} )\}$  a real
valued centered Gaussian field, over a given complete filtered
probability space $(\Omega,\mathfrak{F},(\mathfrak{F})_{t\geq
0},\mathbb{P})$  with
covariance:
\begin{equation}
\label{10s-1}
\mathbf{E}\Big[W_t(A)W_s(B)\Big]=(t\wedge s)\lambda (A
\cap B),\;\mbox{for every}\;A,B\in\mathfrak{B}_d({\colorg \mathbb{R}})
\end{equation}
where $\lambda$ is the %$d$
{ one}-dimensional Lebesgue measure and
$\mathfrak{B}_d({\colorg \mathbb{R}})$ is the set of the Borel-subsets of
${\colorg \mathbb{R}}$ with finite Lebesgue measure. This is usually
called "the space-time white noise".

The mild solution to (\ref{wave}) is a square-integrable process
$u=\{u(t,x);\;t\geq 0, x \in {\mathbb{R}}\}$ which is defined
by:
\begin{equation} \label{def-sol-wave} u(t,x)=\int_{0}^{t}
\int_{{\mathbb{R}}}G_1(t-s,x-y)W(\mathrm{d}s,\mathrm{d}y)
\end{equation}
where the Green kernel  $G_{1}$ is defined by 
\begin{equation}\label{g1}
G_1(t,x)=\frac{1}{2}\mathds{1}_{\{|x|<t\}}, \hskip0.5cm t>0, x\in \mathbb{R}.
\end{equation}

\subsection{Computing the covariance and the correlation}

Fix $t_{1}, t_{2} >0$ and $x,y \in \mathbb{R}$. We { need} a sharp evaluation of the correlation structure of the Gaussian process (\ref{def-sol-wave}).  We first calculate the quantity $ \mathbf{E} { \big[u(t_{1}, x) u(t_{2}, x)\big]} $ at differents times  $t_{1}\not= t_{2}$ and when $t_{1}= t_{2}$.

\begin{lemma}  Let $u$ be given by (\ref{def-sol-wave}). For every  $t_{1}, t_{2} >0$ and $x,y \in \mathbb{R}$, we have
\bea\label{201904211159}
\label{19s-1}\mathbf{E}\big[ u(t_{1}, x) u(t_{2}, y)\big]=\frac{1}{16}  1_{\{\vert t_{1}- t_{2} \vert \leq \vert y-x\vert < t_{1}+t_{2} \} }\left( t_{1}+ t_{2} -\vert x-y\vert \right) ^{2} 
+{\frac{1}{4}}  1_{\{\vert t_{1}- t_{2} \vert {>} \vert y-x\vert
{ \}}} (t_{1} \wedge t_{2} )^{2}.
\nn\\&&
\eea
In particular, for $t_{1}= t_{2}=t>0$, 
\begin{equation}\label{10i-1}
\mathbf{E}\big[ u(t, x) u(t, y)\big]= \frac{1}{16} 1_{\vert x-y\vert < 2t} (2t -\vert x-y\vert ) ^{2}.
\end{equation}
\end{lemma}
\begin{proof}By the  isometry of the Wiener integral and from (\ref{g1}), { we obtain}

\begin{eqnarray*}
  \mathbf{E}\big[ u(t_{1}, x) u(t_{2}, y)\big]
  &=&
  \frac{1}{4} \int_{0} ^{{ t_1\wedge t_2}}ds \int_{\mathbb{R}} dz  G_{1}(t_{1}-u, x-z) G_{1}(t_{2}-u, y-z)\\
&=&
\frac{1}{4} \int_{0} ^{t_{1} \wedge t_{2}} { ds}\int_{\mathbb{R}} dz 
1_{ \{ \vert x-z\vert \leq t_{1}-s\}} 1_{ \{\vert y-z\vert \leq t_{2}-s\}} \\
&=&
\frac{1}{4} \int_{0} ^{t_{1} \wedge t_{2}\wedge \frac{1}{2}(t_{1}+t_{2} -\vert y-x\vert )} { ds}
\int_{\mathbb{R}} dz 1_{ \{\vert x-z\vert \leq t_{1}-s\}} 1_{ \{\vert y-z\vert \leq t_{2}-s\}} \\
&&
+\frac{1}{4} \int_{0} ^{t_{1} \wedge t_{2}} { ds}
\int_{\mathbb{R}} dz  1_{ \{2s> t_{1}+ t_{2}-\vert y-x\vert\} }1_{ \{\vert x-z\vert \leq t_{1}-s\}} 1_{ \{\vert y-z\vert \leq t_{2}-s\}} \\
&=&
1_{\{ t_{1}+t_{2} >\vert y-x\vert\} }
\frac{1}{4} \int_{0} ^{t_{1} \wedge t_{2}\wedge \frac{1}{2}(t_{1}+t_{2} -\vert y-x\vert )} { ds}
\int_{\mathbb{R}} dz 1_{ \{\vert x-z\vert \leq t_{1}-s\}} 1_{ \{\vert y-z\vert \leq t_{2}-s\}} \\
&=&
1_{\{ t_{1}+t_{2} >\vert y-x\vert\} }
\frac{1}{4} \int_{0} ^{t_{1} \wedge t_{2}\wedge \frac{1}{2}(t_{1}+t_{2} -\vert y-x\vert )} { ds}
{\bigg(}\int_{ (x-t_{1}+s) \vee ( y-t_{2}+s)} ^{(x+t_{1}-s ) \vee (y+t_{2}-s)}{ dz}{\bigg)_+}.
\end{eqnarray*}
In order to find the integration domain for the integral $dz$, we will consider several situations. Assume $x\geq y$. 

{\bf If $t_{1} \geq t_{2}$  and $x-y\geq t_{1}-t_{2}$} then 

$$ y+ t_{2}-s\leq x+t_{1}-s \mbox{ and } x-t_{1}+s\geq y-t_{2}+s.$$
In this case, 
\begin{eqnarray*}
&&\mathbf{E}\big[ u(t_{1}, x) u(t_{2}, y)\big] 
\yeq
1_{\{ t_{1}+t_{2} >\vert y-x\vert\} }
\frac{1}{4} \int_{0} ^{t_{1} \wedge t_{2}\wedge \frac{1}{2}(t_{1}+t_{2} -\vert y-x\vert )} ds 
{\bigg(}\int_{x-t_{1}+s} ^{y+t_{2}-s} dz{\bigg)_+} \\
&=& 1_{\{ t_{1}+t_{2} >\vert y-x\vert\} }\frac{1}{4} \int_{0} ^{t_{1} \wedge t_{2}\wedge \frac{1}{2}(t_{1}+t_{2} -\vert y-x\vert )} ds (t_{1}+t_{2} -(x-y)-2s) \\
&=&1_{\{ t_{1}+t_{2} >\vert y-x\vert\} }\frac{1}{4}  \int_{0} ^{\frac{1}{2}(t_{1}+t_{2} -\vert y-x\vert )} ds(t_{1}+t_{2} -(x-y)-2s) \\
&=& 1_{\{ t_{1}+t_{2} >\vert y-x\vert\} }\frac{1}{16} \left( t_{1}+ t_{2} -(x-y) \right) ^{2}.
\end{eqnarray*}

{\bf If $t_{1} \geq t_{2}$  and $x-y < t_{1}-t_{2}$} then

$$ y+ t_{2}-s\leq x+t_{1}-s \mbox{ and } x-t_{1}+s\leq y-t_{2}+s.$$
Then 
\begin{eqnarray*}
\mathbf{E} { \big[}
u(t_{1}, x) u(t_{2}, y){ \big]}
&=& 1_{\{ t_{1}+t_{2} >\vert y-x\vert\} }\frac{1}{4} \int_{0} ^{t_{1} \wedge t_{2}\wedge \frac{1}{2}(t_{1}+t_{2} -\vert y-x\vert )} ds \int_{y-t_{2}+s} ^{y+t_{2}-s} dz \\
&=& 
1_{\{ t_{1}+t_{2} >\vert y-x\vert\} }\frac{1}{4} \int_{0} ^{t_{1} \wedge t_{2}\wedge \frac{1}{2}(t_{1}+t_{2} -\vert y-x\vert )} ds  (2t_{2}-2s)\\
&=&
1_{\{ t_{1}+t_{2} >\vert y-x\vert\} }\frac{1}{4}\int_{0} ^{t_{2}} ds (2t_{2}-2s) =1_{\{ t_{1}+t_{2} >\vert y-x\vert\} } \frac{1}{4} t_{2} ^{2}. 
\end{eqnarray*}

{\bf If $ t_{1}\leq t_{2} $ and $x-y\leq t_{2}-t_{1}$}, then

$$ x+t_{1} -s \leq y+ t_{2}-s \mbox{ and } x-t_{1}+s \geq { y-t_{2}+s}.%y+t_{2}-s.
$$

So
\begin{eqnarray*}
\mathbf{E} { \big[}u(t_{1},x) u(t_{2}, y){ \big]}
&=& 1_{\{ t_{1}+t_{2} >\vert y-x\vert\} }\frac{1}{4} \int_{0} ^{t_{1} \wedge t_{2}\wedge \frac{1}{2}(t_{1}+t_{2} -\vert y-x\vert )} ds \int_{ x-t_{1}+s} ^{x+t_{1}-s} dz \\
&=& 1_{\{ t_{1}+t_{2} >\vert y-x\vert\} }\frac{1}{4} \int_{0} ^{t_{1} \wedge t_{2}\wedge \frac{1}{2}(t_{1}+t_{2} -\vert y-x\vert )} ds (2t_{1}-2s)\\
&=& 1_{\{ t_{1}+t_{2} >\vert y-x\vert\} }\frac{1}{4}\int_{0} ^{t_{1}}  ds (2t_{1}-2s)= 1_{\{ t_{1}+t_{2} >\vert y-x\vert\} } \frac{1}{4} t_{1} ^{2}. 
\end{eqnarray*}

{\bf If $ t_{1}\leq t_{2} $ and $x-y\geq t_{2}-t_{1}$}, then

$$y+t_{2}-s\leq x+t_{1}-s \mbox{ and } x-t_{1}+s \geq { y-t_2+s}%y+t_{2}-s.
$$
Consequently
\begin{eqnarray*}
\mathbf{E}\big[ u(t_{1},x) u(t_{2}, y)\big]
&=& 
1_{\{ t_{1}+t_{2} >\vert y-x\vert\} }\frac{1}{4} \int_{0} ^{t_{1} \wedge t_{2}\wedge \frac{1}{2}(t_{1}+t_{2} -\vert y-x\vert )} ds 
{\bigg(}\int_{ x-t_{1}+s} ^{y+t_{2}-s} dz{\bigg)_+} \\
&=&   
1_{\{ t_{1}+t_{2} >\vert y-x\vert\} }\frac{1}{4} \int_{0} ^{%t_{1} \wedge t_{2}\wedge 
\frac{1}{2}(t_{1}+t_{2} -\vert y-x\vert )} ds (t_{1}+t_{2} -(x-y) -2s) \\
&=& 
1_{\{ t_{1}+t_{2} >\vert y-x\vert\} }\frac{1}{16} \left( t_{1}+ t_{2} -(x-y) \right) ^{2}. 
\end{eqnarray*}
{ By the above four esimates, we showed (\ref{201904211159}) when $x\geq y$. 
By symmetry, it is also valid for $x<y$. }

\end{proof}

Notice that the formula (\ref{10i-1}) has been obtained in \cite{KTZ} in the case $t_{1}=t_{2}$.

Now, we compute the correlation { between the increments} of the solution to the wave equation over small spatial intervals. Denote, for $i=0,.., N-1$ and $t\geq 0$

\begin{equation}\label{not}
A_{i}= \bigg[\frac{i}{N}, \frac{i+1}{N}\bigg] \mbox{ and  } 
u(t, A_{i}) = u\bigg(t, \frac{i+1}{N}\bigg)-u\bigg(t, \frac{i}{N}\bigg).
\end{equation}

\begin{lemma}\label{ll2}Let $u$ be given by (\ref{def-sol-wave}). 
{
\bd
\im[(a)] { Suppose that $t_{1}\not=t_{2}$ and $t_{1}, t_{2} >0$ with   $t_{1}+ t_{2} >1$. Then 
\begin{eqnarray}
\mathbf{E}\big[ u(t_{1}, A_{i}) u(t_{2}, A_{j})\big]
&=&
-\frac{1}{8N ^{2} }1_{\vert t_{1}-t_{2} \vert\leq \frac{\vert i-j\vert  -1 }{N}}
+ f^{(1)} _{t_{1}, t_{2}, N} (\vert i-j\vert )1_{ \frac{\vert i-j\vert  -1 }{N} < \vert t_{1}-t_{2} \vert \leq \frac{\vert i-j\vert  }{N}}\nonumber \\
&&+ f^{(2)} _{t_{1}, t_{2}, N} (\vert i-j\vert )1_{ \frac{\vert i-j\vert   }{N} < \vert t_{1}-t_{2} \vert \leq \frac{\vert i-j\vert +1  }{N}}\label{core1}
%\quad\underline{\text{Is this correct?}}
\end{eqnarray}
for $(i,j,N)\in\{0,.., N-1\}^2\times\bbN$ satisfying $i\not=j$, where we used the notation, for $1\leq k\leq N$
\begin{equation}
\label{f1}
f ^{(1)} _{t_{1}, t_{2}, N} (k)=2\times \frac{1}{16} \left( t_{1}+ t_{2} -\frac{ k}{N} \right) ^{2} -\frac{1}{16} \left( t_{1}+ t_{2} -\frac{ k+1}{N} \right) ^{2}-\frac{1}{4} (t_{1} \wedge t_{2}) ^{2}
\end{equation}
and
\begin{equation}
\label{f2}
f ^{(2)} _{t_{1}, t_{2}, N} (k)=\frac{1}{4} (t_{1} \wedge t_{2}) ^{2}-\frac{1}{16} \left( t_{1}+ t_{2} -\frac{ k+1}{N} \right) ^{2}.
\end{equation}
}
\im[(b)] Suppose that $t_{1}\not=t_{2}$ and $t_{1}, t_{2} >0$ with   $t_{1}+ t_{2} >1$. Then 
\bea\label{core2}
\mathbf{E}\big[ u(t_{1}, A_{i}) u(t_{2}, A_{i})\big] &=& 0
\eea
for any $i\in\{0,.., N-1\}$ and any $N\in\bbN$ satisfying $N>|t_1-t_2|^{-1}$. 

\im[(c)] 
Suppose that $t >\frac{1}{2}$. Then 
\bea\label{core3}
\mathbf{E}\big[ u(t, A_{i}) ^{2}\big]
&=&
\frac{1}{4N} 2t -\frac{1}{8N ^{2} }=\frac{1}{4N}\left( 2t-\frac{1}{2N}\right)
\eea
for any $N\in\bbN$ and any $i\in\{0,1,..,N-1\}$. 

\im[(d)] 
Suppose that $t >\frac{1}{2}$. Then 
\bea\label{core4}
\mathbf{E}\big[ u(t, A_{i}) u(t, A_{j}) \big]&=&  -\frac{1}{8N ^{2} }.
\eea
for any $i,j\in\{0,1,..,N-1\}$ satisfying $i\not=j$. 
\ed
}
\end{lemma}
\begin{proof}
{ Suppose that $t_1,t_2>0$ and $t_1+t_2>1$.}
We have from (\ref{19s-1}), for every $i, j=0,.., N-1$
\begin{eqnarray}&&
\mathbf{E}\big[ u(t_{1}, A_{i}) u(t_{2}, A_{j})\big]\nonumber
\\&=&2 \times \frac{1}{16} \left( t_{1}+ t_{2} -\frac{ \vert i-j\vert}{N} \right) ^{2}1_{\vert t_{1}-t_{2}\vert \leq \frac{ \vert i-j \vert } {N} }+ 2\times \frac{1}{4} (t_{1}\wedge t_{2}) ^{2} 1_{\vert t_{1}-t_{2}\vert > \frac{ \vert i-j \vert } {N} }\nonumber \\
&&-\frac{1}{16} \left( t_{1}+ t_{2} -\frac{ \vert i-j-1\vert}{N} \right) ^{2}1_{\vert t_{1}-t_{2}\vert \leq \frac{ \vert i-j -1\vert } {N} }-\frac{1}{4} (t_{1}\wedge t_{2}) ^{2} 1_{\vert t_{1}-t_{2}\vert > \frac{ \vert i-j-1 \vert } {N} }\nonumber \\
&&-\frac{1}{16} \left( t_{1}+ t_{2} -\frac{ \vert i-j+1\vert}{N} \right) ^{2}1_{\vert t_{1}-t_{2}\vert \leq \frac{ \vert i-j +1\vert } {N} }-\frac{1}{4} (t_{1}\wedge t_{2}) ^{2} 1_{\vert t_{1}-t_{2}\vert > \frac{ \vert i-j+1 \vert } {N} }.\label{2m-1}
\end{eqnarray}

{ First we will show (\ref{core2}).}
Assume that $i=j$.  Take $N$ large enough such that $\vert t_{1}-t_{2}\vert  >\frac{1}{N}$. 

\begin{eqnarray*}
\mathbf{E}\big[ u(t_{1}, A_{i}) u(t_{2}, A_{i})\big]&=&2\times \frac{1}{4} (t_{1}\wedge t_{2}) ^{2} 1_{\vert t_{1}-t_{2}\vert >0} \\
&&-2\times \frac{1}{16} \left( t_{1}+ t_{2} -\frac{ 1}{N} \right) ^{2}1_{\vert t_{1}-t_{2}\vert \leq \frac{1}{N}}- 2\times \frac{1}{4} (t_{1}\wedge t_{2}) ^{2} 1_{\vert t_{1}-t_{2}\vert > \frac{1 } {N} }\\
&=& 0.
\end{eqnarray*}

{ Next, we will verify (\ref{core1}).}
Let us assume $i>j$. In this case we have 
\begin{eqnarray}\label{2m-1bis}
\mathbf{E}\big[ u(t_{1}, A_{i}) u(t_{2}, A_{j})\big]
&=&2 \times \frac{1}{16} \left( t_{1}+ t_{2} -\frac{  i-j}{N} \right) ^{2}1_{\vert  t_{1}-t_{2}\vert \leq \frac{  i-j  } {N} }+ 2\times \frac{1}{4} (t_{1}\wedge t_{2}) ^{2} 1_{\vert t_{1}-t_{2}\vert > \frac{  i-j  } {N} }
\nn\\&&
-\frac{1}{16} \left( t_{1}+ t_{2} -\frac{  i-j-1}{N} \right) ^{2}1_{\vert t_{1}-t_{2}\vert \leq \frac{  i-j -1 } {N} }-\frac{1}{4} (t_{1}\wedge t_{2}) ^{2} 1_{\vert t_{1}-t_{2}\vert > \frac{  i-j-1  } {N} }
\nn\\&&
-\frac{1}{16} \left( t_{1}+ t_{2} -\frac{  i-j+1}{N} \right) ^{2}1_{\vert t_{1}-t_{2}\vert \leq \frac{  i-j +1 } {N} }-\frac{1}{4} (t_{1}\wedge t_{2}) ^{2} 1_{\vert t_{1}-t_{2}\vert > \frac{  i-j+1 } {N} }.
\nn\\&&
\end{eqnarray}
If $\vert t_{1}-t_{2} \vert \leq \frac{i-j-1}{N}$, the above expression gives 
\begin{eqnarray*}
\mathbf{E}\big[ u(t_{1}, A_{i}) u(t_{2}, A_{j})\big]
&=&\frac{1}{16} \left[ 2\left( t_{1}+ t_{2} -\frac{  i-j}{N} \right) ^{2} \right. \\
&&\left. - \left( t_{1}+ t_{2} -\frac{  i-j-1}{N} \right) ^{2}-\left( t_{1}+ t_{2} -\frac{  i-j+1}{N} \right) ^{2}\right]\\
&=&-\frac{1}{8N} \left( 2 (i-j) - ( i-j-1)-( i-j+1) \right) \\
&&+\frac{1}{16N^{2} }  \left( 2( i-j) ^{2} - (i-j-1) ^{2}  -(i-j+1) ^{2}  \right) \\
&=&-\frac{1}{8N} \left( 2( i-j) - ( i-j-1) -( i-j+1) \right) -\frac{1}{8N^{2}}= -\frac{1}{8N^{2}}.
\end{eqnarray*}

If $\frac{i-j-1}{N}< \vert t_{1}-t_{2}\vert \leq \frac{i-j}{N} $ or $\frac{i-j}{N}< \vert t_{1}-t_{2}\vert \leq \frac{i-j+1}{N} $, 
the conclusion is obtained directly 
from { (\ref{2m-1bis})} and finally, if $\vert t_{1}-t_{2}\vert >\frac{i-j+1}{N}$, then 
\begin{eqnarray*}
\mathbf{E}\big[ u(t_{1}, A_{i}) u(t_{2}, A_{j})\big]
&=& 
2\times \frac{1}{4} (t_{1} \wedge t_{2}) ^{2} - \frac{1}{4} (t_{1} \wedge t_{2}) ^{2} -\frac{1}{4} (t_{1} \wedge t_{2}) ^{2} =0.
\end{eqnarray*}

{ We consider the case (\ref{core3}). 
For $t>\frac{1}{2} $,} from (\ref{10i-1}), 
\begin{eqnarray*}
\mathbf{E} \big[u(t, A_{i}) ^{2}\big]&=&2\times \frac{1}{16} (2t) ^{2} -2\times \frac{1}{16} \left( 2t-\frac{1}{N} \right) ^{2} \\
&=&\frac{1}{4N} \left( 2t-\frac{1}{2N}\right) . 
\end{eqnarray*}

{ Regarding (\ref{core4}),} for $i\not=j$, 
{ by (\ref{2m-1}),}
we have 
\begin{eqnarray*}
\mathbf{E}\big[ u(t, A_{i}) u(t, { A_{j}})\big]&=&\frac{1}{16} \left[ 2\left( 2t -\frac{ \vert i-j\vert}{N} \right) ^{2} \right. \\
&&\left. - \left(2t -\frac{ \vert i-j-1\vert}{N} \right) ^{2}-\left( 2t -\frac{ \vert i-j+1\vert}{N} \right) ^{2}\right]\\
&=&-\frac{t}{4N} \left( 2\vert i-j\vert - \vert i-j-1\vert -\vert i-j+1\vert \right) \\
&&+\frac{1}{16N^{2} }  \left( 2\vert i-j\vert ^{2} - \vert i-j-1\vert ^{2}  -\vert i-j+1\vert  ^{2}  \right) \\
&=&-\frac{t}{4N} \left( 2\vert i-j\vert - \vert i-j-1\vert -\vert i-j+1\vert \right) -\frac{1}{8N^{2}}
\\&=&
-\frac{1}{8N^{2}}.
\end{eqnarray*}
Note that 
$$2\vert i-j\vert  - \vert i-j-1\vert   -\vert i-j+1\vert  =-2 \mbox{ if } i=j$$
and
$$2\vert i-j\vert - \vert i-j-1\vert -\vert i-j+1\vert  =0 \mbox{ if } i\not=j.$$
The conclusion is obtained. 
\end{proof}

As an immediate consequence of Lemma \ref{ll2}, we have
{ \begin{corollary}\label{cor1}
If $t_{1} \not= t_{2} $ are fixed in $(0, \infty)$ with $\vert t_{1}-t_{2}\vert \geq 1$ , then $u(t_{1}, A_{i}) $ and $u(t_{2}, A_{j})$ are independent Gaussian random variables for every $i, j{ \in\{0,1,..., N-1\}}$. 
\end{corollary}

}

\subsection{The quadratic variation}\label{201904220153}

Fix the time $t>{ 1/2}$. %$t>0$. 
{ 
We denote $x_{j}=\frac{j}{N}$ for $j=0,\ldots, N$ for every $N\in\bbN$. 
$x_j$ depends on $n$.}
The centered renormalized
quadratic variation statistic over the unit interval $[0,1]$ can be
defined in the following way:
\begin{equation}
\label{vn} V_{{ N,t}} =\sum_{j=0}^{N-1} \left[\frac{\left( u(t,{ x_{j+1}})-u(t,{ x_j}) 
\right)^{2}}{\mathbf{E}\big[\left(
u(t,{ x_{j+1}})-u(t,{ x_j})\right)^{2}\big]}-1\right]=\sum_{j=0}^{N-1} \left[\frac{u(t, A_{j}) ^ {2} }{\mathbf{E}\big[\vert u(t, A_{j})\vert ^ {2}\big] }-1\right].
\end{equation}

{ Starting with  this paragraph,  we will denote by $I_{q}= I _{q} ^ {W} $ the multiple integral of order $q\geq 1$ with respect to the Gaussian noise $W$ (given by (\ref{10s-1})) and by  $D=D^ {W}$ the Malliavin derivative with respect to $W$. In this case, we will write

\begin{equation}
\label{gg}
u(t, x_{i+1})- u(t, x_{i}) = I _{1} (g_{t, i}) \mbox{ with } g_{t,i} (u,y) =G_{1}(t-u, x_{i+1}-y) -G_{1}(t-u, x_{i}-y) 
\end{equation}
where $G_{1}$ from (\ref{g1}).

%we denote by
%$\mathcal{H}$ the canonical Hilbert space 
%{\colorb {\bf CHECH THE CONSTRUCTION OF $\calh$}
%associated to the Gaussian
%solution process $\left( u(t,x) \right)_{x\in \mathbb{R}}$. This
%Hilbert space is defined as the closure of the set $\xi$ of
%indicator functions $\mathds{1}_{[0,x]},\;x>0,$ with respect to the
%inner
%product:$${\langle\mathds{1}_{[0,x]},\mathds{1}_{[0,y]}\rangle}_{\mathcal{H}}=\mathbf{E}
%\Big(u(t, x)u(t, y)\Big),\;\hbox{for a fixed}\;t\in(0,T].$$ 
%}
%We designate also by ${I_{q}},\;q\geq 1$ the multiple Wiener-integral
%with respect to the Gaussian process 
%%$\left(u(t,x)\right)_{x\in[0,1]}$, 
%{ $\{\bbW(h)\}_{h\in\calh}$,} 
%so the increment $u(t,y)-u(t,x)$ can be
%expressed as 
%${I_{1}}(\mathds{1}_{[x,y]})$ 
%{ $u(t,y)-u(t,x)=I_1(h_{[x,y]})$ with some $h_{[x,y]}\in\calh$}
%for every $x<y$.

Using the product formula for multiple
Wiener-integrals (\ref{prod}), we can rewrite:\begin{eqnarray}
\label{V-N-multiple-integral}
V_{N, t}&=& \sum_{j=0} ^ {N-1} \left[  \frac{ I_{1}^2 \left( g_{t, i} \right) }
{ \mathbf{E}\left[ \Big( u(t, x_{+1})- u(t, x_{j})\Big) ^ {2} \right]}-1    \right]\nonumber\\
&=&
\sum_{j=0} ^ {N-1} \left[  \frac{ I_{2} \left( g_{t, i}^ {\otimes 2} \right) 
+ \mathbf{E}\left[ \Big( u(t, x_{j+1})- u(t, x_{j})\Big) ^ {2}\right]}
{ \mathbf{E}\left[ \Big( u(t, x_{j+1})- u(t, x_{j})\Big) ^ {2} \right]}-1\right]\nonumber\\
&=&
\sum_{j=0} ^ {N-1}   \frac{ I_{2} \left( g_{t,i}^ {\otimes 2} \right) }{ \mathbf{E}\left[ \Big( u(t, x_{j+1})-u(t,x_{j})\Big) ^ {2} \right]}
\yeq I_{2}(f_{N})
\end{eqnarray}
with
$$f_{N}= \sum_{j=0} ^ {N-1} \frac{g_{t, i} ^ {\otimes 2}}{\mathbf{E}\big[\vert u(t, A_{j})\vert ^ {2}\big]}.$$

We will use the symbolic notation, for $i, j=0,.., N$, 
\begin{equation*}
\langle A_{i}, A_{j}\rangle := \langle g_{t, i}, g_{t, j} 
\rangle _{ L ^{2} ([0, \infty)\times \mathbb{R},{ \> dtdx})}
=
\mathbf{E} \left[ u(t, A_{i}) u(t, A_{j})\right].
\end{equation*}
}
Let us first estimate the $L^{2}$-mean of the random variable $V_{N, t}$ as $N\to \infty$.

\begin{lemma}\label{201908181648} 
{Suppose that $t>1/2$.}
 Let $v_{N,t} ^{2}= \mathbf{E}\big[V_{N, t} ^{2}\big]$. {Then}
\begin{equation}
\label{4a-3}
\frac{1}{2N} v_{N, t} ^ {2}=1
+{ \half}%2
(N-1) \frac{ \left(-\frac{1}{8N^{2}}\right) ^{2}}{ \left( \frac{1}{4N}(2t-\frac{1}{2N}) \right) ^{2}}
= 1+{ O(\frac{1}{N})}.%o(\frac{1}{N}).
\end{equation}

%{ 
%In the sequel we will oimit the subindex $\mathcal{H}$ for the scalar product, and 
%$h_{A_j}$ will be identified with $A_j$ and simply denoted by $A_j$.}
\end{lemma}
\begin{proof}
Notice that 
\begin{eqnarray*}
&&v_{N, t}^ {2}= 2\sum_{j,k=0}^ {N-1} \frac{ \langle A_{j}, A_{k} \rangle ^ {2} }
 {  \mathbf{E}\big[\vert u(t, A_{j})\vert ^ {2}\big]\mathbf{E}\big[\vert u(t, A_{k})\vert ^ {2}\big]}
\\&=& 
2N+ 2\sum_{j,k=0; j\not=k}^ {N-1} \frac{ \langle A_{j}, A_{k} \rangle ^ {2} } 
{   \mathbf{E}\big[\vert u(t, A_{j})\vert ^ {2}\big]\mathbf{E}\big[\vert u(t, A_{k})\vert ^ {2}\big]}
= 2N+ 2\sum_{j,k=0; j\not=k}^ {N-1} \frac{ (-8N^ {2}) ^ {-2}} {( (4N) ^ {-1} (2t-\frac{1}{2N}) )^ {2} }\\
&=& 2N+ N(N-1)\frac{ (-8N^ {2}) ^ {-2}} {( (4N) ^ {-1} (2t-\frac{1}{2N}) )^ {2} }
\end{eqnarray*}
where we used (\ref{core3}), (\ref{core4}). 

\end{proof}

We define

\begin{equation}
\label{fn}
F_{N, t}= \frac{V_{N,t}}{\sqrt{2N}}.
\end{equation}

We will show below  that the sequence $(F_{N}) _{N\geq 1}$ satisfies the conditions  $[A1], [A2]$ and  $[C]$.

\subsection{The Gamma-factors}\label{201908181701}
We need to compute the Gamma-factors $\Gamma^{(p)}(F_{N, t}) $ for any $p\geq 1$ 
at fixed time { $t>1/2$} %$t\geq 0$ 
with $\Gamma ^{(p)} $ defined in (\ref{gf}). We have 

$$\Gamma ^{(1)} (F_{N, t})= F_{N, t} $$
and
\begin{eqnarray}
\Gamma^{(2)} (F_{N, t})&=&\langle DF_{{ N, t}}, D(-L) ^ {-1} F_{{ N, t}} \rangle \nonumber \\
&=&
\bigg\langle \frac{2}{(2N) ^{\frac{1}{2}}} \sum_{j_{1}=0} ^ {N-1}
\frac{  I_{1} ( A_{j_{1}}) I_{A_{1}}}{\mathbf{E}\big[\vert u(t, A_{j_{1}})\vert ^ {2}\big]}, 
\frac{1}{(2N) ^{\frac{1}{2}}} \sum_{j_{2}=0} ^ {N-1} 
\frac{  I_{1} ( A_{j_{2}}) A_{j_{2}}}{\mathbf{E}\big[\vert u(t, A_{j_{2}})\vert ^ {2}\big]}\bigg\rangle 
\nonumber\\&=&
\frac{2}{2N} \sum_{j_{1}, j_{2}=0 } ^ {N-1} \frac{ I_{1} (A_{j_{1}}) I_{1} (A_{j_{2}}) \langle A_{j_{1}}, A_{j_{2}}\rangle }{ \mathbf{E}\big[\vert u(t, A_{j_{1}})\vert ^ {2}\big]\mathbf{E}\big\vert u(t, A_{j_{2}})\vert ^ {2}\big]}
\nonumber \\&=&
\frac{2}{2N} \sum_{j_{1}, j_{2}=0 } ^ {N-1}
\frac{I_{2} (A_{j_{1}}\otimes A_{j_{2}})  \langle A_{j_{1}}, A_{j_{2}}\rangle }
{ \mathbf{E}\big[\vert u(t, A_{j_{1}})\vert ^ {2}\big]\mathbf{E}\big[\vert u(t, A_{j_{2}})\vert ^ {2}\big]}
+ \frac{2}{2N} \sum_{j_{1}, j_{2}=0 } ^ {N-1}
\frac{  \langle A_{j_{1}}, A_{j_{2}}\rangle ^ {2}  }
{ \mathbf{E}[\big\vert u(t, A_{j_{1}})\vert ^ {2}\big]\mathbf{E}\big[\vert u(t, A_{j_{2}})\vert ^ {2}\big]}
\nonumber \\&=&
\frac{2}{2N} \sum_{j_{1}, j_{2}=0 } ^ {N-1}\frac{I_{2} (A_{j_{1}}\otimes A_{j_{2}})  
\langle A_{j_{1}}, A_{j_{2}}\rangle }
{ \mathbf{E}\big[\vert u(t, A_{j_{1}})\vert ^ {2}\big]\mathbf{E}\big[\vert u(t, A_{j_{2}})\vert ^ {2}\big]}
+ \mathbf{E}\big[ \Gamma^{(2)} (F_{N, t})\big].\label{19s-6}
\end{eqnarray}
Next,
\begin{eqnarray*}
\Gamma ^{(3)}(F_{N, t})
&=&
\langle DF_{{ N,t}}, D(-L) ^ {-1}\Gamma ^{(2)}( F_{N, t}) \rangle 
\\&=&
\frac{4}{(2N) ^{\frac{3}{2}}} \sum_{j_{1}, j_{2}, j_{3}=0} ^ {N-1} \frac{  I_{1} (A_{j_{1}}) I_{1} (A_{j_{2}}) \langle A_{j_{1}}, A_{j_{3}}\rangle   \langle A_{j_{2}}, A_{j_{3}}\rangle }
{ \mathbf{E}\big[\vert u(t, A_{j_{1}})\vert ^ {2}\big]\mathbf{E}\big[\vert u(t, A_{j_{2}})\vert ^ {2}\big]  \mathbf{E}\big[\vert u(t, A_{j_{3}})\vert ^ {2}\big]}
\\&=&
\frac{4}{(2N) ^{\frac{3}{2}}} \sum_{j_{1}, j_{2}, j_{3}=0} ^ {N-1} \frac{I_{2} (A_{j_{1}}\otimes A_{j_{2}})  \langle A_{j_{1}}, A_{j_{3}}\rangle  \langle A_{j_{2}}, A_{j_{3}}\rangle }
{ \mathbf{E}\big[\vert u(t, A_{j_{1}})\vert ^ {2}\big]\mathbf{E}\big[\vert u(t, A_{j_{2}})\vert ^ {2}\big]  \mathbf{E}\big[\vert u(t, A_{j_{3}})\vert ^ {2}\big]}
\\&&
+  \frac{4}{(2N) ^{\frac{3}{2}}} \sum_{j_{1}, j_{2}, j_{3}=0} ^ {N-1} \frac{ \langle A_{j_{1}}, A_{j_{2}}\rangle  \langle A_{j_{1}}, A_{j_{3}}\rangle  \langle A_{j_{2}}, A_{j_{3}}\rangle }
{ \mathbf{E}\big[\vert u(t, A_{j_{1}})\vert ^ {2}\big]\mathbf{E}\big[\vert u(t, A_{j_{2}})\vert ^ {2}\big]  \mathbf{E}\big[\vert u(t, A_{j_{3}})\vert ^ {2}\big]}
\\&=&
\frac{4}{(2N) ^{\frac{3}{2}}} \sum_{j_{1}, j_{2}, j_{3}=0} ^ {N-1} \frac{I_{2} (A_{j_{1}}\otimes A_{j_{2}})  \langle A_{j_{1}}, A_{j_{3}}\rangle  \langle A_{j_{2}}, A_{j_{3}}\rangle }
{ \mathbf{E}\big[\vert u(t, A_{j_{1}})\vert ^ {2}\big]\mathbf{E}\big[\vert u(t, A_{j_{2}})\vert ^ {2}\big]  \mathbf{E}\big[\vert u(t, A_{j_{3}})\vert ^ {2}\big]}
+\mathbf{E}\Gamma ^{(3)}(F_{N, t}).
\end{eqnarray*}

In the same way, for any $p\geq 2$,
\begin{eqnarray}
\Gamma^{(p)}(F_{N, t})&=&\langle DF_{N, t}, D(-L) ^ {-1}\Gamma _{p-1}( F_{N, t}) \rangle \nonumber\\
&=&\frac{2 ^ {p-1}}{(2N)^{\frac{p}{2}}}\sum_{j_{1},..,j_{p}=0}^ {N-1}
\frac{ I_{2} (A_{j_{1}}{\tilde{\otimes}} A_{j_{2}})  \langle A_{j_{2}}, A_{j_{3}}\rangle  \langle A_{j_{3}}, A_{j_{4}}\rangle .... \langle A_{j_{p-1}}, A_{j_{p}}\rangle  \langle A_{j_{p}}, A_{j_{1}}\rangle  }
{ \mathbf{E}\big[\vert u(t, A_{j_{1}})\vert ^ {2}\big].... \mathbf{E}\big[\vert u(t, A_{j_{p}})\vert ^ {2}\big]}
\nonumber 
\\&&
+\frac{2 ^ {p-1}}{(2N)^{\frac{p}{2}}}\sum_{j_{1},..,j_{p}=0}^ {N-1}\frac{   \langle A_{j_{1}}, A_{j_{2}}\rangle \langle A_{j_{2}}, A_{j_{3}}\rangle  \langle A_{j_{3}}, A_{j_{4}}\rangle .... \langle A_{j_{p-1}}, A_{j_{p}}\rangle  \langle A_{j_{p}}, A_{j_{1}}\rangle  }
{ \mathbf{E}\big[\vert u(t, A_{j_{1}})\vert ^ {2}\big].... \mathbf{E}\big[\vert u(t, A_{j_{p}})\vert ^ {2}\big]}
\nonumber \\&=&
\frac{2 ^ {p-1}}{(2N)^{\frac{p}{2}}}\sum_{j_{1},..,j_{p}=0}^ {N-1}
\frac{ I_{2} (A_{j_{1}}{\tilde{\otimes}} A_{j_{2}})  \langle A_{j_{2}}, A_{j_{3}}\rangle  \langle A_{j_{3}}, A_{j_{4}}\rangle .... \langle A_{j_{p-1}}, A_{j_{p}}\rangle  \langle A_{j_{p}}, A_{j_{1}}\rangle  }
{ \mathbf{E}\big[\vert u(t, A_{j_{1}})\vert ^ {2}\big].... \mathbf{E}\big[\vert u(t, A_{j_{p}})\vert ^ {2}\big]}
\nonumber \\&&
+\mathbf{E}\Gamma^{(p)}(F_{N,t}).
\label{4a-2}
\end{eqnarray}
where
\begin{eqnarray}&&
\mathbf{E}\big[\Gamma ^{(p)}(F_{N,t})\big]
\nonumber\\&=&
\frac{2 ^ {p-1}}{(2N)^{\frac{p}{2}}}\sum_{j_{1},..,j_{p}=0}^ {N-1}\frac{   \langle A_{j_{1}}, A_{j_{2}}\rangle \langle A_{j_{2}}, A_{j_{3}}\rangle  \langle A_{j_{3}}, A_{j_{4}}\rangle .... \langle A_{j_{p-1}}, A_{j_{p}}\rangle  \langle A_{j_{p}}, A_{j_{1}}\rangle  }
{ \mathbf{E}\big[\vert u(t, A_{j_{1}})\vert ^ {2}\big].... \mathbf{E}\big[\vert u(t, A_{j_{p}})\vert ^ {2}\big]}.
\label{5s-1}
\end{eqnarray}

Then, using (\ref{core3}) and (\ref{core4}) in Lemma \ref{ll2}
\begin{eqnarray}\label{201904220244}
&&\mathbf{E}\big[\Gamma ^{(p)}(F_{N, t})\big]
\nn\\&=&
\frac{2 ^ {p-1}}{(2N)^{\frac{p}{2}}}\left[ N + a_{p,2} N(N-1) \left( \frac{ \left(-\frac{1}{8N^{2}}\right) }{ \left( \frac{1}{4N}(2t-\frac{1}{2N}) \right) }\right) ^{2} + a_{p,3}N(N-1) (N-2) \left( \frac{ \left(-\frac{1}{8N^{2}}\right) }{ \left( \frac{1}{4N}(2t-\frac{1}{2N}) \right) }\right) ^{3}\right.
\nn\\
&&\left. +\cdots
+ a_{p,p}{ N(N-1)\cdots(N-p+1)}
\left( \frac{ \left(-\frac{1}{8N^{2}}\right) }{ \left( \frac{1}{4N}(2t-\frac{1}{2N}) \right) }\right) ^{p}\right]
\end{eqnarray}
where $a_{p,2},.., a_{p, p}$ are combinatorial constants. In order to get the explicit asymptotic expansion of the $\mathbf{E}\big[\Gamma ^{(p)}(F_{N, t})\big]$ (which is necessary in order to check $ [C]$), we will need  to evaluate the Taylor expansion of $\left( \frac{ \left(-\frac{1}{8N^{2}}\right) }{ \left( \frac{1}{4N}(2t-\frac{1}{2N}) \right) }\right) ^{k}$ for every $k\geq 1$ integer. We can write
$$
\left( \frac{ \left(-\frac{1}{8N^{2}}\right) }{ \left( \frac{1}{4N}(2t-\frac{1}{2N}) \right) }\right) ^{k} =(-1) ^{k} \left( \frac{1}{4Nt} \right) ^{k} \left( \frac{1}{1-\frac{1}{4Nt}}\right) ^{k}.
$$

Using
$$\left( \frac{1}{1-x} \right) ^{k}=\frac{1}{(k-1)!}  \left( \frac{1}{1-x} \right) ^{(k-1)}= \frac{1}{(k-1)!} \sum_{n=k-1} ^{\infty} {\frac{n!}{ (n-k+1)!}} 
x ^{{ n-k+1}}=\sum_{n=0}^{\infty} C_{n+k-1}^{k-1} x ^{n},$$
{ where 
$C^a_b=\left(\begin{array}{c}{ b}\\ a\end{array}\right)$, 
}
we get 

\begin{equation}\label{5s-2}
\left(\frac{ \left(-\frac{1}{8N^{2}}\right) }{ \left( \frac{1}{4N}(2t-\frac{1}{2N}) \right) }\right) ^{k}=(-1) ^{k} \left( \frac{1}{4Nt} \right) ^{k} \sum_{n=0}^{\infty} C_{n+k-1}^{k-1}\left(\frac{1}{4Nt}\right) ^{n}.
\end{equation}
Let us now check assumptions  $[A1],  [A2]$  and $ [C]$ in order to apply Theorem \ref{201901040035}.

\subsection{Checking condition $[C]$}\label{201904220306}
We will show that condition $[C]$ is satisfied for $C=1$ and $\gamma=\frac{1}{2}$.  Let us  first look to assumption $[C] (i).$ Since $d=1$, we have $\mathbb{I}= \{1\} $ and the only non-empty  subset of $\mathbb{I}^{2}$ is $\{{(1, 1)}\}$. Let {$\Gamma ^{(2)} = \Gamma ^{(2)}_{I_{2} }$}. 
In order to check $ [C] (i)$, we need to find the asymptotic behavior of 
$ \mathbf{E}\big[ \Gamma^{(2)} (F_{N, t} )\big]- 1$. From (\ref{5s-1})

\begin{eqnarray*}
 \mathbf{E}\big[ \Gamma ^{(2)}(F_{N, t})\big]-1&=& \frac{2}{2N} \sum_{j_{1}, j_{2}=0 } ^ {N-1}\frac{  \langle A_{j_{1}}, A_{j_{2}}\rangle ^ {2}  }
 { \mathbf{E}\big[\vert u(t, A_{j_{1}})\vert ^ {2}\big]\mathbf{E}\big[\vert u(t, A_{j_{2}})\vert ^ {2}\big]}-1
 \\&=& 
\frac{1}{N} N(N-1) \left( \frac{ \left(-\frac{1}{8N^{2}}\right) }
{ \left( \frac{1}{4N}(2t-\frac{1}{2N}) \right) }\right) ^{2} 
\end{eqnarray*}
and from (\ref{5s-2}), for every $q\geq 3$ 

\begin{eqnarray}
 &&\mathbf{E}\big[ \Gamma ^{(2)}(F_{N, t})\big]-1\nonumber\\
&=&
\frac{1}{N} N(N-1) (-1) ^{2} \frac{1}{(4Nt)^{2}} \sum_{n=0} ^{\infty} C_{n+1}^{1} \left( \frac{1}{4Nt} \right) ^{n}\nonumber\\
&=&
(N-1)\frac{1}{(4Nt)^{2}} \left( 1+ 2\frac{1}{4Nt}+3\frac{1}{(4Nt) ^{2}} +....+ (q+1)\frac{1}{(4Nt) ^{q} }+ o(N ^{-q} ) \right)\nonumber \\
&=& 
\frac{1}{ (4t) ^{2} } \frac{1}{N} 
+ \left( 2\frac{1}{(4t) ^{3}} -\frac{1}{(4t) ^{2}}\right){ \frac{1}{N^2}+\cdots}
 +\left( { q}\frac{1}{(4t)^{q+1}} - (q-1) \frac{1}{(4t) ^{q} } \right) \frac{1}{N ^{{ q}}} + o( N^{-(q+1)}).
\nn\\\label{24i-1}
\end{eqnarray}
Therefore condition $ [C] (i) $ holds by taking $q=\frac{{ p-1}}{2}$ (so $q+1=\frac{p-1}{2}$) with 
$$ c(I_{2}, k)= 0 \mbox{ if } k \mbox{ is odd and }  c(I_{2}, 2k)= k\frac{1}{(4t)^{k+1}} - (k-1) \frac{1}{(4t) ^{k} }, \hskip0.2cm k\geq 1.$$

Next, we show $ [C] (ii). $ We will obtain the asymptotic expansion of 
${\mathbf{E}[\Gamma^{(j)} (F_{N, t})]}$ for every $j\geq 3$ integer. By using { (\ref{201904220244})},%(\ref{5s-1})
\begin{eqnarray}
&&
\mathbf{E}\big[\Gamma ^{(j)}(F_{N, t})\big]\nonumber \\&=&\frac{2 ^ {j-1}}{(2N)^{\frac{j}{2}}}\left[ N + a_{j,2} N(N-1) \left( \frac{ \left(-\frac{1}{8N^{2}}\right) }{ \left( \frac{1}{4N}(2t-\frac{1}{2N}) \right) }\right) ^{2} + a_{j,3}N(N-1) (N-2) \left( \frac{ \left(-\frac{1}{8N^{2}}\right) }{ \left( \frac{1}{4N}(2t-\frac{1}{2N}) \right) }\right) ^{3}\right.\nonumber \\
&&\left. +...
+ a_{j,j}{ N(N-1)\cdots(N-j+1)}
\left( \frac{ \left(-\frac{1}{8N^{2}}\right) }{ \left( \frac{1}{4N}(2t-\frac{1}{2N}) \right) }\right) ^{j}\right]\label{24i-4}
\end{eqnarray}
where $a_{j,2},.., a_{j, j}$ are combinatorial constants. Therefore, via (\ref{5s-2})

\begin{eqnarray*}
\mathbf{E}\big[\Gamma^{(j)}(F_{N, t})\big]
&=&\frac{2 ^ {j-1}}{(2N)^{\frac{j}{2}}} \left[ N + a_{j,2}(-1) ^{2} \frac{1}{({4t})^{2}} \frac{N!}{(N-2)! N ^{2}}\sum_{n=0} ^{\infty} C_{n+1}^{1} \left(\frac{1}{4Nt}\right) ^{n}\right. \\
&&\left.+ a_{j,3}(-1) ^{3} \frac{1}{({4t})^{3}} \frac{N!}{(N-3)! N ^{3}}\sum_{n=0} ^{\infty} C_{n+2}^{2} \left(\frac{1}{4Nt}\right) ^{n}+....+\right.\\
&&\left. +  a_{j,j}(-1) ^{p} \frac{1}{({4t})^{j}} \frac{N!}{(N-j)! N ^{j}}\sum_{n=0} ^{\infty} C_{n+j}^{j} \left(\frac{1}{4Nt}\right) ^{n}\right].
\end{eqnarray*}
Since 
\beas 
\frac{N!}{(N-k)! N ^{k} } %= 1- \frac{(k-1) k}{2} \frac{1}{N}+ o\left(\frac{1}{N}\right)$$
&=&
{ \bigg(1-\frac{1}{N}\bigg)\cdots\bigg(1-\frac{k-1}{N}\bigg)}
\eeas
for $k\geq 2$, 
we obtain
\begin{eqnarray}\label{19s-7}
\mathbf{E}\big[ \Gamma ^{(j)}(F_{N, t })\big]&=& c  (I_{j}, 1)  \frac{1}{ N^{\frac{j}{2}-1}}+ c (I_{j}, 3) \frac{1}{N^{\frac{j}{2}} }+c(I_{j}, 5) \frac{1}{ N^{\frac{j}{2}+1}}\\
&&+\ldots +c(I_{j}, p-j+2) \frac{1}{ N^{\frac{p-1}{2}}}+ o\left(\frac{1}{ N^{\frac{p-1}{2}}}\right) \nonumber
\end{eqnarray}
for $j\in \{ 3,..., p+1\}$. 
All the coefficients above can be written explicitely. In particular $c(I_{j}, 2k)=0$ for every $k\geq 1$ integer while
% Check the following equations: 
$$c (I_{j}, 1)= 2 ^{\frac{j}{2}-1}, c (I_{j}, 3)= 2 ^{\frac{j}{2}-1} \left(  a_{j,2}(-1) ^{2} \frac{1}{(4t)^{2}} + a_{j,3}(-1) ^{3} \frac{1}{(4t)^{3}} +....+ a_{j,j}(-1) ^{j} \frac{1}{(4t)^{j}} \right) $$
and
$$c (I_{j}, 5) = 2 ^{\frac{j}{2}-1}\sum_{k=2} ^{j} a_{j,k}(-1) ^{k} \frac{1}{(4t) ^{k} } \left( \frac{k+1}{4t} -\frac{ (k-1)k}{2} \right). $$

\subsection{Checking conditions [A1] and [A2] }

Condition $[A1] (i)$  is clearly verified for every $l, r$ because $F_{N, t}$ is an element of the  second Wiener chaos with $\mathbf{E}\big[F_{N}^{2}\big] \to1$ as $N\to \infty	$. Condition $ [A1] (ii)$ can be checked similarly to  $ [A2] (i)$. 

{ Let $\sfq=(p-1)/2$. }
Due to Lemma 1 (b) and by the hypercontractivity propery of multiple stochastic integrals (\ref{hyper}), it suffices to check that for every $p\geq 3$
    $$ \big\Vert \Gamma ^{{ (}p+1)} (F_{N, t})- \mathbf{E}\big[ \Gamma ^{(p+1)} (F_{N, t} ) \big]\big\Vert _{ 2} = o (N ^{-\frac{p-1}{2}}).$$

Let $\tilde{\Gamma} ^ {(p+1) }(F_{N, t} )= \Gamma ^{(p+1)} (F_{N,t})
-\mathbf{E}\big[ \Gamma^{(p+1)}(F_{N, t})\big].$ 
Then { by (\ref{4a-2}),}
\begin{eqnarray}
&&\mathbf{E}\bigg[ \left(\tilde{\Gamma} ^ {(p+1) }(F_{N, t} )\right) ^{2}\bigg]
%\nonumber\\&=& 
\left( \frac{2 ^ {p}}{(2N) ^{\frac{p+1}{2}}}\right) ^{2} 
\nn\\&&\times
2\sum_{j_{1},..,j_{p+1}, \atop k_{1},.., k_{p+1}=0}^ {N-1}
\frac{ \langle A_{j_{1}}{\tilde{ \otimes}} A_{j_{2}}, A_{k_{1}}\tilde{ \otimes} A_{k_{2}}\rangle \langle A_{j_{2}}, A_{j_{3}}\rangle .... \langle A_{j_{p+1}}, A_{j_{1}}{\rangle}
\langle A_{k_{2}}, A_{k_{3}}\rangle...  \langle A_{k_{p+1}}, A_{k_{1}}\rangle }
{ \mathbf{E}\big[\vert u(t, A_{j_{1}})\vert ^ {2}\big]\cdots \mathbf{E}\big[\vert u(t, A_{j_{p+1}})\vert ^ {2}\big]\mathbf{E}\big[\vert u(t, A_{k_{1}})\vert ^ {2}\big]\cdots\mathbf{E}\big[\vert u(t, A_{k_{p+1}})\vert ^ {2}\big]}\nonumber \\
&=& 
\left( \frac{2 ^ {p+1}}{(2N) ^{\frac{p+1}{2}}}\right) ^{2} 2
\sum_{j_{1},..,j_{p+1}, \atop k_{1},.., k_{p+1}=0}^ {N-1}
\nonumber\\&&\times 
\frac{  \langle A_{j_{1}}, A_{k_{1}}\rangle\langle A_{j_{2}}, A_{k_{2}}\rangle\langle A_{j_{2}}, A_{j_{3}}\rangle....  \langle A_{j_{p+1}}, A_{j_{1} }\rangle \langle A_{k_{2}}, A_{k_{3}}\rangle....  \langle A_{k_{p+1}}, A_{k_{1}}\rangle}
{ \mathbf{E}\big[\vert u(t, A_{j_{1}})\vert ^ {2}\big]\cdots \mathbf{E}\big[\vert u(t, A_{j_{p+1}})\vert ^ {2}\big]\mathbf{E}\big[\vert u(t, A_{k_{1}})\vert ^ {2}\big]\cdots\mathbf{E}\big[\vert u(t, A_{k_{p+1}})\vert ^ {2}\big]} \label{4a-1}
\end{eqnarray}
where we used the symmetry of the sums above and the formula
$$\langle f\tilde{\otimes } g, f_{1}, \tilde{\otimes } g_{1} \rangle= \frac{1}{2} (\langle f,f_{1}\rangle \langle g,g_{1}\rangle +\langle f,g_{1}\rangle \langle g,f_{1}\rangle ).$$

Consequently, by (\ref{core3}) and (\ref{core4}), we get

\begin{eqnarray}
&&
\mathbf{E}\bigg[ \left(\tilde{\Gamma} ^ {(p) }(F_{N, t} )\right) ^{2}\bigg]
\nonumber \\&=& 
2\left( \frac{2 ^ {p}}{(2N) ^{\frac{p+1}{2}}}\right) ^{2} ( N+ O(1)) = 2 ^{p} N ^{-(p+1)} ( N+ O(1))= 2 ^{p} N ^{-p} + o(N^{-p}).\label{23i-1}
\end{eqnarray}

The above estimate is true for every $p\geq 3$ and it clearly implies $ [A2]$.

To summarize the behavior of the cumulants,  we have the  situation described in Table 1 with $\gamma=\frac{1}{2}$.

\begin{en-text}
\begin{table}[ht]
\caption{The asymptotic behavior of $ \mathbf{E}\left[ \Gamma ^ {(j)} \right]$ }
\centering
\begin{tabular}{c c c c c}
\hline\hline
Summand  & First order  term &Second term& Third term &Fourth term \\ [0.5ex] % inserts table %heading
\hline
$\mathbf{E}\left[ \Gamma ^ {(3)}(F_{N})\right]$  & $\frac{1}{N ^{\gamma}}$& $\frac{1}{ N ^ {\gamma +1}}$ &$\frac{1}{ N ^ {\gamma +2}}$& $\frac{1}{ N ^ {\gamma +3}}$ \\
$\mathbf{E} \left[ \Gamma ^ {(4)}(F_{N}) \right]+ (\mathbf{E}\left[ \Gamma ^ {(2)}(F_{N})\right] -1)$&$\frac{1}{ N ^ {\gamma +\frac{1}{2} }}$ &$\frac{1}{ N ^ {\gamma +\frac{3}{2} }}$&$\frac{1}{ N ^ {\gamma +\frac{5}{2} }}$&$\frac{1}{ N ^ {\gamma +\frac{7}{2} }}$ \\
$\mathbf{E}\left[ \Gamma ^ {(5)}(F_{N})\right]$  &$\frac{1}{N ^{\gamma +1}}$ &$\frac{1}{N ^{\gamma+2}}$ &$\frac{1}{N ^{\gamma+3}}$&$\frac{1}{N ^{\gamma+4}}$ \\
$\mathbf{E}\left[ \Gamma ^ {(6)}(F_{N})\right]$  & $\frac{1}{ N ^ {\gamma +\frac{3}{2} }}$ & $\frac{1}{ N ^ {\gamma +\frac{5}{2} }}$ & $\frac{1}{ N ^ {\gamma +\frac{7}{2} }}$&$\frac{1}{ N ^ {\gamma +\frac{9}{2} }}$ \\
$\mathbf{E}\left[ \Gamma ^ {(7)}(F_{N})\right]$   &  $\frac{1}{ N ^ {\gamma +2}}$ &  $\frac{1}{ N ^ {\gamma +3}}$ &  $\frac{1}{ N ^ {\gamma +4}}$& $\frac{1}{ N ^ {\gamma +5}}$ \\ 
$\ldots $\\
$\mathbf{E}\left[ \Gamma ^ {(p)}(F_{N})\right]$   &  $\frac{1}{ N ^ {\gamma +\frac{p-3}{2}}}$ & $\frac{1}{ N ^ {\gamma +\frac{p-3}{2} +1}}$ &  $\frac{1}{ N ^ {\gamma +\frac{p-3}{2}+2}}$& $\frac{1}{ N ^ {\gamma +\frac{p-3}{2}+3}}$ \\
$\mathbf{E}\left[ \Gamma ^ {(p+1)}(F_{N})\right]$   &  $\frac{1}{ N ^ {\gamma +\frac{p-2}{2}}}$ & $\frac{1}{ N ^ {\gamma +\frac{p-2}{2} +1}}$ &  $\frac{1}{ N ^ {\gamma +\frac{p-2}{2}+2}}$& $\frac{1}{ N ^ {\gamma +\frac{p-2}{2}+3}}$\\
%$ \Gamma ^ {(p+1)}(F_{N})- \Gamma ^ {(p+1)}(F_{N})$   &  $\frac{1}{ N ^ {\gamma +\frac{p-1}{2}}}$ \\
\hline
\end{tabular}
\label{table1}
\end{table}
\end{en-text}
\begin{table}[ht]
\caption{The asymptotic behavior of $ \mathbf{E}\left[ \Gamma ^ {(j)} \right]$ }
\centering
{
\begin{tabular}{c c c c c}
\hline\hline
Summand  & First order  term &Second term& Third term &Fourth term \\ [0.5ex] % inserts table %heading
\hline
$\mathbf{E}\left[ \Gamma ^ {(3)}(F_{N})\right]$  & $\frac{1}{N ^{\gamma}}$& $\frac{1}{ N ^ {3\gamma }}$ &$\frac{1}{ N ^ {5\gamma }}$& $\frac{1}{ N ^ {7\gamma }}$ \\
$\mathbf{E} \left[ \Gamma ^ {(4)}(F_{N}) \right]+ (\mathbf{E}\left[ \Gamma ^ {(2)}(F_{N})\right] -1)$&$\frac{1}{ N ^ {2\gamma }}$ &$\frac{1}{ N ^ {4\gamma }}$&$\frac{1}{ N ^ {6\gamma }}$&$\frac{1}{ N ^ {8\gamma }}$ \\
$\mathbf{E}\left[ \Gamma ^ {(5)}(F_{N})\right]$  &$\frac{1}{N ^{3\gamma}}$ &$\frac{1}{N ^{5\gamma}}$ &$\frac{1}{N ^{7\gamma}}$&$\frac{1}{N ^{9\gamma}}$ \\
$\mathbf{E}\left[ \Gamma ^ {(6)}(F_{N})\right]$  & $\frac{1}{ N ^ {4\gamma }}$ & $\frac{1}{ N ^ {6\gamma  }}$ & $\frac{1}{ N ^ {8\gamma}}$&$\frac{1}{ N ^ {10\gamma }}$ \\
$\mathbf{E}\left[ \Gamma ^ {(7)}(F_{N})\right]$   &  $\frac{1}{ N ^ {5\gamma}}$ &  $\frac{1}{ N ^ {7\gamma }}$ &  $\frac{1}{ N ^ {9\gamma }}$& $\frac{1}{ N ^ {11\gamma }}$ \\ 
$\ldots $\\
$\mathbf{E}\left[ \Gamma ^ {(p)}(F_{N})\right]$   &  $\frac{1}{ N ^ {(p-2)\gamma}}$ & $\frac{1}{ N ^ {p\gamma }}$ &  $\frac{1}{ N ^ {(p+2)\gamma}}$& $\frac{1}{ N ^ {(p+4)\gamma}}$ \\
$\mathbf{E}\left[ \Gamma ^ {(p+1)}(F_{N})\right]$   &  $\frac{1}{ N ^ {(p-1)\gamma}}$ & $\frac{1}{ N ^ {(p+1)\gamma}}$ &  $\frac{1}{ N ^ {(p+3)\gamma}}$& $\frac{1}{ N ^ {(p+5)\gamma}}$\\
%$ \Gamma ^ {(p+1)}(F_{N})- \Gamma ^ {(p+1)}(F_{N})$   &  $\frac{1}{ N ^ {\gamma +\frac{p-1}{2}}}$ \\
\hline
\end{tabular}
}
\label{table1}
\end{table}

\subsection{The multidimensional case}\label{201904220344}
Consider the sequence
\begin{equation}
\label{fn2}{\bf F}_{N, t} = (F_{N, t_{1}}, F_{N, t_{2}}) 
\end{equation}
with $F_{N, t}$ given by (\ref{fn}) and $t_{1}\not=t_{2}$ 
and { $t_1,t_2>1/2$}. 
%and $t_{1}+ t_{2}>1$. 
Using relations (\ref{core1}) and (\ref{core2}) we can prove that this two-dimensional sequence satisfies $ [A1], [A2]$ and $ [C]$.

Recall that we  denoted by $I_{q}= I _{q} ^ {W} $ the multiple integral of order $q\geq 1$ with respect to the Gaussian field $W$ (whose covariance is given by (\ref{10s-1})) and by  $D=D^ {W}$ the Malliavin derivative with respect to $W$. In this case, the sequence (\ref{fn}) can be written

\begin{equation*}
F_{N, t_{i}}= \frac{1}{{\sqrt{2N}}}%{v_{N, t_{i}}}
\sum_{j=0} ^ {N-1} \frac{ I_{2}(g_{t_{i},  j}^ {\otimes 2})} 
{ \mathbf{E} \big[\vert u(t, A_{j})\vert ^ {2}\big]}
\end{equation*}
for $i=1,2$, where we used the notation $u(t, A_{j})$ from (\ref{not}) and $g_{t, i}$ from (\ref{gg}). Let us first compute the Gamma factors (\ref{gf}) of the vector (\ref{fn2}). Since, for $a=1,2$

\begin{equation*}
DF_{N, t_{a}}=\frac{2}{{\sqrt{2N}}}%{v_{N, t_{a}} }
\sum _{j=0} ^ {N-1} \frac{ I_{1}(g_{t_{a},  j}) g_{t_{a},  j} }
{ \mathbf{E}\big[ \vert u(t_{a}, A_{j})\vert ^ {2}\big]}
\mbox{ and }
D(-L) ^ {-1} F_{N, t_{a}} =  \frac{1}{{\sqrt{2N}}}%{v_{N, t_{a}} }
\sum _{j=0} ^ {N-1} \frac{ I_{1}(g_{t_{a},  j}) g_{t_{a},  j} }{ \mathbf{E} \big[\vert u(t_{a}, A_{j})\vert ^ {2}\big]}
\end{equation*}
and by using the symbolic notation, for $a, b=1,2$ and $j_{1}, j_{2}=0,1,.., N-1$
\begin{equation*}
 \langle A_{j_{1}}, A_{j_{2}}\rangle _{\mathcal{H} _{t_{a}, t_{b}}}
 \yeq \mathbf{E} \big[u(t_{a}, A_{j_{1}}) u(t_{b}, A_{j_{2}})\big]
 {=\langle g_{t_a,j_1},g_{t_b,j_2}\rangle}.
\end{equation*}
we get, for every $i_{1}, i_{2}=1,2$

\begin{eqnarray*}
\Gamma ^{(2)}_ {i_{1}, i_{2}} ({\bf F }_{N,t} )
&=&
\langle DF_{N, t_{i_{1}}}, D(-L) ^ {-1} F_{N, t_{{ i_{2}}}}\rangle \\
&=&
{ \frac{1}{N}}%{ v_{N, t_{i_{1}}}v_{N, t_{i_{2}}}} 
\sum_{j_{1}, j_{2}=0} ^ {N-1} \frac{  I_{1}(g_{t_{i_{1}}, j_{1}})   I_{1}(g_{t_{i_{2}},  j_{2}})   \langle A_{j_{1}}, A_{j_{2}}\rangle _{\mathcal{H} _{t_{i_{i}}, t_{i_{2}}}}}
{ \mathbf{E} \big[\vert u(t_{i_{1}}, A_{j_{1}})\vert ^ {2}\big] \mathbf{E}\big[\vert u(t_{i_{2}}, A_{j_{2}})\vert ^ {2}\big]}
\\&=&
{ \frac{1}{N}}%\frac{2}{ v_{N, t_{i_{1}}}v_{N, t_{i_{2}}}} 
\sum_{j_{1}, j_{2}=0} ^ {N-1} \frac{  I_{2}((g_{t_{i_{1}},  j_{1}})
{\tilde{\otimes}}  (g_{t_{i_{2}}, j_{2}}) )  \langle A_{j_{1}}, A_{j_{2}}\rangle _{\mathcal{H} _{t_{i_{i}}, t_{i_{2}}}}}
{ \mathbf{E} \big[\vert u(t_{i_{1}}, A_{j_{1}})\vert ^ {2}\big] \mathbf{E}\big[\vert u(t_{i_{2}}, A_{j_{2}})\vert ^ {2}\big]}
\\&&
+{ \frac{1}{N}}%\frac{2}{ v_{N, t_{i_{1}}}v_{N, t_{i_{2}}}} 
\sum_{j_{1}, j_{2}=0} ^ {N-1} \frac{   \langle A_{j_{1}}, A_{j_{2}}\rangle^ {2} _{\mathcal{H} _{t_{i_{i}}, t_{i_{2}}}}}
{ \mathbf{E} \big[\vert u(t_{i_{1}}, A_{j_{1}})\vert ^ {2}\big] \mathbf{E}\big[\vert u(t_{i_{2}}, A_{j_{2}})\vert ^ {2}\big]}
\\&=&
{ \frac{1}{N}}%\frac{2}{ v_{N, t_{i_{1}}}v_{N, t_{i_{2}}}} 
\sum_{j_{1}, j_{2}=0} ^ {N-1} \frac{  I_{2}((g_{t_{i_{1}},  j_{1}}){\tilde{\otimes}}  (g_{t_{i_{2}}, j_{2}}) )  \langle A_{j_{1}}, A_{j_{2}}\rangle _{\mathcal{H} _{t_{i_{i}}, t_{i_{2}}}}}
{ \mathbf{E} \big[\vert u(t_{i_{1}}, A_{j_{1}})\vert ^ {2}\big] \mathbf{E}\big[\vert u(t_{i_{2}}, A_{j_{2}})\vert ^ {2}\big]}
\\&&
+\mathbf{E} \big[\Gamma ^{(2)}_ {(i_{1}, i_{2})} ({\bf F }_{N, t} )\big].
\end{eqnarray*}
In the same way, we get for $p\geq 2$

\begin{eqnarray*}
\Gamma ^ {(p)} _{i_{1},.., i_{p}}( {\bf F}_{N, t}) &=& \frac{2 ^ {p-1}}{(2N) ^ {\frac{p}{2}}}  \sum_{j_{1},.., j_{p} =0}^ {N-1}  
 \frac{  I_{2}((g_{t_{i_{1}}, x, j_{1}}){\tilde{\otimes}}  (g_{t_{i_{2}}, x, j_{2}}) )  \langle A_{j_{2}}, A_{j_{3}}\rangle _{\mathcal{H} _{t_{i_{2}}, t_{i_{3}}}} ....  \langle A_{j_{p}}, A_{j_{1}}\rangle _{\mathcal{H} _{t_{i_{p}}, t_{i_{1}}}}}
 { \mathbf{E} \big[\vert u(t_{i_{1}}, A_{j_{1}})\vert ^ {2}\big]\cdots\mathbf{E}\big[ \vert u(t_{i_{p}}, A_{j_{p}})\vert ^ {2}\big]}
 \\&&
 + \mathbf{E}\big[ \Gamma ^ {(p)} _{i_{1},.., i_{p}} ({\bf F}_{N, t} )\big]
\end{eqnarray*}
where
\begin{equation}
\mathbf{E} {\big[}\Gamma ^ {(p)} _{i_{1},.., i_{p}} ({\bf F}_{N, t} ){\big]}
= 
\frac{2 ^ {p-1}}{(2N) ^ {\frac{p}{2}}}  \sum_{j_{1},.., j_{p} =0}^ {N-1}  \frac{  \langle A_{j_{1}}, A_{j_{2}}\rangle _{\mathcal{H} _{t_{i_{i}}, t_{i_{2}}}}  \langle A_{j_{2}}, A_{j_{3}}\rangle _{\mathcal{H} _{t_{i_{2}}, t_{i_{3}}}}...  \langle A_{j_{p}}, A_{j_{1}}\rangle _{\mathcal{H} _{t_{i_{p}}, t_{i_{1}}}}}
{ \mathbf{E} \big[\vert u(t_{i_{1}}, A_{j_{1}})\vert ^ {2}\big]\cdots\mathbf{E}\big[ \vert u(t_{i_{p}}, A_{j_{p}})\vert ^ {2}\big]}.
\label{24i-2}
\end{equation}

\subsection{Checking conditions [A1], [A2] and [C] }

We will check our main {\colorg assumptions} in the particular case $p=2$ with $ \gamma=\frac{1}{2} $ and $C$ {\colorg  a symmetric} matrix in dimension two. Let us first determine the limit covariance matrix $C$. 

Denote by $\lceil x\rceil$ the minimum integer that is not less than $x$. 
Let $\eta_N=\lceil\tau N\rceil-N\tau$ for $\tau=|t_1-t_2|$.

\begin{lemma}\label{ll10}
Let $t_{1} \not=t_{2}$ 
{ and $t_1,t_2>1/2$}. 
\begin{enumerate}
\item Assume $\vert t_{1}- t_{2}\vert  \geq 1$. Then
{
\begin{equation*}
\mathbf{E} \big[  \Gamma ^ {(2)} _{i,j} ({\bf F}_{N, t} )\big] = %C_{i,j}
1_{\{i=j\}}\big(1+O(N^{-1})\big)
\end{equation*}
}
%with $ C_{1,1}= C_{2,2}=1$ and $ C_{1,2}= C_{2,1} =0$. 
as $N\to\infty$ for $i,j=1,2$. 

\item Assume  $\vert t_{1}-t_{2}\vert <1$
{ (and hence $t_{1}+ t_{2} >1$)}. 
\begin{en-text}
Then for $i,j=1,2$, 
 \begin{equation*}
\lim _{N \to \infty} \mathbf{E} \left[  \Gamma ^ {(2)} _{i,j} ({\bf F}_{N, t} )\right] = C_{i,j}
\end{equation*}
with $ C_{1,1}= C_{2,2}=1$ and 
{
$ C_{1,2}= C_{2,1} $ given by (\ref{3m-3}). 
\end{en-text}
{
Then, as $N\to\infty$,  
\beas
\mathbf{E} \big[  \Gamma ^ {(2)} _{i,i} ({\bf F}_{N, t} )\big] 
&=&
1+O(N^{-1})
\eeas
for $i=1,2$, and 
\bea\label{3m-3}
\mathbf{E} \big[  \Gamma ^ {(2)} _{i,j} ({\bf F}_{N, t}) \big]
&=& 
\frac{(1-|t_1-t_2|)(t_1\wedge t_2)^2}{2t_1t_2}
\big(2\eta_N^2-2\eta_N+1\big)+O(N^{-1})
\eea
for $(i,j)=(1,2)$, $(2,1)$. 
%along any subsequence such that $\eta_N\to a\in[0,1]$. 
}
\end{enumerate} 
\end{lemma}

\begin{proof} 
In both cases,  
we have  
\begin{equation*}
\mathbf{E} \big[  \Gamma ^ {(2)} _{i,i} ({\bf F}_{N, t} )\big] = %C_{i,j}
1+O(N^{-1})
\end{equation*}
for $i=1,2$ by Lemma \ref{201908181648}.

{
We will investigate the asymptotic behavior of 
$\mathbf{E} [  \Gamma ^ {(2)} _{1,2} ({\bf F}_{N, t} )]$. }
{ In Case 1,}
by Corollary \ref{cor1}, { we obtain} 
\begin{equation*}
\mathbf{E} \left[  \Gamma ^ {(2)} _{1,2} ({\bf F}_{N, t} )\right]= \mathbf{E} \left[ F_{N, t_{1}} F_{N, t_{2}} \right] =0.
\end{equation*}

In Case 2, using (\ref{core1}) and { (\ref{core2}),}%(\ref{core3}),
we can write 
\begin{eqnarray}
\mathbf{E} \left[  \Gamma ^ {(2)} _{1,2} ({\bf F}_{N, t} )\right]&=&\frac{1}{N} \sum_{j_{1}, j_{2}=0 } ^{N-1} \frac{ \left( \mathbf{E} \big[u(t_{1}, A_{j_{1}})u(t_{2}, A_{j_{2}})\big]\right) ^{2}   }
{ \mathbf{E}\big[ \vert u(t_{1}, A_{j_{1}})\vert ^ {2} \big]\mathbf{E} \big[\vert u(t_{1}, A_{j_{1}})\vert ^ {2}\big]}\nonumber \\
&=&\frac{1}{N} \left( \frac{1}{4N}(2t_{1}-\frac{1}{2N}) \right) ^{-1} \left( \frac{1}{4N}(2t_{2}-\frac{1}{2N}) \right) ^{-1}  \nonumber\\
&&
{ \times}
\sum_{j_{1}, j_{2}=0, j_{1}\not=j_{2}} ^{N-1} \left[ \frac{-1}{8N ^{2}}1_{\vert t_{1} -t_{2} \vert \leq \frac{ \vert j_{1}-j_{2}\vert -1}{N}}+  f ^{(1)} _{t_{1}, t_{2}, N} (\vert j_{1}- j_{2} \vert )1_{\frac{ \vert j_{1}-j_{2}\vert -1}{N}<\vert t_{1} -t_{2} \vert \leq \frac{ \vert j_{1}-j_{2}\vert }{N}}\right. \nonumber\\
&&\left. +f ^{(2)} _{t_{1}, t_{2}, N} (\vert j_{1}- j_{2} \vert )1_{\frac{ \vert j_{1}-j_{2}\vert }{N}<\vert t_{1} -t_{2} \vert \leq \frac{ \vert j_{1}-j_{2}\vert +1 }{N}} \right] ^{2}\label{3m-1}
\end{eqnarray}
with $f^{(1)}, f^{(2)}$ from (\ref{f1}) and (\ref{f2}) respectively.

To get the limit  of the sequence $u_{N}:= \mathbf{E}\big[ \Gamma ^{(2)}_{1,2} ({\bf F} _{N, t}) \big]$, it suffices to get the limit  of $ u_{ N\vert t_{1}-t_{2}\vert ^{-1}}.$ We can write, by (\ref{3m-1}), with $N_{1,2}= N\vert t_{1}- t_{2}\vert ^{-1}$

\begin{eqnarray}
u_{ N\vert t_{1}-t_{2}\vert ^{-1}}&=&\frac{1}{N_{1,2}} \left( \frac{1}{4N_{1,2}}(2t_{1}-\frac{1}{2N_{1,2}}) \right) ^{-1} \left( \frac{1}{4N_{1,2} }(2t_{2}-\frac{1}{2N_{1,2}}) \right) ^{-1}  \nonumber\\
&&{\times}
(u_{N} ^{(1)}+ u_{N} ^{(2)} + u_{N} ^{(3)})\label{3m-2}
\end{eqnarray}
with
\begin{equation*}
u_{N} ^{(1)} = \left( -\frac{1}{8N^{2}}\right) ^{2} \sum_{j_{1}, j_{2}=0, j_{1}\not=j_{2}} ^{[N_{1,2}]-1} 1_{\vert j_{1}-j_{2} \vert \geq N+1},
\end{equation*}

\begin{equation*}
u_{N} ^{(2)} = \sum_{j_{1}, j_{2}=0, j_{1}\not=j_{2}} ^{[N_{1,2}]-1}\left( f ^{(1)} _{t_{1}, t_{2}, N_{1,2}}(\vert j_{1}-j_{2}\vert)\right) ^{2} 1_{ \vert j_{1}-j_{2}\vert =N}, \hskip0.2cm u_{N} ^{(3)} = \sum_{j_{1}, j_{2}=0, j_{1}\not=j_{2}} ^{[N_{1,2}]-1}\left( f ^{(2)} _{t_{1}, t_{2}, N_{1,2}}(\vert j_{1}-j_{2}\vert) \right) ^{2} 1_{ \vert j_{1}-j_{2}\vert =N-1}.
\end{equation*}

Notice that
\begin{eqnarray}
&&f ^{(1)} _{t_{1}, t_{2}, N_{1,2}}(N) = f ^{(2)} _{t_{1}, t_{2}, N_{1,2}}(N-1) \nonumber\\
&&=\frac{1}{16} \left(t_{1}+ t_{2} -\vert t_{1}-t_{2}\vert \right) ^{2}-\frac{1}{16} \left(t_{1}+ t_{2} -\vert t_{1}-t_{2}\vert -\vert t_{1}-t_{2}\vert \frac{1}{N} \right)^{2} \nonumber \\
&&= \frac{1}{8}\vert t_{1}-t_{2}\vert (t_{1}\wedge t_{2})\frac{1}{N}+ c_{2} \frac{1}{N ^{2} }+c_{3}\frac{1}{N ^{3}}+\ldots.\label{3m-4}
\end{eqnarray}

Coming back to (\ref{3m-2}), we deduce that 
\begin{eqnarray}
\mathbf{E} \left[  \Gamma ^ {(2)} _{1,2} ({\bf F}_{N, t} )\right]&\sim& \frac{1}{N _{1,2}} \frac{ 4 N_{1,2} ^{2}}{t_{1}t_{2}} \left( \frac{1}{8} (t_{1}\wedge t_{2}) \right) ^{2} \frac{1}{N_{1,2} ^{2} } \sum_{j_{1}, j_{2}=0, j_{1}\not=j_{2}} ^{[N_{1,2}]-1}\left[  1_{\vert j_{1}-j_{2} \vert = N+1}+1_{\vert j_{1}-j_{2} \vert = N-1}\right]\nonumber \\
&\sim & \frac{4}{64} \frac{(t_{1}\wedge t_{2}) ^{2} }{t_{1}t_{2}} 4 \left( \vert t_{1}-t_{2} \vert ^{-1}-1\right): =C_{1,2}= C_{2,1}.\label{3m-33}
\end{eqnarray}

%%
%%%%%%%%%%%%%%%%%%%%%

\noindent
%\underline{Discussion on the paragraph starting with (\ref{3m-2})} 
Denote $u_{N}= \mathbf{E}\big[ \Gamma ^{(2)}_{1,2} ({\bf F} _{N, t}) \big]$. 
Let $\tau=|t_1-t_2|$. Then 
\bea\label{201908190221}
u_N&=&\frac{1}{N} \left( \frac{1}{4N}(2t_{1}-\frac{1}{2N}) \right) ^{-1} \left( \frac{1}{4N }(2t_{2}-\frac{1}{2N}) \right) ^{-1}
(u_{N} ^{(1)}+ u_{N} ^{(2)} + u_{N} ^{(3)})\label{3m-2bis}
\eea
with
\begin{equation*}
u_{N} ^{(1)} = \left( -\frac{1}{8N^{2}}\right) ^{2} 
\sum_{j_{1}, j_{2}=0, j_{1}\not=j_{2}} ^{N-1} 1_{\{\vert j_{1}-j_{2} \vert \geq \tau N+1\}},
\end{equation*}
\beas
u_{N} ^{(2)} = \sum_{j_{1}, j_{2}=0, j_{1}\not=j_{2}} ^{N-1}\left( f ^{(1)} _{t_{1}, t_{2}, N}(\vert j_{1}-j_{2}\vert)\right) ^{2} 1_{\{ \tau N\leq\vert j_{1}-j_{2}\vert <\tau N+1\}}
\eeas
and
\beas
u_{N} ^{(3)} = \sum_{j_{1}, j_{2}=0, j_{1}\not=j_{2}} ^{N-1}\left( f ^{(2)} _{t_{1}, t_{2}, N}(\vert j_{1}-j_{2}\vert) \right) ^{2} 1_{ \{\tau N-1\leq\vert j_{1}-j_{2}|< \tau N\}}.
\eeas
The contribution of $u^{(1)}_N$ to $u_N$ is asymptotically negligible since 
\beas 
u^{(1)}_N
&=& 
O(N^{-2}). 
\eeas
Moreover, 
\beas 
u^{(2)}_N
&=& 
2(N-\lceil\tau N\rceil)\times\bigg\{
\frac{2}{16}\bigg(t_1+t_2-\frac{\lceil\tau N\rceil}{N}\bigg)^2
-\frac{1}{16}\bigg(t_1+t_2-\frac{\lceil\tau N\rceil+1}{N}\bigg)^2
-\frac{1}{4}(t_1\wedge t_2)^2
\bigg\}^2
\eeas
and 
\beas 
u^{(3)}_N
&=& 
2(N-\lceil\tau N\rceil+1)\times\bigg\{
\frac{1}{4}(t_1\wedge t_2)^2
-\frac{1}{16}\bigg(t_1+t_2-\frac{\lceil\tau N\rceil}{N}\bigg)^2
\bigg\}^2.
\eeas
{\colorg Since }
\beas
t_1+t_2-\frac{\lceil\tau N\rceil}{N}
&=&
t_1+t_2-\tau-\frac{\eta_N}{N}
\yeq 
2(t_1\wedge t_2)-\frac{\eta_N}{N}
\eeas
{\colorg we can express $u^{(2)}_N , u^{(3)}_N $ as }
\beas 
u^{(2)}_N 
&=& 
2(1-\tau)N^{-1}\times\bigg\{
\frac{1}{4}(t_1\wedge t_2)(1-\eta_N)+O(N^{-1})\bigg\}^2
+O(n^{-2})
\eeas
and 
\beas 
u^{(3)}_N 
&=& 
2(1-\tau)N^{-1}\times\bigg\{\frac{1}{4}(t_1\wedge t_2)\eta_N+O(N^{-1})\bigg\}^2
+O(n^{-2}). 
\eeas
Consequently, 
\beas 
u_N 
&=& 
\frac{(1-|t_1-t_2|)(t_1\wedge t_2)^2}{2t_1t_2}
\big\{(1-\eta_N)^2+\eta_N^2\big\}+O(N^{-1}).
\eeas

\end{proof}

\halflineskip

When $|t_1-t_2|<1$, the sequence $u_N$ does not converge. We need to consider the limit 
{\cred of $u_N$} %$N\to\infty$ 
along a subsequence of $N$ such that $\eta_N\to a\in[0,1]$. 
{\cred More precisely, }
$u_N$ does not converge, 
{\cred however,}
\bea\label{201908221054} 
u_N
&=& 
\frac{(1-|t_1-t_2|)(t_1\wedge t_2)^2}{2t_1t_2}
\big(2a^2-2a+1\big)+o(N^{-1/2})
\eea
along any subsequence such that 
\bea\label{201908220853}
\eta_N=a+o(N^{-1/2})
\eea 
as $N\to\infty$ 
for some $a\in[0,1]$.

\begin{en-text}
[We should remark that the representation (\ref{201908190221}) of $u_N$ 
is valid even when 1 $|t_1-t_2|\geq1$, and then 
$u^{(2)}_N$ and $u^{(3)}_N$ vanishes, and hence we obtain 
$C_{1,2}=0$ also from this consideration. ]
\end{en-text}

Thus, in Case 2, we can consider 
approximation to the distribution of ${\bf F}_{N,t}$ 
by the asymptotic expansion along the subsequence. 

%%%%%%%%%%%%%%%%%%
%%
\begin{en-text}
We now check  [A1], [A2] and [C] with  $C$ from the above lemma, $p=2$ and $\gamma =\frac{1}{2}$. Since $\mathbb{I} = \{1, 2\}$ we need to evaluate  
$$\mathbf{E} \big[\Gamma ^{(2)} _{I_{2}}  ({\bf F} _{N, t} ) \big]- C_{I_{2} }  \mbox{ and } 
\mathbf{E} \big[\Gamma ^{(3) } ({\bf F} _{N, t} ) \big]$$
for every  $I_{2}$ in  $ \mathbb{I} ^{2}$. 

Consider fist the situation when $ I_{2}= (i, i)$ with $i=1, 2$. Then
$$ \mathbf{E}\big[ \Gamma ^{(2)} _{I_{2}}  ({\bf F} _{N, t} )\big] - C_{I_{2} }\yeq
 \mathbf{E}\big[\Gamma ^{(2)}(F_{N, t_{i}})\big]-1 $$
and we can use the estimate (\ref{24i-1})  with $t=t_{i}$.

If $ I_{2}= (1, 2) $ (the case $I_{2}= (2,1) $ is similar), we will have  $\mathbf{E}\big[ \Gamma ^{(2)}_{1,2} ({\bf F} _{N, t}) \big]$ given by (\ref{3m-1}) (recall that we assumed $t_{1}+ t_{2}>1$)

The  quantity in (\ref{3m-1}) depends on the lenght of the time interval $\vert t_{1}- t_{2} \vert$. If $\vert t_{1}- t_{2}\vert  \geq 1$, then by Corollary \ref{cor1},  
$$\mathbf{E}\big[ \Gamma ^{(2)}_{1,2} ({\bf F} _{N, t})\big]=0.$$
If $\vert t_{1}-t_{2} \vert <1$,  then this correlation term does not vanish. From relations (\ref{3m-2}) and (\ref{3m-4}) in the proof of Lemma \ref{ll10}, the explicit Taylor expansion of $\mathbf{E} \Gamma ^{(2)}_{1,2} ({\bf F} _{2N, t})$ can be  obtained as follows 
\begin{equation}\label{201908190529}
\mathbf{E} \big[\Gamma ^{(2)}_{1,2} ({\bf F} _{N, t}) \big] -C_{1,2}= C_{1}\frac{1}{N}+ C_{2} \frac{1}{ N^{2}}+ \ldots
\end{equation}
with some explicite constants $C_{i}$, $i=1,2,..$.
\end{en-text}

{
Define the matrix $C=(C_{i,j})_{i,j=1,2}$ by 
$C_{i,j}=1_{\{i=j\}}$ in the case $|t_1-t_2|\geq1$, and by 
$C_{1,1}=C_{2,2}=1$ and 
\beas 
C_{1,2}\yeq C_{2,1}
&=& 
\frac{(1-|t_1-t_2|)(t_1\wedge t_2)^2}{2t_1t_2}
\big(2a^2-2a+1\big)
\eeas
in the case $|t_1-t_2|<1$. 

We now check  $[A1], [A2]$ and  $[C]$ %with $C$ specified by the above lemma, 
for $p=2$ and $\gamma =\frac{1}{2}$. 
In what follows, we will consider the full sequence $(N)_{N\in\bbN}$ when $|t_1-t_2|\geq1$, but 
only consider 
a subsequence of $(N)_{N\in\bbN}$ satisfying (\ref{201908220853}) 
when $|t_1-t_2|<1$. 
Since $\mathbb{I} = \{1, 2\}$, we need to evaluate  
$$\mathbf{E} \big[\Gamma ^{(2)} _{I_{2}}  ({\bf F} _{N, t} ) \big]%- C_{I_{2} }  
\mbox{ and } 
\mathbf{E} \big[\Gamma ^{(3) } ({\bf F} _{N, t} ) \big]$$
for every  $I_{2}$ in  $ \mathbb{I} ^{2}$. 

By Lemma \ref{ll10}, we have the property  $[C](i)$, that is, 
\beas
\mathbf{E} \big[  \Gamma ^ {(2)} _{I_2} ({\bf F}_{N, t} )\big] 
&=&
C_{I_2}+o(N^{-1/2})
\eeas
for $I_2\in\bbI^2$ ($i=1,2$) as $N\to\infty$ 
(but along the subsequence when $|t_1-t_2|<1$). 

\begin{en-text}
If $ I_{2}= (1, 2) $ (the case $I_{2}= (2,1) $ is similar), we will have  $\mathbf{E}\big[ \Gamma ^{(2)}_{1,2} ({\bf F} _{N, t}) \big]$ given by (\ref{3m-1}) (recall that we assumed $t_{1}+ t_{2}>1$)

The  quantity in (\ref{3m-1}) depends on the lenght of the time interval $\vert t_{1}- t_{2} \vert$. If $\vert t_{1}- t_{2}\vert  \geq 1$, then by Corollary \ref{cor1},  
$$\mathbf{E}\big[ \Gamma ^{(2)}_{1,2} ({\bf F} _{N, t})\big]=0.$$
If $\vert t_{1}-t_{2} \vert <1$,  then this correlation term does not vanish. From relations (\ref{3m-2}) and (\ref{3m-4}) in the proof of Lemma \ref{ll10}, the explicit Taylor expansion of $\mathbf{E} \Gamma ^{(2)}_{1,2} ({\bf F} _{2N, t})$ can be  obtained as follows 
\begin{equation}\label{201908190529}
\mathbf{E} \big[\Gamma ^{(2)}_{1,2} ({\bf F} _{N, t}) \big] -C_{1,2}= C_{1}\frac{1}{N}+ C_{2} \frac{1}{ N^{2}}+ \ldots
\end{equation}
with some explicite constants $C_{i}$, $i=1,2,..$.
\end{en-text}
}

In order to check $ [C] (ii)$, we need to estimate 
$\mathbf{E}\big[\Gamma ^ {(3)} _{i_{1}, i_{2}, i_{3}}({\bf F}_{N, t} )\big]$ with $(i_{1}, i_{2}, i_{3} ) \in \{ 1, 2\} ^{2}$. 

If $i_{1}=i_{2}=i_{3}$ (and they are $1$ or $2$), then we can follow the lines from the one-dimensional case, see relation (\ref{24i-4}).  We will have, for $i=1, 2$,

\begin{eqnarray*}
\mathbf{E} \big[\Gamma ^{(3)} _{i, i, i}({\bf F}_{N, t} )\big]
&=&\mathbf{E} \big[\Gamma ^{(3) } (F_{N, t_{i}})\big]\\
&=&\frac{ 4}{ (2N) ^ {\frac{3}{2}}}\left[ N+3N(N-1)  \left( \frac{ \left(-\frac{1}{8N^{2}}\right) }{ \left( \frac{1}{4N}(2t_{i}-\frac{1}{2N}) \right) }\right) ^{2}  \right.\\
&&\left. + N(N-1) (N-2) \left( \frac{ \left(-\frac{1}{8N^{2}}\right) }{ \left( \frac{1}{4N}(2t_{i}-\frac{1}{2N}) \right) }\right) ^{{ 3}}\right]
\end{eqnarray*}
and by  (\ref{5s-2}) we obtain
\begin{eqnarray*}
\mathbf{E} \big[\Gamma^ {(3)}_{i, i, i}({\bf F}_{N, t})\big]
&=&\frac{ 4}{ (2N) ^ {\frac{3}{2}}} \left[ N + 3(-1) ^{2} \frac{1}{(2N(2t_{i}))^{2}} 
{\frac{N!}{(N-2)!}}%\frac{N!}{(N-2)! N ^{2}}
\sum_{n=0} ^{\infty} C_{n+1}^{1} \left(\frac{1}{2N(2t_{i})}\right) ^{n}\right. \\
&&\left.+ (-1) ^{3} \frac{1}{(2N(2t_{i}))^{3}} 
{ \frac{N!}{(N-3)! }}%\frac{N!}{(N-3)! N ^{3}}
\sum_{n=0} ^{\infty} C_{n+2}^{2} \left(\frac{1}{2N(2t_{i})}\right) ^{n}\right]. 
\end{eqnarray*}
We thus obtain the estimate (\ref{19s-7}) where in the expression of the coefficients we replace $t$ by $t_{i}$.

If $i_{1}, i_{2}, i_{3} \in \{1, 2\}$ are not all equal, then we have a different behaviors of the quantity 
$\mathbf{E}\big[ \Gamma ^ {(3)} _{i_{1}, i_{2}, i_{3}}({\bf F}_{N, t} )\big]$. We can assume $i_{1}=i_{2}=1$ and $i_{3}= 2$ since the other cases can be treated analogously.  In this situation, from (\ref{core1})-(\ref{core4})
\begin{eqnarray*}&&
\mathbf{E} \big[\Gamma ^ {(3)} _{i_{1}, i_{2}, i_{3}}({\bf F}_{N, t} )\big]
\\&=&
\frac{ 4}{ (2N) ^ {\frac{3}{2}}}\sum_{j_{1}, j_{2}, j_{3}=0} ^{N-1} 
\frac{ \mathbf{E}\big[ u(t_{1}, A_{j_{1}})u(t_{1}, A_{j_{2}})\big]
\mathbf{E}\big[ u(t_{1}, A_{j_{2}})u(t_{2}, A_{j_{3}})\big]
\mathbf{E}\big[ u(t_{2}, A_{j_{3}})u(t_{1}, A_{j_{1}})\big]}
{ \mathbf{E}\big[ \vert u(t_{1}, A_{j_{1}})\vert ^{2}\big] \mathbf{E}\big[ \vert u(t_{1}, A_{j_{2}})\vert ^{2} \big]\mathbf{E} \big[\vert u(t_{2}, A_{j_{3}})\vert ^{2}\big] }\\
&=&
\frac{ 4}{ (2N) ^ {\frac{3}{2}}} \left( \frac{1}{4N}\left(2t_{1}-\frac{1}{2N}\right)\right) ^{-2} \left( \frac{1}{4N}\left(2t_{2}-\frac{1}{2N}\right)\right) ^{-1} \\
&&\times \sum_{j_{1}, j_{2}, j_{3}=0} ^{N-1}  
\mathbf{E}\big[ u(t_{1}, A_{j_{1}})u(t_{1}, A_{j_{2}})\big]
\mathbf{E}\big[ u(t_{1}, A_{j_{2}})u(t_{2}, A_{j_{3}})\big]
\mathbf{E}\big[ u(t_{2}, A_{j_{3}})u(t_{1}, A_{j_{1}})\big]\\
&=&
\frac{ 4}{ (2N) ^ {\frac{3}{2}}} \left( \frac{1}{4N}\left(2t_{1}-\frac{1}{2N}\right)\right) ^{-2} \left( \frac{1}{4N}\left(2t_{2}-\frac{1}{2N}\right)\right) ^{-1} \\
&&
\times \sum_{j_{1}, j_{2}, j_{3}=0; j_{2}\not=j_{3}\not=j_{1}} ^{N-1}  
\mathbf{E}\big[ u(t_{1}, A_{j_{1}})u(t_{1}, A_{j_{2}})\big]
\mathbf{E}\big[ u(t_{1}, A_{j_{2}})u(t_{2}, A_{j_{3}})\big]
\mathbf{E}\big[ u(t_{2}, A_{j_{3}})u(t_{1}, A_{j_{1}})\big]
\\
&=&
\frac{ 4}{ (2N) ^ {\frac{3}{2}}} \left( \frac{1}{4N}\left(2t_{1}-\frac{1}{2N}\right)\right) ^{-2} \left( \frac{1}{4N}\left(2t_{2}-\frac{1}{2N}\right)\right) ^{-1}\\
&&
\times \left[ \sum_{j_{1},  j_{2}=0; j_{1}\not=j_{2}} ^{N-1} 
\mathbf{E}\big[\vert  u(t_{1}, A_{j_{1}}) \vert ^{2}\big] \left(  \mathbf{E}\big[ u(t_{1}, A_{j_{1}})u(t_{2}, A_{j_{2}})\big]\right)^{2}\right. \\
&&
\left. +  \sum_{j_{1}, j_{2}, j_{3}=0; j_{2}\not=j_{3}\not=j_{1}\not=j_{2}} ^{N-1}
\mathbf{E}\big[ u(t_{1}, A_{j_{1}})u(t_{1}, A_{j_{2}})\big]
\mathbf{E}\big[ u(t_{1}, A_{j_{2}})u(t_{2}, A_{j_{3}})\big]
\mathbf{E}\big[ u(t_{2}, A_{j_{3}})u(t_{1}, A_{j_{1}})\big]
 \right]
 \\&=:&
 \frac{ 4}{ (2N) ^ {\frac{3}{2}}} \left( \frac{1}{4N}\left(2t_{1}-\frac{1}{2N}\right)\right) ^{-2} \left( \frac{1}{4N}\left(2t_{2}-\frac{1}{2N}\right)\right) ^{-1}
\big(v^{(1)}_N+v^{(2)}_N\big). 
\end{eqnarray*}
{
In Case $|t_1-t_2|\geq1$, we see 
$\mathbf{E} \big[\Gamma ^ {(3)} _{i_{1}, i_{2}, i_{3}}({\bf F}_{N, t} )\big]=0$
by Corollary \ref{cor1}. 
In Case $|t_1-t_2|<1$, we are only considering the subsequence of $(N)_{N\in\bbN}$. 
By Lemma \ref{ll2} (d) and (a), we have 
\beas
|v^{(2)}_N|
&\simleq&
O(N^{-2})\times
\bigg|\sum_{j_{1}, j_{2}, j_{3}=0; j_{2}\not=j_{3}\not=j_{1}\not=j_{2}} ^{N-1}
\bigg|\mathbf{E}\big[ u(t_{1}, A_{j_{2}})u(t_{2}, A_{j_{3}})\big]
\mathbf{E}\big[ u(t_{2}, A_{j_{3}})u(t_{1}, A_{j_{1}})\big]\bigg|
\\&=&
O(N^{-2})\times \bigg(O(N^2\times N^{-2}\times N^{-2})
+O(N\times N^{-1}\times N^{-1})\bigg)
\yeq O(N^{-3})
\eeas
since the last sum is essentially one-dimensional. 
Therefore, $v^{(2)}_N$ has no essential contribution in the limit. 
As for $v^{(1)}_N$, by Lemma \ref{ll2} (c), 
\beas 
v^{(1)}_N
&=&
\sum_{j_{1},  j_{2}=0; j_{1}\not=j_{2}} ^{N-1} 
\mathbf{E}\big[\vert  u(t_{1}, A_{j_{1}}) \vert ^{2}\big] 
\left(  \mathbf{E}\big[ u(t_{1}, A_{j_{1}})u(t_{2}, A_{j_{2}})\big]\right)^{2}
\\&=&
\bigg(\frac{t_1}{2N}+O(N^{-2})\bigg)
\sum_{j_{1},  j_{2}=0; j_{1}\not=j_{2}} ^{N-1} 
\left(  \mathbf{E}\big[ u(t_{1}, A_{j_{1}})u(t_{2}, A_{j_{2}})\big]\right)^{2}.
\eeas
Therefore, 
\beas &&
\frac{ 4}{ (2N) ^ {\frac{3}{2}}} \left( \frac{1}{4N}\left(2t_{1}-\frac{1}{2N}\right)\right) ^{-2} \left( \frac{1}{4N}\left(2t_{2}-\frac{1}{2N}\right)\right) ^{-1}
v^{(1)}_N
\\&=&
\frac{ 4}{ 2 ^ {\frac{3}{2}}N^{1/2}} \left( \frac{1}{4N}\left(2t_{1}-\frac{1}{2N}\right)\right)^{-1}\bigg(\frac{t_1}{2N}+O(N^{-2})\bigg)u_N. 
\eeas
From (\ref{201908221054}), we obtain, {\colorg with $a$ from (\ref{201908220853}), }
\beas 
 \mathbf{E}\big[\Gamma ^ {(3)} _{1.1.2}({\bf F}_{N, t} )\big]
&=&
\frac{ 4}{ 2 ^ {\frac{3}{2}}N^{1/2}} \left( \frac{1}{4N}\left(2t_{1}-\frac{1}{2N}\right)\right)^{-1}\bigg(\frac{t_1}{2N}+O(N^{-2})\bigg)
\\&&\times
\frac{(1-|t_1-t_2|)(t_1\wedge t_2)^2}{2t_1t_2}
\big(2a^2-2a+1\big)+o(N^{-1/2})
\\&=&
\frac{(1-|t_1-t_2|)(t_1\wedge t_2)^2}{\sqrt{2N}t_1t_2}
\big(2a^2-2a+1\big)+o(N^{-1/2})
\eeas
Consequently, we obtained 
\begin{eqnarray}\label{201908190557}
\mathbf{E}\big[ \Gamma ^ {(3)} _{i_{1}, i_{2}, i_{3}}({\bf F}_{N, t} )\big]=
D_{1}N^{-1/2} + o(N^{-1/2})
\end{eqnarray}
and this verifies $ [C](ii)$.  

Similarly to the one-dimensional case, we can check $ [A1]--[A2] $. 
Recall that
\begin{eqnarray*}
\tilde{\Gamma}^ {(3)} _{i_{1},.., i_{3}}( {\bf F}_{N, t}) 
&=& 
\frac{2 ^ {2}}{(2N) ^ {\frac{3}{2}}}  
\sum_{j_{1},j_{2}, j_{3} =0}^ {N-1}  
 \frac{  I_{2}((g_{t_{i_{1}}, x, j_{1}}){\tilde{\otimes}}  (g_{t_{i_{2}}, x, j_{2}}) )  \langle A_{j_{2}}, A_{j_{3}}\rangle _{\mathcal{H} _{t_{i_{2}}, t_{i_{3}}}}   \langle A_{j_{3}}, A_{j_{1}}\rangle _{\mathcal{H} _{t_{i_{3}}, t_{i_{1}}}}}
 { \mathbf{E} \big[\vert u(t_{i_{1}}, A_{j_{1}})\vert ^ {2}\big]
 \mathbf{E} \big[\vert u(t_{i_{2}}, A_{j_{2}})\vert ^ {2}\big]
 \mathbf{E}\big[ \vert u(t_{i_{3}}, A_{j_{3}})\vert ^ {2}\big]}
\end{eqnarray*}
for 
\beas 
\tilde{\Gamma} ^ {(3)} _{i_{1},.., i_{3}}( {\bf F}_{N, t}) 
&=&
\Gamma ^ {(3)} _{i_{1},.., i_{3}}( {\bf F}_{N, t}) 
-\mathbf{E}\big[ \Gamma ^ {(3)} _{i_{1},.., i_{3}} ({\bf F}_{N, t} )\big].
\eeas
Therefore, 
\beas &&
{\colorg \mathbf{E}}\big[|\tilde{\Gamma}^ {(3)} _{i_{1},.., i_{3}}( {\bf F}_{N, t}) |^2\big]
\\&=&
\bigg(\frac{2 ^ {2}}{(2N) ^ {\frac{3}{2}}} \bigg)^2 
\sum_{j_{1},j_2, j_{3} =0\atop
j'_1,j'_2,j'_3=0}^ {N-1}  
\bigg\{
\frac{  \langle A_{j_1}, A_{j'_1}\rangle _{\mathcal{H} _{t_{i_1}, t_{i_1}}}
 \langle A_{j_2}, A_{j'_2}\rangle _{\mathcal{H} _{t_{i_2}, t_{i_2}}}
  \langle A_{j_{2}}, A_{j_{3}}\rangle _{\mathcal{H} _{t_{i_{2}}, t_{i_{3}}}}   
}
 { \mathbf{E} \big[\vert u(t_{i_{1}}, A_{j_{1}})\vert ^ {2}\big]
 \mathbf{E} \big[\vert u(t_{i_{2}}, A_{j_{2}})\vert ^ {2}\big]
 \mathbf{E}\big[ \vert u(t_{i_{3}}, A_{j_{3}})\vert ^ {2}\big]}
 \\&&\qquad\qquad\qquad\qquad\times
 \frac{  
 \langle A_{j_{3}}, A_{j_{1}}\rangle _{\mathcal{H} _{t_{i_{3}}, t_{i_{1}}}}
  \langle A_{j'_{2}}, A_{j'_{3}}\rangle _{\mathcal{H} _{t_{i_{2}}, t_{i_{3}}}}   
\langle A_{j'_{3}}, A_{j'_{1}}\rangle _{\mathcal{H} _{t_{i_{3}}, t_{i_{1}}}}}
 { \mathbf{E} \big[\vert u(t_{i_{1}}, A_{j'_{1}})\vert ^ {2}\big]
 \mathbf{E} \big[\vert u(t_{i_{2}}, A_{j'_{2}})\vert ^ {2}\big]
 \mathbf{E}\big[ \vert u(t_{i_{3}}, A_{j'_{3}})\vert ^ {2}\big]}
 \bigg\}
 \\&&+
\bigg(\frac{2 ^ {2}}{(2N) ^ {\frac{3}{2}}} \bigg)^2 
\sum_{j_{1},j_2, j_{3} =0\atop
j'_1,j'_2,j'_3=0}^ {N-1}  
\bigg\{
\frac{  \langle A_{j_1}, A_{j'_2}\rangle _{\mathcal{H} _{t_{i_1}, t_{i_2}}}
 \langle A_{j_2}, A_{j'_1}\rangle _{\mathcal{H} _{t_{i_2}, t_{i_1}}}
  \langle A_{j_{2}}, A_{j_{3}}\rangle _{\mathcal{H} _{t_{i_{2}}, t_{i_{3}}}}   
}
 { \mathbf{E} \big[\vert u(t_{i_{1}}, A_{j_{1}})\vert ^ {2}\big]
 \mathbf{E} \big[\vert u(t_{i_{2}}, A_{j_{2}})\vert ^ {2}\big]
 \mathbf{E}\big[ \vert u(t_{i_{3}}, A_{j_{3}})\vert ^ {2}\big]}
 \\&&\qquad\qquad\qquad\qquad\times
 \frac{  
 \langle A_{j_{3}}, A_{j_{1}}\rangle _{\mathcal{H} _{t_{i_{3}}, t_{i_{1}}}}
  \langle A_{j'_{2}}, A_{j'_{3}}\rangle _{\mathcal{H} _{t_{i_{2}}, t_{i_{3}}}}   
\langle A_{j'_{3}}, A_{j'_{1}}\rangle _{\mathcal{H} _{t_{i_{3}}, t_{i_{1}}}}}
 { \mathbf{E} \big[\vert u(t_{i_{1}}, A_{j'_{1}})\vert ^ {2}\big]
 \mathbf{E} \big[\vert u(t_{i_{2}}, A_{j'_{2}})\vert ^ {2}\big]
 \mathbf{E}\big[ \vert u(t_{i_{3}}, A_{j'_{3}})\vert ^ {2}\big]}
 \bigg\}
 %\\&=& O(N^{-3})\times\bigg(O(N)+\sum_{\nu=2,4,6}O(N^\nu\times (N^{-1})^\nu)\bigg)
\eeas
and hence 
\beas&&
{\colorg \mathbf{E}}\big[|\tilde{\Gamma}^ {(3)} _{i_{1},.., i_{3}}( {\bf F}_{N, t}) |^2\big]
\\&\simleq&
\sum{}^*
N^3\sum_{j_1,...,j_6=0}^{N-1}  
\bigg\{
\big|\langle A_{j_1}, A_{j_2}\rangle _{\mathcal{H} _{t_{i'_1}, t_{i'_2}}}\big|
\big|\langle A_{j_2}, A_{j_3}\rangle _{\mathcal{H} _{t_{i'_2}, t_{i'_3}}}\big|
\big| \langle A_{j_3}, A_{j_4}\rangle _{\mathcal{H} _{t_{i'_3}, t_{i'_4}}}\big|
\\&&\qquad\qquad\qquad\qquad\times 
\big|\langle A_{j_4}, A_{j_5}\rangle _{\mathcal{H} _{t_{i'_4}, t_{i'_5}}}\big|
\big|\langle A_{j_5}, A_{j_6}\rangle _{\mathcal{H} _{t_{i'_5}, t_{i'_6}}}\big| 
\big|\langle A_{j_6}, A_{j_1}\rangle _{\mathcal{H} _{t_{i'_6}, t_{i'_1}}}\big|
\bigg\}
\eeas
where $\sum^*$ is the sum for all permutations 
$(i'_1,...,i'_6)$ of $(i_1,i_1,i_1,i_1,i_2,i_2,i_2,i_2,i_3,i_3,i_3,i_3)$. 
By Lemma \ref{ll2} (a), (c), (d), we conclude
\begin{en-text}
\beas 
\sum_{i,j=0,i\not=j}^{N-1}
\big|\langle A_i, A_j\rangle _{\mathcal{H}_{t_a, t_b}}\big|
&=&
O(1)
\eeas
\beas 
\sum_{i,j=0,i=j}^{N-1}
\big|\langle A_i, A_j\rangle _{\mathcal{H}_{t_a, t_b}}\big|
&=&
O(1)
\eeas
\end{en-text}
that there exists a constant $K$ such that 
\beas 
\sup_{i=0,...,N-1}\sum_{j=0}^{N-1}
\big|\langle A_i, A_j\rangle _{\mathcal{H}_{t_a, t_b}}\big|
&\leq&
KN^{-1}
\eeas
for any $a,b\in\{1,2\}$ and $n\in\bbN$. 
{\cred By the Schwarz inequality and Lemma \ref{ll2} (c), we have 
\beas 
\big|\langle A_{j_1}, A_{j_2}\rangle _{\mathcal{H} _{t_{i_1}, t_{i_2}}}\big|
&\leq&
\frac{\max\{t_1,t_2\}}{2N}
\quad(j_1,j_2=0,...,N-1;\>i_1,i_2=1,2). 
\eeas
}
\begin{en-text}
\beas
\bbI(j_1)
&:=&
\max_{j'_1}
\sum_{j_2,...,j_6=0}^{N-1}  
\bigg\{
\big|\langle A_{j'_1}, A_{j_2}\rangle _{\mathcal{H} _{t_{i'_1}, t_{i'_2}}}\big|
\big|\langle A_{j_2}, A_{j_3}\rangle _{\mathcal{H} _{t_{i'_2}, t_{i'_3}}}\big|
\big| \langle A_{j_3}, A_{j_4}\rangle _{\mathcal{H} _{t_{i'_3}, t_{i'_4}}}\big|
\\&&\qquad\qquad\qquad\times 
\big|\langle A_{j_4}, A_{j_5}\rangle _{\mathcal{H} _{t_{i'_4}, t_{i'_5}}}\big|
\big|\langle A_{j_5}, A_{j_6}\rangle _{\mathcal{H} _{t_{i'_5}, t_{i'_6}}}\big| 
\big|\langle A_{j_6}, A_{j_1}\rangle _{\mathcal{H} _{t_{i'_6}, t_{i'_1}}}\big|
\bigg\}
\\&\leq&
\max_{j'_1}
\sum_{j_2=0}^{N-1}  
\bigg\{
\big|\langle A_{j'_1}, A_{j_2}\rangle _{\mathcal{H} _{t_{i'_1}, t_{i'_2}}}\big|
\max_{j'_2}
\sum_{j_3,...,j_6=0}^{N-1}  
\big|\langle A_{j'_2}, A_{j_3}\rangle _{\mathcal{H} _{t_{i'_2}, t_{i'_3}}}\big|
\big| \langle A_{j_3}, A_{j_4}\rangle _{\mathcal{H} _{t_{i'_3}, t_{i'_4}}}\big|
\\&&\qquad\qquad\qquad\times 
\big|\langle A_{j_4}, A_{j_5}\rangle _{\mathcal{H} _{t_{i'_4}, t_{i'_5}}}\big|
\big|\langle A_{j_5}, A_{j_6}\rangle _{\mathcal{H} _{t_{i'_5}, t_{i'_6}}}\big| 
\big|\langle A_{j_6}, A_{j_1}\rangle _{\mathcal{H} _{t_{i'_6}, t_{i'_1}}}\big|
\bigg\}
\\&\leq&
KN^{-1}
\>\bigg\{
\max_{j'_2}
\sum_{j_3,...,j_6=0}^{N-1}  
\big|\langle A_{j'_2}, A_{j_3}\rangle _{\mathcal{H} _{t_{i'_2}, t_{i'_3}}}\big|
\big| \langle A_{j_3}, A_{j_4}\rangle _{\mathcal{H} _{t_{i'_3}, t_{i'_4}}}\big|
\\&&\qquad\qquad\qquad\times 
\big|\langle A_{j_4}, A_{j_5}\rangle _{\mathcal{H} _{t_{i'_4}, t_{i'_5}}}\big|
\big|\langle A_{j_5}, A_{j_6}\rangle _{\mathcal{H} _{t_{i'_5}, t_{i'_6}}}\big| 
\big|\langle A_{j_6}, A_{j_1}\rangle _{\mathcal{H} _{t_{i'_6}, t_{i'_1}}}\big|
\bigg\}. 
\eeas
By induction, %the left-hand side of the above inequality is dominated by 
\beas 
\bbI(j_1) \yleq 
K^{\cred5}N^{-5}
\>\bigg\{
\max_{j'_6}
\big|\langle A_{j'_6}, A_{j_1}\rangle _{\mathcal{H} _{t_{i'_6}, t_{i'_1}}}\big|
\bigg\}
&\leq&
\frac{K^{\cred5}\max\{t_1,t_2\}}{2N^6}
\quad(j_1=0,...,N-1)
\eeas
where the last inequality is 
by the Schwarz inequality and Lemma \ref{ll2} (c). 
\end{en-text}
{\cred Then, by Lemma \ref{201909130444} below with $\Lambda=\{1,2\}^2$,}
we obtain 
\beas 
\big\|\tilde{\Gamma}^ {(3)} _{i_{1},.., i_{3}}( {\bf F}_{N, t}) \big\|_2
&=&
O(N^{-1}),
\eeas
in particular, $ [A2](ii)$. 
Since $\tilde{\Gamma}^ {(3)} _{i_{1},.., i_{3}}( {\bf F}_{N, t})$ is in the second chaos, 
we obtain $ [A2] (i)$ for any $\ell_1\in\bbN$. 

Verification of $ [A1]$ can be done in a similar way. 
The components of  ${\bf F}_{N, {\colorg t}}$ and $\Gamma^{(2)}({\bf F}_{N,{\colorg t}})-C$ are 
in the second chaos. 
First we can show %the inequalities in  $[A1]$ for $r=2$ 
$\sup_{N\in\bbN}\|{\bf F}_{N, {\colorg t}}\|_2<\infty$ and 
$\|\Gamma^{(2)}({\bf F}_{N, {\colorg t}})-C\|_2=O(N^{-1/2})$, 
and next use hypercontractivity to 
obtain $L^r$-estimates for any $r>2$. 
Condition [A1] is verified if we follow the same procedure after applying the Malliavin operator $\lceil(\ell+1)/2\rceil$-times to these variables. 

In conclusion, the %regularly ordered 
asymptotic expansion for the multi-variate ${\bf F}_{N, {\colorg t}}$ 
is valid for $p=2$ and $\gamma=1/2$ as $N\to\infty$ when $|t_1-t_2|\geq1$, and so
along a subsequence satisfying (\ref{201908220853}) when $|t_1-t_2|<1$. {\colorg With more tedious computation, our approach can be extended to any $p\geq 2$.}
}

{\cred In the above discussion, we used the following lemma. 
\begin{lemma}\label{201909130444}
Let $\Lambda$ and $\bbJ$ be finite sets. 
%Let $\bbI$ be a finite set. Let $(\bbJ_N)_{N\in\bbN}$ be a sequence of sets such that 
%\beas 
%\#\bbJ_N&\leq& CN \quad(N\in\bbN)
%\eeas
%for some constant $C$. 
Let $a^{\lambda}(j_1,j_2)\in\bbR$ 
for $\lambda\in\Lambda$, $j_1,j_2\in\bbJ$. 
Let $\Delta\in\bbR$. 
Suppose that there exists a constant $K$ such that 
\beas 
\max_{\lambda\in\Lambda}\max_{j_1\in\bbJ}\sum_{j_2\in\bbJ}
\big|a^{\lambda}(j_1,j_2)\big|
&\leq& K\Delta.%N^{-1}
\eeas
Let $k\in\bbN$ satisfying $k\geq2$. 
Let 
\beas 
S(j_1;\lambda_1,...,\lambda_k)
&=&
\sum_{j_2,...,j_k\in\bbJ_N}a^{\lambda_1}(j_1,j_2)a^{\lambda_2}(j_2,j_3)\cdots
a^{\lambda_{k-1}}(j_{k-1},j_k)a^{\lambda_k}(j_k,j_1).
\eeas
Let 
\beas 
T(\lambda_1,...,\lambda_k)
&=&
\sum_{j_1,j_2,...,j_k\in\bbJ_N}a^{\lambda_1}(j_1,j_2)a^{\lambda_2}(j_2,j_3)\cdots
a^{\lambda_{k-1}}(j_{k-1},j_k)a^{\lambda_k}(j_k,j_1).
\eeas
Let 
\beas 
\Delta'
&=&
\max_{\lambda\in\Lambda}\max_{j_1,j_2\in\bbJ}\big|a^{\lambda}(j_1,j_2)\big|. 
\eeas
Then 
\bd
\im[(a)] 
$\ds 
\max_{\lambda_1,...,\lambda_k\in\Lambda}
\max_{j_1\in\bbJ}\big|S(j_1;\lambda_1,...,\lambda_k)\big|
\yleq
K^{k-1}\Delta^{k-1}\Delta'
$. 
\im[(b)] 
$\ds \max_{\lambda_1,...,\lambda_k\in\Lambda}\big|T(\lambda_1,...,\lambda_k)\big|
\yleq 
K^{k-1}\Delta^{k-1}\Delta'\#\bbJ
$. 
\ed
\end{lemma}
\proof We may assume that $a^{\lambda}(j_1,j_2)\geq0$. 
The property (b) follows from (a). We will show (a). 
Let 
\beas 
\bbS^{(p-1)}(j_1)
&=& 
\max_{\lambda_1,...,\lambda_p\in\Lambda}
\max_{j_1'\in\bbJ}
\sum_{j_2,...,j_p\in\bbJ}a^{\lambda_1}(j_1',j_2)a^{\lambda_2}(j_2,j_3)\cdots
a^{\lambda_{p-1}}(j_{p-1},j_p)a^{\lambda_p}(j_p,j_1)
\eeas
for $p=2,3,...$. 
Then 
\beas &&
\bbS^{(p-1)}(j_1)
\\&=&
\max_{\lambda_1,...,\lambda_p\in\Lambda}
\max_{j_1'\in\bbJ}\sum_{j_2\in\bbJ}a^{\lambda_1}(j_1',j_2)
\sum_{j_3,...,j_p\in\bbJ}a^{\lambda_2}(j_2,j_3)\cdots
a^{\lambda_{p-1}}(j_{p-1},j_p)a^{\lambda_p}(j_p,j_1)
\\&\leq& 
\max_{\lambda_1\in\Lambda}
\max_{j_1'\in\bbJ}\sum_{j_2\in\bbJ}a^{\lambda_1}(j_1',j_2)
\\&&\times
\max_{\lambda_2,...,\lambda_p\in\Lambda}
\max_{j_2'\in\bbJ}\sum_{j_3,...,j_k\in\bbJ}a^{\lambda_2}(j_2',j_3)\cdots
a^{\lambda_{p-1}}(j_{p-1},j_p)a^{\lambda_p}(j_p,j_1), 
\eeas
therefore 
\bea\label{201909130619} 
\bbS^{(p-1)}_N(j_1)
&\leq& 
K\Delta\bbS^{(p-2)}(j_1)
\eea
for all $j_1\in\bbJ_N$, $N\in\bbN$ and $p\in\{3,4,...\}$. 
By inductively applying (\ref{201909130619}), we obtain 
\bea\label{201909130637} 
\bbS^{(k-1)}(j_1)
&\leq& 
(K\Delta)^{k-2}\bbS^{(1)}(j_1). 
\eea
Moreover, 
\bea\label{201909130638} 
\bbS^{(1)}(j_1)% p=2
\nn&=& 
\max_{\lambda_1,\lambda_2\in\Lambda}
\max_{j_1'\in\bbJ}
\sum_{j_2\in\bbJ_N}a^{\lambda_1}(j_1',j_2)a^{\lambda_2}(j_2,j_1)
\nn\\&\leq&
\max_{\lambda_1,\lambda_2\in\Lambda}
\max_{j_1'\in\bbJ}
\sum_{j_2\in\bbJ}a^{\lambda_1}(j_1',j_2)
\max_{j_2'\in\bbJ}a^{\lambda_2}(j_2',j_1)
\nn\\&\leq&
K\Delta
\max_{\lambda_2\in\Lambda}\max_{j_2'\in\bbJ}a^{\lambda_2}(j_2',j_1)
\nn\\&\leq&
K\Delta\Delta'. 
\eea
From (\ref{201909130637}) and (\ref{201909130638}), we conclude 
\beas 
\max_{j_1\in\bbJ}\bbS^{(k-1)}(j_1)
&\leq& 
(K\Delta)^{k-1}\Delta'.
\eeas
Since $S(j_1;i_1,...,i_k)\leq \bbS^{(k-1)}(j_1)$ by definition, 
we obtain the result. 
\qed\halflineskip
}

{\colorg \section{ Elements from Malliavin calculus}\label{app}}
{
In this section, we recall the basics of the Mallaivin calculus.  For complete presentations, we refer to \cite{N} or \cite{NPbook}. Consider $H$ a real separable Hilbert space and $(W(h), h \in H)$ an isonormal Gaussian process on a probability space $(\Omega, {\cal{A}}, P)$, which is a centered Gaussian family of random variables such that ${\bf E}\left[ W(\varphi) W(\psi) \right]  = \langle\varphi, \psi\rangle_{H}$. 

 We denote by $D$  the Malliavin  derivative operator that acts on smooth functions $\mathcal{S}$ of the form $F=g(W(h_1), \ldots , W(h_n))$ ($g$ is a smooth function with compact support and $h_i \in H$)
\begin{equation*}
DF=\sum_{i=1}^{n}\frac{\partial g}{\partial x_{i}}(W(h_1), \ldots , W(h_n)) h_{i}.
\end{equation*}
By iteration, we can also define  $D ^{k}F$, the $k$th iterated Malliavin derivative. Let $\mathbb{D} _ {k, p}$ (for any natural number $k$ and for any real number $p\geq 1$) be the closure of $\mathcal{S}$ with respect to the norm
\begin{equation*}
\Vert F\Vert _{k,p} ^{p}:= E [\vert F\vert ^{p} ] + \sum_{i=1}^{k} E \left[ \Vert D ^{i} F \Vert ^{p} _{H ^{\otimes i}}\right].
\end{equation*}
The adjoint of $D$ is denoted by $\delta $ and is called the divergence (or
Skorohod) integral. Its domain ($Dom(\delta)$) coincides with the class of stochastic processes $u\in L^{2}(\Omega \times T)$ such that
\begin{equation*}
\left| \mathbf{E}\big[\langle DF, u\rangle \big]\right| \leq c\Vert F\Vert _{2}
\end{equation*}
for all $F\in \mathbb{D}_ {1,2}$ and $\delta (u)$ is the element of $L^{2}(\Omega)$ characterized by the duality relationship
\begin{equation}\label{dua}
\mathbf{E}\big[(F\delta (u))\big]= \mathbf{E}\big[\langle DF, u\rangle _{H}\big].
\end{equation}

The chain rule for the Malliavin derivative (see Proposition 1.2.4 in \cite{N}) will be used several times. If $\varphi: \mathbb{R}\to \mathbb{R}$ is a continuously differentiable function 
{ having bounded derivative} 
and $F\in \mathbb{D} _ {1,2}$, then  $\varphi (F) \in \mathbb{D} _ {1,2}$ and
\begin{equation}
\label{chain}
D\varphi(F)= \varphi ' (F) DF.
\end{equation}

Denote by  $I_{n}$ the multiple stochastic integral with respect to
$B$ (see \cite{N}). This mapping $I_{n}$ is actually an isometry between the Hilbert space $H^{\odot n}$(symmetric tensor product) equipped with the scaled norm $\frac{1}{\sqrt{n!}}\Vert\cdot\Vert_{H^{\otimes n}}$ and the Wiener chaos of order $n$ which is defined as the closed linear span of the random variables $h_{n}(W(h))$ where $h \in H, \|h\|_{H}=1$ and $h_{n}$ is the Hermite polynomial of degree $n \in {\mathbb N}$
\begin{equation*}
h_{n}(x)=\frac{(-1)^{n}}{n!} \exp \left( \frac{x^{2}}{2} \right)
\frac{d^{n}}{dx^{n}}\left( \exp \left( -\frac{x^{2}}{2}\right)
\right), \hskip0.5cm x\in \mathbb{R}.
\end{equation*}
The isometry of multiple integrals can be written as follows: for $m,n$ positive integers,
\begin{eqnarray}
\mathbf{E}\left[I_{n}(f) I_{m}(g) \right] &=& n! \langle \tilde{f},\tilde{g}\rangle _{H^{\otimes n}}\quad \mbox{if } m=n,\nonumber \\
\mathbf{E}\left[I_{n}(f) I_{m}(g) \right] &= & 0\quad \mbox{if } m\not=n.\label{iso}
\end{eqnarray}
It also holds that
$I_{n}(f) = I_{n}\big( \tilde{f}\big)$ where $\tilde{f} $ denotes the symmetrization of $f$.

We recall that any square integrable random variable which is measurable with respect to the $\sigma$-algebra generated by $W$ can be expanded into an orthogonal sum of multiple stochastic integrals
\begin{equation}
\label{sum1} F=\sum_{n=0}^\infty I_{n}(f_{n})
\end{equation}
where $f_{n}\in H^{\odot n}$ are (uniquely determined)
symmetric functions and $I_{0}(f_{0})=\mathbf{E}\left[  F\right]$.
\\\\
Let $L$ be the Ornstein-Uhlenbeck operator
\begin{equation*}
LF=-\sum_{n\geq 0} nI_{n}(f_{n})
\end{equation*}
if $F$ is given by (\ref{sum1}) and it is such that $\sum_{n=1}^{\infty} n^{2}n! \Vert f_{n} \Vert ^{2} _{{\cal{H}}^{\otimes n}}<\infty$. 
Notice that
\begin{equation*}
LF= L(F-EF) \mbox{ and }L ^ {-1} F= L ^ {-1} (F-EF).
\end{equation*}
It holds that
\begin{equation}\label{aaa}
\delta D(-L) ^ {-1} F= F-EF.
\end{equation}

We  recall  the product formula for multiple integrals.
It is well-known that for $f\in H^{\odot n}$ and $g\in H^{\odot m}$
\begin{equation}\label{prod}
I_n(f)I_m(g)= \sum _{r=0}^{n\wedge m} r! \left( \begin{array}{c} n\\r\end{array}\right) \left( \begin{array}{c} m\\r\end{array}\right) I_{m+n-2r}(f\otimes _r g)
\end{equation}
where $f\otimes _r g$ means the $r$-contraction of $f$ and $g$.

Another important  property of  finite sums of multiple integrals is the hypercontractivity. Namely, if $F= \sum_{k=0} ^{n} I_{k}(f_{k}) $ with $f_{k}\in H ^{\otimes k}$ then
\begin{equation}
\label{hyper}
\mathbf{E}\big[\vert F \vert ^{p}\big] \leq C_{p} \left( \mathbf{E}\big[F ^{2} \big]\right) ^{\frac{p}{2}}.
\end{equation}
for every $p\geq 2$.

{\colorg
We can also define associated Sobolev spaces and Malliavin derivatives for vector-valued random variables. Let $V$ be an Hilbert space. 
Let $\mathcal{S}_{V}$ denote the set\
\[\mathcal{S}_{V}=\left\{\sum_{i=1}^n F_i h_i\Big| F_1,\ldots ,F_n\in\mathcal{S},\,h_1,\ldots , h_n\in V,\,n\geq 1 \right\}.\]
Then for $u=\sum_{i=1}^n F_i h_i$ we can define $Du := \sum_{i=1}^n DF_i \otimes h_i$ and consider the norm
\[\|u\|_{k,\,p,\,V}=\left(\E[\|u\|^2_{V}]+\sum_{i=1} ^{k}\E[\|D^{k} u\|^p_{H^{\otimes i}\otimes V}]\right)^{\frac{1}{p}}.\]
Now we can, just as for the space $\mathcal{S}$, consider the closure of $\mathcal{S}_{V}$ with respect to this norm and call it $\mathbb{D}_{k,\,p}(V)$.\\}

\end{document}